%% file: ABS1.tex
\documentclass[a4paper, 12pt, reqno]{amsart}

\input{ABS1_preamble}

\hypersetup{pdfauthor={Becky Armstrong, Nathan Brownlowe, and Aidan Sims}, pdftitle={Simplicity of twisted C*-algebras of Deaconu–Renault groupoids}}

\date{\today}
\title[Simplicity of twisted C*-algebras of Deaconu--Renault groupoids]{Simplicity of twisted C*-algebras of Deaconu--Renault groupoids}

\author[Armstrong]{Becky Armstrong}
\author[Brownlowe]{Nathan Brownlowe}
\author[Sims]{Aidan Sims}
\address[B.~Armstrong]{School of Mathematics and Statistics, Victoria University of Wellington, PO Box 600, Wellington 6140, NEW ZEALAND}
\email{\href{mailto:becky.armstrong@vuw.ac.nz}{becky.armstrong@vuw.ac.nz}}
\address[N.~Brownlowe]{School of Mathematics and Statistics, The University of Sydney, NSW 2006, AUSTRALIA}
\email{\href{mailto:nathan.brownlowe@sydney.edu.au}{nathan.brownlowe@sydney.edu.au}}
\address[A.~Sims]{School of Mathematics and Applied Statistics, University of Wollongong, NSW 2522, AUSTRALIA}
\email{\href{mailto:asims@uow.edu.au}{asims@uow.edu.au}}

\thanks{The first author was supported by an Australian Government Research Training Program Stipend Scholarship. The second and third authors were supported by the Australian Research Council grant DP200100155. The first author would like to thank the second and third authors for being awesome PhD super(hero)visors.}

\begin{document}

\begin{abstract}
We consider Deaconu--Renault groupoids associated to actions of finite-rank free abelian monoids by local homeomorphisms of locally compact Hausdorff spaces. We study simplicity of the twisted C*-algebra of such a groupoid determined by a continuous circle-valued groupoid $2$-cocycle. When the groupoid is not minimal, this C*-algebra is never simple, so we focus on minimal groupoids. We describe an action of the quotient of the groupoid by the interior of its isotropy on the spectrum of the twisted C*-algebra of the interior of the isotropy. We prove that the twisted groupoid C*-algebra is simple if and only if this action is minimal. We describe applications to crossed products of topological-graph C*-algebras by quasi-free actions.
\end{abstract}

\maketitle

\section{Introduction}
\label{sec: intro}

The purpose of this paper is to characterise simplicity of twisted C*-algebras arising from continuous $2$-cocycles on Deaconu--Renault groupoids of actions of $\N^k$ on second-countable locally compact Hausdorff spaces. The study of twisted C*-algebras associated to continuous groupoid $2$-cocycles dates back to Renault's seminal work \cite{Renault1980}. They serve both as a very flexible C*-algebraic framework for modelling dynamical systems, and as a source of tractable models for classifiable C*-algebras \cite{BCSS2016, aHKS2011, Li2020, RT1985}. So it is important to be able to determine when a given twisted groupoid C*-algebra is simple; but this is in general a complicated question.

Deaconu--Renault groupoids encode actions of submonoids of abelian groups by local homeomorphisms of locally compact Hausdorff spaces. In hindsight, the first example of such a groupoid was the one associated to the one-sided full shift on $n$ letters, introduced by Renault in \cite{Renault1980} as a model for the Cuntz algebra. However groupoids of this type for generic local homeomorphisms (that is, actions of $\N$) were first studied by Deaconu \cite{Deaconu1995}, and have come to be known as (rank-$1$) Deaconu--Renault groupoids. Shortly afterwards they were used as models for graph C*-algebras in \cite{KPR1998, KPRR1997}, and later still, Yeend \cite{Yeend2007JOT} showed that rank-$1$ Deaconu--Renault groupoids provide models for the topological-graph C*-algebras of Katsura \cite{Katsura2004TAMS}.

For the dual reasons that most of the key examples studied had been related to $0$\nobreakdash-dimensional spaces, and that $\N$ embeds in $\Z$, which has trivial cohomology, no work was done on twisted C*-algebras associated to Deaconu--Renault groupoids for many years. However, in 2000, Kumjian and Pask \cite{KP2000} introduced higher-rank graphs (or $k$\nobreakdash-graphs) and demonstrated that the associated C*-algebras can be described as the C*-algebras of Deaconu--Renault groupoids of actions of $\N^k$. This led to the development \cite{KPS2012JFA, KPS2013JMAA, KPS2015TAMS} of twisted $k$-graph C*-algebras. Kumjian, Pask, and Sims showed that from a $2$-cocycle on a $k$-graph, one can construct a $2$-cocycle on the associated Deaconu--Renault groupoid so that the twisted C*-algebras coincide, and they used this model to characterise simplicity of twisted $k$-graph C*-algebras \cite{KPS2016JNG}, as well as to describe applications of this characterisation to the study of crossed products of graph algebras by quasi-free actions.

Here we build substantially on elements of the analysis of \cite{KPS2016JNG} to describe precisely when the twisted C*-algebra of a Deaconu--Renault groupoid for an action of $\N^k$ by local homeomorphisms is simple (\cref{thm: simplicity characterisation}). To demonstrate the applicability of our main theorem, we use this result to investigate simplicity of crossed products of C*\nobreakdash-algebras associated to rank-$1$ Deaconu--Renault groupoids by actions of $\Z$ induced by $\T$-valued $1$-cocycles (\cref{thm: CP simple}), and we specialise to the Deaconu--Renault groupoids of topological graphs to characterise simplicity of crossed products of topological-graph C*-algebras by quasi-free automorphisms (\cref{cor: topgraph exact,cor: topgraph checkable}).

The paper is organised as follows. In \cref{sec: background} we establish background and notation. In \cref{sec: isotropy} we describe the periodicity group $\PT$ of a minimal action $T$ of $\N^k$ on a second-countable locally compact Hausdorff space $X$, and we show that the interior of the isotropy of $\GT$ is isomorphic to the group bundle $X \times \PT$. In \cref{sec: cohomology} we show that every $2$-cocycle on $\GT$ is cohomologous to one whose restriction to $X \times \PT$ is determined by a fixed bicharacter $\omega$ of $\PT$ that vanishes on its own centre $Z_\omega$, and we use this to give a concrete description of the spectral action $\theta$ of $\GT/\IT$ on $X \times \widehat{Z}_\omega$. Then in \cref{sec: simplicity} we state and prove our main theorem. We finish in \cref{sec: application} by describing an application to crossed products of rank-$1$ Deaconu--Renault groupoid C*-algebras by automorphisms induced by continuous $1$-cocycles. We provide two appendices---one on group cohomology and one on twisted group C*-algebras---to provide a handy reference to some key results on these two topics that we need in the body of the paper, and have found difficult to locate explicitly in the literature.

\section{Background}
\label{sec: background}

\subsection{Group \texorpdfstring{$2$}{2}-cocycles and bicharacters}
\label{sec: group cohomology}

Here we briefly recall some key facts about second cohomology for discrete groups. For more detail see \cite[Chapter~IV]{Brown1982} and \cite{BK1973, Kleppner1965} (or \cite[Chapter~2]{Armstrong2019} for the key points relevant here collected in one place).

Let $G$ be a discrete group and let $A$ be a multiplicative abelian group. We write $Z^2(G,A)$ for the group of normalised $A$-valued $2$-cocycles on $G$, $B^2(G, A)$ for the subgroup of coboundaries, $\delta^1$ for the coboundary map, and $H^2(G, A)$ for the second cohomology group $Z^2(G, A) / B^2(G, A)$. Given $\sigma \in Z^2(G, A)$, we write $\sigma^*$ for the $2$-cocycle $(g,h) \mapsto \sigma(h,g)^{-1}$. We call $\sigma$ \hl{antisymmetric} if $\sigma = \sigma^*$.

A \hl{bicharacter} of $G$ is a map $\omega\colon G \times G \to \T$ such that $\omega(\cdot, g)$ and $\omega(g, \cdot)$ are homomorphisms from $G \to \T$ for each $g \in G$. Every bicharacter is a $\T$-valued $2$-cocycle. If $G$ is a discrete abelian group, then $\sigma \mapsto \sigma \sigma^*$ is a homomorphism from $Z^2(G, \T)$ to the group of antisymmetric bicharacters of $G$, which descends to an isomorphism of $H^2(G, \T)$ onto the same group \cite[Proposition~3.2]{OPT1980}.

In this paper, the \hl{centre} $Z_\sigma$ of a $2$-cocycle $\sigma$ on $G$ is the joint kernel of the associated antisymmetric bicharacter: $Z_\sigma \coloneqq \{ g \in G : (\sigma\sigma^*)(g,h) = 1\text{ for all } h \in G \}$. If $\sigma(g,h) = 1 = \sigma(h,g)$ for all $g \in Z_\sigma$ and $h \in G$, then we say that $\sigma$ \hl{vanishes on its centre}. An adaptation of the argument of \cite[Proposition~3.2]{OPT1980} (see \cite[Theorem~2.2.8]{Armstrong2019} for details) shows that every $2$-cocycle on a finitely generated discrete abelian group is cohomologous to a bicharacter that vanishes on its centre.

\subsection{Hausdorff \texorpdfstring{\'etale}{etale} groupoids}

We refer to a topological groupoid $\GG$ with a locally compact Hausdorff topology under which multiplication and inversion are continuous as a \hl{Hausdorff groupoid}. We write $\GGo$ for the unit space of $\GG$, and $\GGc$ for the set of composable pairs in $\GG$. Given subsets $A, B \subseteq \GG$, we write $AB \coloneqq \{ \alpha\beta : (\alpha, \beta) \in (A \times B) \cap \GGc \}$ and $A^{-1} \coloneqq \{ \alpha^{-1} : \alpha \in A \}$, and for $\gamma \in \GG$, we write $\gamma A \coloneqq \{\gamma\} A$ and $A \gamma \coloneqq A \{\gamma\}$. We say that $\GG$ is \hl{\'etale} if the range and source maps $r,s\colon \GG \to \GGo$ are local homeomorphisms. We call a subset $B$ of $\GG$ a \hl{bisection} if $B$ is contained in an open subset $U$ of $\GG$ such that $r\restr{U}$ and $s\restr{U}$ are homeomorphisms onto open subsets of $\GGo$. Every second-countable Hausdorff \'etale groupoid has a countable basis of open bisections. For each $x \in \GGo$, we define $\GG_x \coloneqq s^{-1}(x)$ and $\GG^x \coloneqq r^{-1}(x)$; and $\GG_x^x \coloneqq \GG_x \cap \GG^x$. We say that $\GG$ is \hl{minimal} if $r(\GG_x)$ is dense in $\GGo$ for every $x \in \GGo$. The \hl{isotropy subgroupoid} of $\GG$ is the groupoid $\Iso(\GG) \coloneqq \bigcup_{x \in \GGo} \, \GG_x^x = \{ \gamma \in \GG : r(\gamma) = s(\gamma) \}$. The interior $\II$ of the isotropy of a Hausdorff \'etale groupoid $\GG$ is itself a Hausdorff \'etale groupoid with unit space $\II^{(0)} = \GGo$. We say that $\GG$ is \hl{effective} if $\II = \GGo$, and we say that $\GG$ is \hl{topologically principal} if $\{ x \in \GGo : \GG_x^x = \{x\} \}$ is dense in $\GGo$. By \cite[Lemma~3.1]{BCFS2014}, every topologically principal Hausdorff \'etale groupoid is effective, and every effective second-countable Hausdorff \'etale groupoid is topologically principal.

The following definition of a groupoid action comes from \cite[Definition~1.60]{Goehle2009}.

\begin{definition} \label{def: groupoid action}
Suppose that $\GG$ is a topological groupoid and $X$ is a topological space. We say that \hl{$\GG$ acts continuously on (the left of) $X$}, and that $X$ is a \hl{continuous (left) $\GG$-space}, if there is a continuous surjective map $R\colon X \to \GGo$ and a continuous map $\theta\colon (\gamma,x) \mapsto \gamma \cdot x$ from $\GG \star X \coloneqq \{ (\gamma,x) \in \GG \times X : s(\gamma) = R(x) \}$ to $X$, satisfying
\begin{enumerate}[label=(A\arabic*)]
\item \label{item: composing groupoid actions} if $(\alpha,\beta) \in \GGc$ and $(\beta,x) \in \GG \star X$, then $(\alpha\beta,x),\, (\alpha,\, \beta \cdot x) \in \GG \star X$, and we have $\alpha \cdot (\beta \cdot x) = (\alpha\beta) \cdot x$; and
\item \label{item: action of groupoid unit} for all $x \in X$, we have $(R(x),x) \in \GG \star X$, and $R(x) \cdot x = x$.
\end{enumerate}
We refer to the map $\theta$ as a \hl{continuous (left) action} of $\GG$ on $X$. For each $x \in X$, the \hl{orbit} of $x$ under $\theta$ is the set
\[
[x]_{\theta} \coloneqq \{ \gamma \cdot x : (\gamma, x) \in \GG \star X \}.
\]
\end{definition}

\subsection{Cohomology of groupoids}

We now recall the relevant cohomology theory for groupoids from \cite[Section~I.1]{Renault1980}.

\begin{definition} \label{def: groupoid cohomology}
Let $\GG$ be a topological groupoid, and let $A$ be a topological abelian group with identity $e_A$.
\begin{enumerate}[label=(\roman*), ref={\cref{def: groupoid cohomology}(\roman*)}]
\item A \hl{continuous $A$-valued $1$-cochain} on $\GG$ is a continuous map $b\colon \GG \to A$. We say that $b$ is \hl{normalised} if $b(r(\gamma)) = b(s(\gamma)) = e_A$ for all $\gamma \in \GG$.
\item A \hl{continuous $A$-valued $1$-cocycle} on $\GG$ is a continuous $1$-cochain $c\colon \GG \to A$ satisfying $c(\alpha\beta) = c(\alpha) c(\beta)$ for all $(\alpha,\beta) \in \GGc$.
\item A \hl{continuous $A$-valued $2$-cocycle} on $\GG$ is a continuous map $\sigma\colon \GGc \to A$ that satisfies the \hl{$2$-cocycle identity}: $\sigma(\alpha,\beta) \, \sigma(\alpha\beta,\gamma) = \sigma(\alpha,\beta\gamma) \, \sigma(\beta,\gamma)$ for all $\alpha, \beta, \gamma \in \GG$ such that $s(\alpha) = r(\beta)$ and $s(\beta) = r(\gamma)$, and is \hl{normalised}, in the sense that $\sigma(r(\gamma),\gamma) = \sigma(\gamma,s(\gamma)) = e_A$ for all $\gamma \in \GG$. We write $Z^2(\GG,A)$ for the group of continuous $A$-valued $2$-cocycles on $\GG$.
\item \label{item: 2-coboundary} The \hl{continuous $2$-coboundary} associated to a continuous normalised $A$-valued $1$\nobreakdash-cochain $b\colon \GG \to A$ is the map $\delta^1 b \colon \GGc \to A$ given by
\[
\delta^1 b(\alpha,\beta) \coloneqq b(\alpha) \, b(\beta) \, b(\alpha\beta)^{-1}.
\]
\item We say that two continuous $2$-cocycles $\sigma,\tau\colon \GGc \to A$ are \hl{cohomologous} if there exists a continuous normalised $1$-cochain such that $\delta^1 b(\alpha,\beta) = \sigma(\alpha,\beta)^{-1} \, \tau(\alpha,\beta)$ for all $(\alpha,\beta) \in \GGc$.
\end{enumerate}
\end{definition}

We write $\T$ for the multiplicative group of complex numbers of modulus $1$. Suppose that $\GG$ is a Hausdorff groupoid and take $\sigma \in Z^2(\GG,\T)$. Let $\GG \times_\sigma \T$ be the set $\GG \times \T$ endowed with the product topology, and equipped with the multiplication operation
\begin{equation} \label{eqn: twisted groupoid multiplication}
(\alpha,w)(\beta,z) \coloneqq \big(\alpha\beta, \, \sigma(\alpha,\beta) wz\big),
\end{equation}
defined for all $(\alpha,\beta) \in \GGc$ and $w, z \in \T$, and the inversion operation \begin{equation} \label{eqn: twisted groupoid inversion}
(\alpha,w)^{-1} \coloneqq \big(\alpha^{-1}, \, \overline{\sigma(\alpha,\alpha^{-1})} \overline{w}\big),
\end{equation}
defined for all $(\alpha,w) \in \GG \times \T$. Then $\GG \times_\sigma \T$ is a Hausdorff groupoid.

\subsection{Twisted groupoid C*-algebras}

We now recall Renault's construction of the full twisted groupoid C*-algebra $C^*(\GG,\sigma)$ associated to a Hausdorff \'etale groupoid $\GG$ and a continuous $\T$-valued $2$-cocycle $\sigma$ on $\GG$. Note that Renault gives this construction for groupoids that are not necessarily \'etale, but we specialise to the \'etale case since we will primarily be dealing with Deaconu--Renault groupoids, which are \'etale. Renault also defines reduced twisted groupoid C*-algebras, but we will only be working with amenable groupoids, and in this setting, the full and reduced C*-algebras coincide. Let $C_c(\GG,\sigma)$ denote the complex vector space of continuous compactly supported complex-valued functions on $\GG$, equipped with multiplication given by the \hl{twisted convolution} formula
\[
(f * g)(\gamma) \coloneqq \sum_{\substack{(\alpha,\beta) \in \GGc, \\ \alpha\beta = \gamma}} \sigma(\alpha,\beta) \, f(\alpha) \, g(\beta) = \sum_{\zeta \in \GG^{r(\gamma)}} \sigma(\zeta, \zeta^{-1}\gamma) \, f(\zeta) \, g(\zeta^{-1}\gamma),
\]
and involution given by
\[
f^*(\gamma) \coloneqq \overline{\sigma(\gamma,\gamma^{-1})} \, \overline{f(\gamma^{-1})}.
\]
Then $C_c(\GG,\sigma)$ is a $*$-algebra. We write $fg$ for the twisted convolution product $f * g$ when the intended meaning is clear. The \hl{full twisted groupoid C*-algebra} $C^*(\GG,\sigma)$ is the completion of $C_c(\GG,\sigma)$ with respect to the \hl{full C*-norm}, which is given by
\[
\lv f \rv \coloneqq \sup\big\{ \lv \pi(f) \rv : \pi \text{ is a $*$-representation of } C_c(\GG,\sigma) \big\}.
\]

Given a locally compact Hausdorff space $Y$ and a function $f \in C_c(Y)$, we define the \hl{open support} of $f$ to be the set $\osupp(f) \coloneqq f^{-1}(\C {\setminus} \{0\})$, and the \hl{support} of $f$ to be the set $\supp(f) \coloneqq \overline{\osupp(f)}$.

\subsection{Deaconu--Renault groupoids}

We recall the definition of the Deaconu--Renault groupoid associated to an action of $\N^k$ by
local homeomorphisms. Details appear in \cite[Proposition~3.1]{SW2016}.

Fix $k \in \N {\setminus} \{0\}$. Let $T\colon n \mapsto T^n$ be an action of $\N^k$ on a locally compact Hausdorff space $X$ by local homeomorphisms. We call the pair $(X,T)$ a \hl{rank-$k$ Deaconu--Renault system}. Define
\[
\GT \coloneqq \{ (x,m-n,y) \in X \times \Z^k \times X \,:\, m, n \in \N^k, \, T^m(x) = T^n(y) \},
\]
and
\[
\GTc \coloneqq \{ ((x,m,y), (w,n,z)) \in \GT \times \GT \,:\, y = w \}.
\]
If $((x,m,y), (y,n,z)) \in \GTc$, then $(x,m+n,z), (y,-m,x) \in \GT$. We define multiplication from $\GTc$ to $\GT$ by $(x,m,y)(y,n,z) \coloneqq (x,m+n,z)$, and inversion on $\GT$ by $(x,m,y)^{-1} \coloneqq (y,-m,x)$. Then $\GT$ is a groupoid, called a \hl{Deaconu--Renault groupoid}. The unit space of $\GT$ is $\GTo = \{ (x,0,x) : x \in X \}$, and we identify it with $X$. The range and source maps of $\GT$ are given by $r(x,m,y) \coloneqq x$ and $s(x,m,y) \coloneqq y$. For open sets $U, V \subseteq X$ and for $m, n \in \N^k$, we define
\[
Z(U,m,n,V) \coloneqq \{ (x,m-n,y) \,:\, x \in U, \, y \in V,\, \text{and } T^m(x) = T^n(y) \}.
\]
The collection $\{ Z(U,m,n,V) : U, V \subseteq X \text{ are open, and } m, n \in \N^k \}$ is a basis for a locally compact Hausdorff topology on $\GT$. The sets $Z(U,m,n,V)$ such that $T^m\restr{U}$ and $T^n\restr{V}$ are homeomorphisms onto their ranges and $T^m(U) = T^n(V)$ form a basis for the same topology. Under this topology, $\GT$ is a locally compact Hausdorff \'etale groupoid. If $X$ is second-countable, then $\GT$ is also second-countable.

\begin{remark}
The action of $\N^k$ in the above definition of a Deaconu--Renault system can be replaced with an action of a more general monoid $P$ contained in a group $G$, and this gives rise to a $G$-graded Deaconu--Renault groupoid. Such groupoids are studied in \cite{ER2007}, but we do not investigate them here.
\end{remark}

\begin{lemma} \label{lemma: 1-cocycle c}
Let $(X,T)$ be a rank-$k$ Deaconu--Renault system. The map $c\colon (x,n,y) \mapsto n$ is a continuous $\Z^k$-valued $1$-cocycle on $\GT$, and for each $x \in X$, the restriction of $c$ to $(\GT)_x^x$ is injective.
\end{lemma}

\begin{proof}
Fix $\alpha = (x, p, y)$ and $\beta = (y,q,z) \in \GT$. Then $c(\alpha\beta) = p+q = c(\alpha) + c(\beta)$, and so $c$ is a $1$-cocycle. Since each $c\restr{Z(U, m, n, V)}$ is constant, $c$ is locally constant and hence continuous.
\end{proof}

\begin{definition}
Let $(X,T)$ be a rank-$k$ Deaconu--Renault system. The \hl{orbit} under $T$ of $x \in X$ is
\[
\Orb[T]{x} \coloneqq \bigcup_{m,n \in \N^k} (T^m)^{-1}\big(T^n(x)\big) = \{ y \in X : T^m(y) = T^n(x) \text{ for some } m, n \in \N^k \}.
\]
We say that $(X,T)$ is \hl{minimal} if $\Orb[T]{x}$ is dense in $X$ for each $x \in X$. We frequently just write $\Orb{x}$ for $\Orb[T]{x}$.
\end{definition}

\begin{remark} \label{rem: minimality}
We have $\Orb{x} = r(s^{-1}(x)) \subseteq \GTo$, and so $\GT$ is minimal if and only if $(X,T)$ is minimal.
\end{remark}

\begin{remark}
By \cite[Lemma~3.5]{SW2016}, every Deaconu--Renault groupoid is amenable, and so we can discuss the twisted C*-algebra associated to a Deaconu--Renault groupoid and a continuous $2$-cocycle without any ambiguity as to whether we mean the full or reduced C*-algebra.
\end{remark}

\begin{remark}
The C*-algebras studied here are related to previous work. Suppose that $\Lambda$ is a proper, source-free topological $k$-graph with infinite-path space $\Lambda^\infty$ (as defined in \cite[Section~3]{AB2018}). For each $n \in \N^k$, let $T^n\colon \Lambda^\infty \to \Lambda^\infty$ be the shift map. Then $(\Lambda^\infty,T)$ is a rank-$k$ Deaconu--Renault system, and the associated Deaconu--Renault groupoid $\GG_\Lambda \coloneqq \GT$ is called the \hl{boundary-path groupoid} of the topological $k$-graph. The twisted C*\nobreakdash-algebras $C^*(\GG_\Lambda,\sigma)$ associated to continuous $2$-cocycles $\sigma \in Z^2(\GG_\Lambda, \T)$ on proper, source-free topological $k$-graphs generalise the twisted C*-algebras of discrete $k$-graphs studied in \cite{KPS2012JFA, KPS2013JMAA, KPS2015TAMS, KPS2016JNG}, and are studied in the first-named author's PhD thesis \cite{Armstrong2019}. In \cite{AB2018}, the first- and second-named authors study an alternative notion of a twisted C*\nobreakdash-algebra of a topological $k$-graph associated to a continuous $2$-cocycle on the topological $k$-graph itself, which is constructed using a product system of Hilbert bimodules. In the case where $\Lambda$ is a discrete $k$-graph, it is known (see \cite[Theorem~7.2.2]{Armstrong2014}) that these two constructions give the same C*-algebra, but in the more general topological setting, the relationship is unknown.
\end{remark}

\section{The interior of the isotropy of a Deaconu--Renault groupoid}
\label{sec: isotropy}

In this section we introduce the periodicity group $\PT$ of a minimal Deaconu--Renault groupoid $(X,T)$, and we show that the interior $\IT$ of the isotropy of $\GT$ can be identified with $X \times \PT$.

\begin{definition} \label{def: PT}
Let $(X,T)$ be a rank-$k$ Deaconu--Renault system. For each nonempty precompact open set $U \subseteq X$, we define
\[
\PT(U) \coloneqq \{ m - n \,:\, m, n \in \N^k, \text{ and } T^m\restr{U} = T^n\restr{U} \text{ is injective} \}.
\]
We define
\[
\PT \coloneqq \bigcup_{\varnothing \ne U \subseteq X \text{ precompact open}} \PT(U).
\]
\end{definition}

\begin{remark}
When $k = 1$, the set $\PT(U)$ is related to the group $\Stabess{x}$ from \cite[Page~29]{CRST2021}. Specifically, $\Stabess{x}$ contains $\PT(U)$ for any precompact open set $U$ containing $x$; but also, since $\Stabess{x}$ is a subgroup of $\Z^k$, and hence finitely generated, it is not too hard to check that there is an open cover of $X$ by sets $U$ such that $\PT(U) = \Stabess{x}$ for each $x$ in $U$.
\end{remark}

In addition to being needed for our own arguments, our next result, \cref{prop: IT characterisation}, plugs a gap in the literature---it is mentioned without proof in \cite[Page~30]{CRST2021}.

\begin{prop} \label{prop: PT characterisation}
Let $(X,T)$ be a minimal rank-$k$ Deaconu--Renault system. Then
\[
\PT = \{ p \in \Z^k : (x,p,x) \in \GT \text{ for all } x \in X \},
\]
and $\PT$ is a subgroup of $\Z^k$.
\end{prop}

In order to prove \cref{prop: PT characterisation}, we need the following lemma.

\begin{lemma} \label{lemma: open nbhd when T^m(y) ne T^n(y)}
Let $(X,T)$ be a minimal rank-$k$ Deaconu--Renault system. Suppose that $m, n \in \N^k$ and $y \in X$ satisfy $T^m(y) \ne T^n(y)$. Then there exists an open neighbourhood $W \subseteq X$ of $y$ such that $T^m\restr{W}$ and $T^n\restr{W}$ are injective and $T^m(W) \medcap T^n(W) = \varnothing$.
\end{lemma}

\begin{proof}
Since $X$ is Hausdorff, we can choose open neighbourhoods $U \subseteq X$ of $T^m(y)$ and $V \subseteq X$ of $T^n(y)$ such that $U \cap V = \varnothing$. Define $A \coloneqq (T^m)^{-1}(U) \medcap (T^n)^{-1}(V)$. Then $y \in A$. Since $T^m$ and $T^n$ are local homeomorphisms, there is an open neighbourhood $W \subseteq A$ of $y$ such that $T^m\restr{W}$ and $T^n\restr{W}$ are injective, and we have $T^m(W) \medcap T^n(W) \subseteq U \cap V = \varnothing$.
\end{proof}

\begin{proof}[Proof of \cref{prop: PT characterisation}]
Fix $p \in \PT$. Then there exist $m, n \in \N^k$ and a nonempty open set $U \subseteq X$ such that $p = m - n$, and $T^m\restr{U} = T^n\restr{U}$ is injective. Fix $z \in X$. We claim that $(z,p,z) \in \GT$. Since $\Orb{z}$ is dense in $X$, we have $U \cap \Orb{z} \ne \varnothing$, and so there exist $y \in U$ and $a,b \in \N^k$ such that $T^a(y) = T^b(z)$. Thus,
\[
T^{b+m}(z) = T^m(T^b(z)) = T^m(T^a(y)) = T^a(T^m(y)),
\]
and
\[
T^{b+n}(z) = T^n(T^b(z)) = T^n(T^a(y)) = T^a(T^n(y)).
\]
Since $y \in U$, we have $T^m(y) = T^n(y)$, and hence $T^{b+m}(z) = T^{b+n}(z)$. Therefore, $(z,p,z) = \big(z, (b+m) - (b+n), z\big) \in \GT$, and so
\[
\PT \subseteq \{ p \in \Z^k : (x,p,x) \in \GT \text{ for all } x \in X \}.
\]

We now show that $\Z^k \!\setminus\! \PT \subseteq \{ p \in \Z^k : (x, p, x) \notin \GT \text{ for some } x \in X \}$. To see this, fix $p \in \Z^k \!\setminus\! \PT$. Let $(m_i,n_i)_{i=1}^\infty$ be an enumeration of $\{(m,n) \in \N^k \times \N^k \,:\, m - n = p\}$. We must find $x \in X$ such that $T^{m_i}(x) \ne T^{n_i}(x)$ for all $i \ge 1$. We claim that there exist nonempty precompact open subsets $V_0, V_1, V_2, \dotsc$ of $X$ satisfying
\begin{enumerate}[label=(\arabic*)]
\item \label{item: Vthing1} $\overline{V_i} \subseteq V_{i-1}$ for all $i \ge 1$,
\item \label{item: Vthing2} $T^{m_i}\restr{\overline{V_{i-1}}}$ and $T^{n_i}\restr{\overline{V_{i-1}}}$ are injective for all $i \ge 1$, and
\item \label{item: Vthing3} $T^{m_i}\big(\overline{V_i}\big) \medcap T^{n_i}\big(\overline{V_i}\big) = \varnothing$ for all $i \ge 1$.
\end{enumerate}
To start, let $V_0$ be a nonempty precompact open subset of $X$ such that $T^{m_1}\restr{\overline{V_0}}$ and $T^{n_1}\restr{\overline{V_0}}$ are injective. Now fix $i \ge 1$ and suppose that $V_0, \dotsc, V_{i-1}$ satisfy \cref{item: Vthing1}--\cref{item: Vthing3}. Since $m_i - n_i = p \notin \PT$, we have $T^{m_i}\restr{V_{i-1}} \ne T^{n_i}\restr{V_{i-1}}$, and so there exists $y \in V_{i-1}$ such that $T^{m_i}(y) \ne T^{n_i}(y)$. Thus, by \cref{lemma: open nbhd when T^m(y) ne T^n(y)}, there exists an open neighbourhood $W \subseteq V_{i-1}$ of $y$ such that $T^{m_i}(W) \medcap T^{n_i}(W) = \varnothing$. Since $X$ is locally compact and Hausdorff and $T^{m_{i+1}}$ and $T^{n_{i+1}}$ are local homeomorphisms, there is an open neighbourhood $V_i$ of $y$ such that $\overline{V_i} \subseteq W$ and $T^{m_{i+1}}\restr{\overline{V_i}}$ and $T^{n_{i+1}}\restr{\overline{V_i}}$ are injective. So induction gives the desired sets $V_i$. Each $V_i$ is contained in the compact set $\overline{V_0}$, and so the descending intersection $\bigcap_{i=1}^\infty \overline{V_i}$ is nonempty. Any $x \in \bigcap_{i=1}^\infty \overline{V_i}$ satisfies $T^{m_i}(x) \ne T^{n_i}(x)$ for all $i \ge 1$.

We conclude by showing that $\PT$ is a subgroup of $\Z^k$. For all $x \in X$, we have $(x,0,x) \in \GTo \subseteq \GT$, and so $0 \in \PT$. Suppose that $p, q \in \PT$. For all $x \in X$, we have $(x,p,x), (x,q,x) \in \GT$, and hence $(x,p-q,x) = (x,p,x)(x,q,x)^{-1} \in \GT$. Thus $p - q \in \PT$, and so $\PT$ is a subgroup of $\Z^k$.
\end{proof}

Given a rank-$k$ Deaconu--Renault system $(X,T)$, we write $\IT$ for the topological interior of $\Iso(\GT)$. Since $\GT$ is a locally compact Hausdorff \'etale groupoid, so is $\IT$. From this point forward, we will assume that $X$ is second-countable (and hence so are $\GT$ and $\IT$). We know from \cite[Lemma~3.5]{SW2016} that $\GT$ is amenable, and hence \cite[Proposition~5.1.1]{ADR2000} implies that $\IT$ is amenable.

\begin{prop} \label{prop: IT characterisation}
Let $(X,T)$ be a minimal rank-$k$ Deaconu--Renault system such that $X$ is second-countable. Let $\PT$ be as in \cref{def: PT}. Then
\[
\IT = \{ (x,p,x) : x \in X, \, p \in \PT \} \cong X \times \PT.
\]
\end{prop}

\begin{proof}
For $\subseteq$, fix $\gamma \in \IT$. Let $c\colon \GT \to \Z^k$ be the continuous $1$-cocycle defined in \cref{lemma: 1-cocycle c}. Let $p \coloneqq c(\gamma)$ so that $\gamma = (x,p,x)$ for some $x \in X$. We claim that $p \in \PT$. By \cref{rem: minimality}, $\GT$ is minimal, and hence \cite[Proposition~2.1]{KPS2016JNG} implies that for all $y \in X$,
\[
c\big(\IT \medcap (\GT)_y^y\big) = c\big(\IT \medcap (\GT)_x^x\big),
\]
and thus
\[
p = c(x,p,x) \in c\big(\IT \medcap (\GT)_x^x\big) = c\big(\IT \medcap (\GT)_y^y\big).
\]
So \cref{prop: PT characterisation} gives $p \in \PT$.

For $\supseteq$, fix $x \in X$ and $p \in \PT$. By the definition of $\PT$, there exist $m, n \in \N^k$ and a nonempty precompact open set $U \subseteq X$ such that $p = m-n$ and $T^m\restr{U} = T^n\restr{U}$ is injective. This injectivity forces $Z(U,m,n,U) = \{ (y,p,y) : y \in U \} \subseteq \IT$. Fix $y \in U$. Then $(y,p,y) \in \IT$, and so \cite[Proposition~2.1]{KPS2016JNG} implies that
\[
p = c(y,p,y) \in c\big(\IT \medcap (\GT)_y^y\big) = c\big(\IT \medcap (\GT)_x^x\big),
\]
and hence $(x,p,x) \in \IT$.
\end{proof}

\begin{remark}
\Cref{prop: IT characterisation} is related to the sets $\Sigma_X$ and $H(T)$ of \cite{SW2016} as follows. Let $(X,T)$ be a minimal rank-$k$ Deaconu--Renault system such that $X$ is second-countable. In the notation of \cite[Section~3]{SW2016}, suppose that $\Sigma = \Sigma_X$. Then $T$ is an irreducible action of $\N^k$ on $X$, and \cite[Proposition~3.10]{SW2016} implies that
\[
\IT = \{ (x,p,x) : x \in X, \, p \in H(T) \} \cong X \times H(T).
\]
Thus \cref{prop: IT characterisation} implies that $\PT = H(T)$.
\end{remark}

We now present two corollaries of \cref{prop: IT characterisation}.

\begin{cor} \label{cor: GT effective iff PT trivial}
Let $(X,T)$ be a minimal rank-$k$ Deaconu--Renault system such that $X$ is second-countable. Then $\GT$ is effective if and only if $\PT = \{0\}$.
\end{cor}

\begin{proof}
By \cref{prop: IT characterisation}, $\IT = \{ (x,p,x) : x \in X, \, p \in \PT \}$. Hence
\[
\GT \text{ is effective} \iff \IT = \GTo \iff \PT = \{0\}. \qedhere
\]
\end{proof}

\begin{cor}
Let $(X,T)$ be a minimal rank-$k$ Deaconu--Renault system such that $X$ is second-countable. Let $c\colon \GT \to \Z^k$ be as in \cref{lemma: 1-cocycle c}. Then for each $p \in \PT$, we have $(c\restr{\IT})^{-1}(p) = \{(x, p, x) : x \in X\}$, and $\left\{ c\restr{\IT}^{-1}(p) : p \in \PT \right\}$ is a collection of mutually disjoint clopen bisections whose union is $\IT$.
\end{cor}

\begin{proof}
Fix $p \in \PT$. Since $c$ is continuous and $\Z^k$ is discrete, $c\restr{\IT}^{-1}(p)$ is clopen. Fix $x \in X$. If $\alpha, \beta \in c\restr{\IT}^{-1}(p)$ and $r(\alpha) = x = r(\beta)$, then $s(\alpha) = x = s(\beta)$ because $\alpha, \beta \in \Iso(\GT)$, and hence $\alpha = (x,p,x) = \beta$. So $r\restr{c\restr{\IT}^{-1}(p)}$ is injective, and a similar argument shows that $s\restr{c\restr{\IT}^{-1}(p)}$ is also injective. Hence $c\restr{\IT}^{-1}(p)$ is a bisection. By \cref{prop: IT characterisation}, we have $c(\IT) = \PT$, and the result follows.
\end{proof}

We now prove that when $\GT$ is minimal, we can form the quotient groupoid $\GT/\IT$. As the anonymous referee correctly points out, a more general result is possible---the salient point is that $\IT$ is a closed normal subgroupoid of the isotropy---but our application is to simplicity of twisted C*-algebras associated to $\GT$, for which minimality of $\GT$ is a necessary condition (see \cref{thm: simplicity characterisation}\cref{item: main simple implies minimal}).

\begin{prop} \label{prop: HT quotient groupoid}
Let $(X,T)$ be a minimal rank-$k$ Deaconu--Renault system such that $X$ is second-countable. Then
$\IT$ is a closed subgroupoid of $\GT$ and acts freely and properly on $\GT$ by
right-multiplication. The set $\HT \coloneqq \GT / \IT$ is a locally compact Hausdorff \'etale
groupoid, with multiplication given by $[\alpha][\beta] \coloneqq [\alpha\beta]$ for
$(\alpha,\beta) \in \GTc$, inversion given by $[\gamma]^{-1} \coloneqq [\gamma^{-1}]$ for $\gamma
\in \GT$, and range and source maps given by $r([\gamma]) = [r(\gamma)]$ and $s([\gamma]) =
[s(\gamma)]$.
\end{prop}
\begin{proof}
Together, \cref{rem: minimality,lemma: 1-cocycle c} allow us to apply \cite[Proposition~2.1]{KPS2016JNG} to see that $\IT$ is a closed subgroupoid of $\GT$. Therefore, \cite[Proposition~2.5(d)]{SW2016} implies that $\HT$ is a locally compact Hausdorff \'etale groupoid under the given operations.
\end{proof}

We conclude this section with two technical lemmas that we use in the proof of our characterisation of simplicity of $C^*(\GT,\sigma)$ in \cref{sec: simplicity}.

\begin{lemma} \label{lemma: IT U/V bar}
Let $(X,T)$ be a minimal rank-$k$ Deaconu--Renault system such that $X$ is second-countable. Let $c \colon \GT \to \Z^k$ be as in \cref{lemma: 1-cocycle c}. Fix $m, n \in \Z^k$, and let $U$ and $V$ be precompact open bisections of $\GT$ such that $U \subseteq c^{-1}(m)$ and $V \subseteq c^{-1}(n)$. Then $\overline{\IT U \medcap V} \subseteq \IT \overline{U}$ and $\overline{\IT V \medcap U} \subseteq \IT \overline{V}$.
\end{lemma}

\begin{proof}
Define $K \coloneqq r(\overline{U}) \times \{n - m\} \times r(\overline{U})$ and $W \coloneqq K \medcap \IT$. Since $r$ is continuous, $K$ is compact, and hence closed. Since $\IT$ is closed by \cref{prop: HT quotient groupoid}, $W$ is closed, and hence is a compact subset of $K$ and of $\IT$. We claim that $\IT U \medcap V \subseteq W \overline{U}$. For this, suppose that $\gamma \in \IT U \medcap V$. Then there exist $\xi \in \IT$ and $\eta \in U \subseteq c^{-1}(m)$ such that $\gamma = \xi \eta \in V \subseteq c^{-1}(n)$. Hence $\xi = \gamma \eta^{-1} \subseteq c^{-1}(n - m)$. We also have $r(\xi) = s(\xi) = r(\eta) \in r(\overline{U})$, and so $\xi \in K \medcap \IT = W$. Hence $\gamma = \xi \eta \in W \overline{U}$, and so $\IT U \medcap V \subseteq W \overline{U}$. Since $W$ and $\overline{U}$ are compact, $W \overline{U}$ is compact, and hence closed. Thus $\overline{\IT U \medcap V} \subseteq \IT \overline{U}$. A symmetric argument shows that $\overline{\IT V \medcap U} \subseteq \IT \overline{V}$.
\end{proof}

\begin{lemma} \label{lemma: s(gamma) in IT U cap V}
Let $(X,T)$ be a minimal rank-$k$ Deaconu--Renault system such that $X$ is second-countable. Let $c \colon \GT \to \Z^k$ be as in \cref{lemma: 1-cocycle c}. Fix $m, n \in \Z^k$, and let $U$ and $V$ be precompact open bisections of $\GT$ such that $U \subseteq c^{-1}(m)$ and $V \subseteq c^{-1}(n)$. Then $s(\IT U \medcap V) = s(\IT V \medcap U)$. Moreover, if $(V^{-1} \IT U) \medcap \IT \ne \varnothing$, then $n - m \in \PT$, and for each $\gamma \in (V^{-1} \IT U) \medcap \IT$, we have $s(\gamma) \in s(\IT U \medcap V)$.
\end{lemma}

\begin{proof}
We first show that $s(\IT U \medcap V) = s(\IT V \medcap U)$. By symmetry, it suffices to show that $s(\IT U \medcap V) \subseteq s(\IT V \medcap U)$. Suppose that $x \in s(\IT U \medcap V)$. Then there exist $\zeta \in \IT$ and $\eta \in U$ such that $\zeta \eta \in V$ and $x = s(\zeta \eta) = s(\eta)$. Since $\zeta^{-1} \in \IT$, we have $\eta = \zeta^{-1} (\zeta \eta) \in \IT V \medcap U$, and hence $x = s(\eta) \in s(\IT V \medcap U)$. Thus $s(\IT U \medcap V) \subseteq s(\IT V \medcap U)$, as required.

For the second statement, suppose that $\gamma \in (V^{-1} \IT U) \medcap \IT$. Then there exist $\alpha \in U$, $\beta \in V$, and $\xi \in \IT$ such that $\gamma = \beta^{-1} \xi \alpha$, and hence $c(\gamma) = -n + c(\xi) + m$. Since $\gamma, \xi \in \IT$, we have $c(\gamma), c(\xi) \in \PT$ by \cref{prop: IT characterisation}, and hence $n - m = c(\xi) - c(\gamma) \in \PT$, because $\PT$ is a group by \cref{prop: PT characterisation}. Since $\gamma \in \IT$, we have $s(\beta) = r(\gamma) = s(\gamma) = r(\alpha^{-1})$, and hence $(\beta,\alpha^{-1}) \in \GTc$. Since $\xi \in \IT$, we have $r(\beta\alpha^{-1}) = s(\beta^{-1}) = r(\xi) = s(\xi) = r(\alpha) = s(\beta\alpha^{-1})$. We also have $c(\beta\alpha^{-1}) = n - m \in \PT$, and thus \cref{prop: IT characterisation} implies that $\beta\alpha^{-1} \in \IT$. Hence $\beta = (\beta\alpha^{-1})\alpha \in \IT U \medcap V$, and so $s(\gamma) = r(\gamma) = s(\beta) \in s(\IT U \medcap V)$.
\end{proof}

\section{Cohomology of Deaconu--Renault groupoids}
\label{sec: cohomology}

In this section we show that every continuous $\T$-valued $2$-cocycle on a minimal Deaconu--Renault groupoid $\GT$ is cohomologous to a continuous $\T$-valued $2$-cocycle $\sigma$ on $\GT$ that is constant on $\IT$ (in the sense of \cref{def: constant on IT}). We also introduce the spectral action $\theta$ of $\HT \coloneqq \GT / \IT$, analogous to \cite[Lemma~3.6]{KPS2016JNG}.

\begin{definition} \label{def: constant on IT}
Let $(X,T)$ be a minimal rank-$k$ Deaconu--Renault system such that $X$ is second-countable. Suppose that $\sigma \in Z^2(\GT,\T)$. We say that $\sigma$ is \hl{constant on $\IT$} if
\[
\sigma\big((x,m,x), (x,n,x)\big) = \sigma\big((y,m,y),(y,n,y)\big) \quad \text{ for all } x, y \in X \text{ and } m, n \in \PT.
\]
If $\omega \in Z^2(\PT,\T)$ is the $2$-cocycle satisfying $\sigma\big((x,m,x), (x,n,x)\big) = \omega(m,n)$ for all $x \in X$ and $m, n \in \PT$, then we say that $\sigma$ is \hl{$\omega$-constant on $\IT$}, and we write $\sigma\bigrestr{\ITc} = 1_X \times \omega$.
\end{definition}

The following proposition and the lemmas used in its proof are extensions of cohomological results from \cite[Section~3]{KPS2016JNG} about boundary-path groupoids of cofinal, row-finite, source-free $k$-graphs to the more general setting of Deaconu--Renault groupoids. \Cref{prop: DR cohomologous vanishing bicharacter} is a generalisation of \cite[Proposition~3.1]{KPS2016JNG}, but we have adapted it slightly to prove that the bicharacter $\omega \in Z^2(\PT,\T)$ can be chosen in such a way that it vanishes on its centre, and hence descends to a bicharacter $\tilde{\omega} \in Z^2(\PT/Z_\omega,\T)$.

\begin{prop} \label{prop: DR cohomologous vanishing bicharacter}
Let $(X,T)$ be a minimal rank-$k$ Deaconu--Renault system such that $X$ is second-countable. Suppose that $\rho \in Z^2(\GT,\T)$. For each $x \in X$, define $\rho_x\colon \PT \times \PT \to \T$ by
\[
\rho_x(m,n) \coloneqq \rho\big((x,m,x),(x,n,x)\big).
\]
Then $\rho_x \in Z^2(\PT,\T)$. There exists a bicharacter $\omega \in Z^2(\PT,\T)$ such that $\omega$ vanishes on $Z_\omega$ in each coordinate, and $\omega$ is cohomologous to $\rho_x$ for every $x \in X$. For any such bicharacter $\omega$, there exists $\sigma \in Z^2(\GT,\T)$ such that $\sigma$ is cohomologous to $\rho$ and is $\omega$-constant on $\IT$ (in the sense of \cref{def: constant on IT}), and there exists a bicharacter $\tilde{\omega} \in Z^2(\PT/Z_\omega,\T)$ such that
\[
\tilde{\omega}(p+Z_\omega,q+Z_\omega) = \omega(p,q) \ \text{ for all } p, q \in \PT.
\]
\end{prop}

In order to prove \cref{prop: DR cohomologous vanishing bicharacter}, we need the following two results. The first of these results is an extension of \cite[Lemma~3.2]{KPS2016JNG} to the setting of Deaconu--Renault groupoids.

\begin{lemma} \label{lemma: twistchar_gamma^sigma}
Let $(X,T)$ be a minimal rank-$k$ Deaconu--Renault system such that $X$ is second-countable. Fix $\sigma \in Z^2(\GT,\T)$. For each $x \in X$, define $\sigma_x\colon
\PT \times \PT \to \T$ by
\[
\sigma_x(m,n) \coloneqq \sigma\big((x,m,x),(x,n,x)\big).
\]
Then $\sigma_x \in Z^2(\PT,\T)$. For $\gamma \in \GT$ and $y \coloneqq s(\gamma) \in X$, define $\twistchar_\gamma^\sigma\colon \PT \to \T$ by
\begin{equation} \label{eqn: twistchar_gamma^sigma(p)}
\twistchar_\gamma^\sigma(p) \coloneqq \sigma\big(\gamma,(y,p,y)\big) \, \sigma\big(\gamma(y,p,y),\gamma^{-1}\big) \, \overline{\sigma(\gamma,\gamma^{-1})}.
\end{equation}
\begin{enumerate}[label=(\alph*)]
\item \label{item: x to eval sigma_x cts} For all $m, n \in \PT$, the map $x \mapsto \sigma_x(m,n)$ from $X$ to $\T$ is continuous.
\item \label{item: twistchar_.^sigma cts} For each $p \in \PT$, the map $\gamma \mapsto \twistchar_\gamma^\sigma(p)$ from $\GT$ to $\T$ is continuous.
\item \label{item: twistchar_.^sigma in conjugation} Fix $\gamma = (x,m,y) \in \GT$, $p \in \PT$, and $w, z \in \T$. Under the multiplication and inversion operations on $\GT \times_\sigma \T$ (as defined in \cref{eqn: twisted groupoid multiplication,eqn: twisted groupoid inversion}), we have
\begin{equation} \label{eqn: chi conjugation}
(\gamma, w) \big((y,p,y), z\big) (\gamma, w)^{-1} = \big((x,p,x), \, \twistchar_\gamma^\sigma(p) z\big).
\end{equation}
\item \label{item: twistchar_gamma^sigma(p+q)} For all $\gamma \in \GT$ and $p, q \in \PT$, we have
\begin{equation} \label{eqn: twistchar_gamma^sigma(p+q)}
\twistchar_\gamma^\sigma(p+q) = \sigma_{r(\gamma)}(p,q)\, \overline{\sigma_{s(\gamma)}(p,q)}\, \twistchar_\gamma^\sigma(p) \, \twistchar_\gamma^\sigma(q).
\end{equation}
\item \label{item: twistchar_.^sigma 1-cocycle} If $\omega$ is a bicharacter of $\PT$ such that $\sigma$ is $\omega$-constant on $\IT$, then $\sigma_x = \sigma_y$ for all $x, y \in X$, and $\gamma \mapsto \twistchar_\gamma^\sigma$ is a continuous $\PTHat$-valued $1$-cocycle on $\GT$.
\end{enumerate}
\end{lemma}

\begin{proof}
Routine calculations show that since $\sigma$ is normalised and satisfies the $2$-cocycle identity, we have $\sigma_x \in Z^2(\PT,\T)$ for each $x \in X$.

For part~\cref{item: x to eval sigma_x cts}, note that for each $m, n \in \PT$, the map $x \mapsto \sigma_x(m,n)$ is the composition of the continuous maps $x \mapsto \!\big((x,m,x),(x,n,x)\big)$ and $\sigma$.

For part~\cref{item: twistchar_.^sigma cts}, fix $p \in \PT$. For $\gamma \in \GT$,
\[
\twistchar_\gamma^\sigma(p) \,=\, \sigma\big(\gamma,(s(\gamma),p,s(\gamma))\big) \, \sigma\big(\gamma \, (s(\gamma),p,s(\gamma)),\gamma^{-1}\big) \, \overline{\sigma\big(\gamma,\gamma^{-1}\big)}.
\]
Thus the map $\gamma \mapsto \twistchar_\gamma^\sigma(p)$ from $\GT$ to $\T$ is continuous because it is a product of continuous functions.

For part~\cref{item: twistchar_.^sigma in conjugation}, fix $\gamma = (x,m,y) \in \GT$, $p \in \PT$, and $w, z \in \T$. We have
\[
\gamma (y,p,y) \gamma^{-1} = (x,m,y)(y,p,y)(y,-m,x) = (x,p,x),
\]
and hence
\begin{align*}
(\gamma, w) \big((y,p,y), z\big) (\gamma, w)^{-1} &= \big(\gamma(y,p,y),\, \sigma(\gamma,(y,p,y)) \, wz\big) \big(\gamma^{-1}, \, \overline{\sigma(\gamma,\gamma^{-1})} \, \overline{w}\big) \\
&= \big(\gamma(y,p,y)\gamma^{-1}, \, \sigma(\gamma(y,p,y), \gamma^{-1}) \, \sigma(\gamma,(y,p,y)) \, \overline{\sigma(\gamma,\gamma^{-1})} \, z\big) \\
&= \big((x,p,x), \, \twistchar_\gamma^\sigma(p) z\big).
\end{align*}

For part~\cref{item: twistchar_gamma^sigma(p+q)}, fix $\gamma = (x,m,y) \in \GT$ and $p, q \in \PT$. For all $z \in \T$, we have
\begin{align*}
\big((y,p,y),1\big) \big((y,q,y),z\big) &= \big((y,p,y)(y,q,y), \, \sigma\big((y,p,y),(y,q,y)\big) z\big) \\
&= \big((y,p+q,y), \, \sigma_y(p,q) z\big),
\end{align*}
and so, taking $z = \overline{\sigma_y(p,q)}$, we see that
\begin{equation} \label{eqn: y p plus q y}
\big((y,p,y), 1\big) \big((y,q,y),\, \overline{\sigma_y(p,q)}\big) = \big((y,p+q,y), 1\big).
\end{equation}
Together, \cref{eqn: chi conjugation,eqn: y p plus q y} imply that
\begin{align*}
\big((x,p+q,x), \, \twistchar_\gamma^\sigma(p+q)\big) &= (\gamma,1) \, \big((y,p+q,y), 1\big) \, (\gamma,1)^{-1} \\
&= (\gamma,1) \, \big((y,p,y), 1\big) \big((y,q,y), \, \overline{\sigma_y(p,q)}\big) \, (\gamma,1)^{-1} \\
&= (\gamma,1) \, \big((y,p,y), 1\big) \, (\gamma,1)^{-1} (\gamma,1) \, \big((y,q,y), \, \overline{\sigma_y(p,q)}\big) \, (\gamma,1)^{-1} \\
&= \big((x,p,x), \, \twistchar_\gamma^\sigma(p)\big) \, \big((x,q,x), \, \twistchar_\gamma^\sigma(q) \, \overline{\sigma_y(p,q)}\big) \\
&= \big((x,p+q,x), \, \sigma_x(p,q) \, \overline{\sigma_y(p,q)} \, \twistchar_\gamma^\sigma(p) \, \twistchar_\gamma^\sigma(q)\big),
\end{align*}
and hence
\[
\twistchar_\gamma^\sigma(p+q) = \sigma_{r(\gamma)}(p,q) \, \overline{\sigma_{s(\gamma)}(p,q)} \, \twistchar_\gamma^\sigma(p) \, \twistchar_\gamma^\sigma(q).
\]

For part~\cref{item: twistchar_.^sigma 1-cocycle}, since $\sigma$ is $\omega$-constant on $\IT$, for all $x, y \in X$ and $p, q \in \PT$, we have
\[
\sigma_x(p,q) = \sigma\big((x,p,x),(x,q,x)\big) = \omega(p,q) = \sigma\big((y,p,y),(y,q,y)\big) = \sigma_y(p,q).
\]
So $\sigma_x = \sigma_y$, and for each $\gamma \in \GT$, \cref{eqn: twistchar_gamma^sigma(p+q)} reduces to
\[
\twistchar_\gamma^\sigma(p+q) = \twistchar_\gamma^\sigma(p)\, \twistchar_\gamma^\sigma(q).
\]
Thus $\twistchar_\gamma^\sigma\colon \PT \to \T$ is a homomorphism, and so
$\twistchar_\gamma^\sigma \in \PTHat$ for each $\gamma \in \GT$.

We now show that the map $\gamma \mapsto \twistchar_\gamma^\sigma$ is multiplicative. Fix $\alpha = (x,m,u),\, \beta = (u,n,y) \in \GT$, and $p \in \PT$. Using \cref{eqn: chi conjugation}, we compute
\begin{align*}
\big((x,p,x), \, \twistchar_\alpha^\sigma(p) \, \twistchar_{\beta}^\sigma(p)\big) &= (\alpha,1) \, \big((u,p,u),\, \twistchar_{\beta}^\sigma(p)\big) \, (\alpha,1)^{-1} \\
&= (\alpha,1) (\beta,1) \, \big((y,p,y), 1\big) \, (\beta,1)^{-1} (\alpha,1)^{-1} \\
&= \big(\alpha\beta, \, \sigma(\alpha,\beta)\big) \big((y,p,y), 1\big) \big(\alpha\beta, \, \sigma(\alpha,\beta)\big)^{-1} \\
&= \big((x,p,x), \, \twistchar_{\alpha\beta}^\sigma(p)\big).
\end{align*}
Hence $\twistchar_\alpha^\sigma\, \twistchar_{\beta}^\sigma = \twistchar_{\alpha\beta}^\sigma$, and so $\gamma \mapsto \twistchar_\gamma^\sigma$ is a $\PTHat$-valued $1$-cocycle on $\GT$.

We conclude by showing that the map $\gamma \mapsto \twistchar_\gamma^\sigma$ is continuous. Fix a finite subset $F \subseteq \PT$ and an open subset $U \subseteq \T$. The set
\[
S_{\PTHat}(F,U) = \{ \phi \in \PTHat \,:\, \phi(F) \subseteq U \}
\]
is a typical subbasis element for the compact-open topology on $\PTHat$, and so it suffices to show that $\big\{ \gamma \in \GT : \twistchar_\gamma^\sigma \in S_{\PTHat}(F,U) \big\}$ is an open subset of $\GT$. We have
\[
\big\{ \gamma \in \GT : \twistchar_\gamma^\sigma(F) \subseteq U \big\} = \bigcap_{p \in F} \big\{\gamma \in \GT : \twistchar_\gamma^\sigma(p) \in U \big\},
\]
which is open by part~\cref{item: twistchar_.^sigma cts}.
\end{proof}

The following lemma is an extension of \cite[Lemma~3.3]{KPS2016JNG} to the setting of Deaconu--Renault groupoids.

\begin{lemma} \label{lemma: [sigma_x]=[sigma_y]}
Let $(X,T)$ be a minimal rank-$k$ Deaconu--Renault system such that $X$ is second-countable. Fix $\sigma \in Z^2(\GT,\T)$. As in \cref{lemma: twistchar_gamma^sigma}, for each $x \in X$, define $\sigma_x \in Z^2(\PT,\T)$ by
\[
\sigma_x(m,n) \coloneqq \sigma\big((x,m,x),(x,n,x)\big).
\]
Then the cohomology class of $\sigma_x$ does not depend on $x$.
\end{lemma}

\begin{proof}
By \cite[Proposition~3.2]{OPT1980}, it suffices to show that $\sigma_x\sigma_x^* =
\sigma_y\sigma_y^*$ for all $x,y \in X$. By \cref{prop: PT characterisation}, $\PT$ is a subgroup of the finitely generated free abelian group $\Z^k$, and so $\PT \cong \Z^l$ for some $l \le k$. Fix free abelian generators $g_1, \dotsc, g_l$ of $\PT$. Since each $\sigma_x\sigma_x^*$ is a bicharacter (by \cite[Proposition~3.2]{OPT1980}), it suffices to show that $(\sigma_x\sigma_x^*)(g_i,g_j) = (\sigma_y\sigma_y^*)(g_i,g_j)$ for all $i,j \in \{1,\dotsc,l\}$ and $x,y \in X$. To see this, we first show that $\sigma_{r(\gamma)}\sigma_{r(\gamma)}^*(g_i,g_j) = \sigma_{s(\gamma)}\sigma_{s(\gamma)}^*(g_i,g_j)$, for all $\gamma \in \GT$ and $i,j \in \{1,\dotsc,l\}$.

Let $\PT \times_\sigma \T \coloneqq \PT \times \T$ be the semidirect product group, which is equal to $\PT \times \T$ as a set, but has group operation
\[
(p,w)(q,z) = \big(p + q, \, \sigma(p,q) \, w \, z\big).
\]
Define $i_\sigma\colon \T \to \PT \times_\sigma \T$ by $i_\sigma(z) = (0, z)$ and $q_\sigma\colon \PT \times_\sigma \T \to \PT$ by $q_\sigma(p,z) = p$. Consider the bijection $M\colon H^2(\PT,\T) \to \Ext(\PT,\T)$ that maps the cohomology class of a $2$-cocycle $\sigma \in Z^2(\PT,\T)$ to the congruence class of the central extension
\[
1 \to \T \xrightarrow{i_\sigma} \PT \times_\sigma \T \xrightarrow{q_\sigma} \PT \to 0
\]
(see \cite[Theorem~IV.3.12]{Brown1982}). Fix $\gamma \in \GT$. We aim to prove that $\sigma_{s(\gamma)}$ and $\sigma_{r(\gamma)}$ are cohomologous by showing that their cohomology classes have the same image under $M$. So we must find a homomorphism
\[
\varphi_\gamma\colon \PT \times_{\sigma_{s(\gamma)}} \T \to \PT \times_{\sigma_{r(\gamma)}} \T
\]
that makes the diagram
\begin{center}
\vspace{-3ex}
\begin{equation} \label{diagram: r and s class extensions}
\begin{tikzcd}
{} & {\PT \times_{\sigma_{s(\gamma)}} \T} \arrow{dd}{\varphi_\gamma} \arrow{dr}{q_{\sigma_{s(\gamma)}}} & \\
\T \arrow{ur}{i_{\sigma_{s(\gamma)}}} \arrow{dr}{i_{\sigma_{r(\gamma)}}} && {\PT} \\
& {\PT \times_{\sigma_{r(\gamma)}} \T} \arrow{ur}{q_{\sigma_{r(\gamma)}}} & {}
\end{tikzcd}
\end{equation}
\end{center}
commute. Let $\twistchar_\gamma^\sigma\colon \PT \to \T$ be the map of \cref{lemma: twistchar_gamma^sigma}, and define $\varphi_\gamma \colon \PT \times_{\sigma_{s(\gamma)}} \T \to \PT \times_{\sigma_{r(\gamma)}} \T$ by $\varphi_\gamma(m,z) \coloneqq \left(m, \, \twistchar_\gamma^\sigma(m) z\right)$. Fix $(m,z), (n,w) \in \PT \times_{\sigma_{s(\gamma)}} \T$. Recalling from \cref{lemma: twistchar_gamma^sigma}\cref{item: twistchar_gamma^sigma(p+q)} that
\[
\twistchar_\gamma^\sigma(m+n) = \sigma_{r(\gamma)}(m,n)\, \overline{\sigma_{s(\gamma)}(m,n)}\, \twistchar_\gamma^\sigma(m)\, \twistchar_\gamma^\sigma(n),
\]
we obtain
\begin{align*}
\varphi_\gamma\big((m,z)(n,w)\big) &= \varphi_\gamma\big(m+n,\,\sigma_{s(\gamma)}(m,n) \, zw\big) \\
&= \big(m+n, \, \twistchar_\gamma^\sigma(m+n) \, \sigma_{s(\gamma)}(m,n) \, zw\big) \\
&= \big(m+n, \, \sigma_{r(\gamma)}(m,n) \, \twistchar_\gamma^\sigma(m) \, \twistchar_\gamma^\sigma(n) \, zw\big) \\
&= \big(m, \, \twistchar_\gamma^\sigma(m) z\big) \big(n, \, \twistchar_\gamma^\sigma(n) w\big) \\
&= \varphi_\gamma(m,z) \, \varphi_\gamma(n,w),
\end{align*}
and thus $\varphi_\gamma$ is a homomorphism. Since $\sigma$ is normalised, the formula~\labelcref{eqn: twistchar_gamma^sigma(p)} from \cref{lemma: twistchar_gamma^sigma} gives $\twistchar_\gamma^\sigma(0) = 1$, and it follows that the diagram~\labelcref{diagram: r and s class extensions} commutes. Therefore, $\sigma_{r(\gamma)}$ is cohomologous to $\sigma_{s(\gamma)}$, and so \cite[Proposition~3.2]{OPT1980} implies that
\begin{equation} \label{eqn: cocycles equivalent on generators}
\big(\sigma_{r(\gamma)}\sigma_{r(\gamma)}^*\big)(g_i,g_j) = \big(\sigma_{s(\gamma)}\sigma_{s(\gamma)}^*\big)(g_i,g_j) \quad \text{for all } i, j \in \{1,\dotsc,l\},
\end{equation}
as claimed.

Now fix $x,y \in X$. Since $(X,T)$ is minimal, there is a sequence $(\gamma_n)_{n\in\N}$ in $\GT$ such that $s(\gamma_n) = x$ for all $n \in \N$, and $r(\gamma_n) \to y$ as $n \to \infty$. Fix $i,j \in \{1,\dotsc,l\}$. By \cref{lemma: twistchar_gamma^sigma}\cref{item: x to eval sigma_x cts}, the map $u \mapsto \sigma_u(g_i,g_j)$ is continuous, and hence the map $u \mapsto \big(\sigma_u\sigma_u^*\big)(g_i,g_j)$ is continuous. So $(\sigma_y \sigma^*_y)(g_i, g_j) = \lim_{n \to \infty} (\sigma_{r(\gamma_n)} \sigma^*_{r(\gamma_n)})(g_i, g_j)$. \Cref{eqn: cocycles equivalent on generators} gives $(\sigma_{r(\gamma_n)} \sigma^*_{r(\gamma_n)})(g_i, g_j) = (\sigma_x \sigma^*_x)(g_i, g_j)$ for each $n \in \N$, and so $\big(\sigma_y\sigma_y^*\big)(g_i,g_j) = \big(\sigma_x\sigma_x^*\big)(g_i,g_j)$.
\end{proof}

\begin{proof}[Proof of \cref{prop: DR cohomologous vanishing bicharacter}]
\Cref{lemma: [sigma_x]=[sigma_y]} shows that $\rho_x$ is a $\T$-valued $2$-cocycle on $\PT$ whose cohomology class is independent of $x$. So there exists a $2$-cocycle $\omega \in Z^2(\PT,\T)$ whose cohomology class agrees with that of each $\rho_x$. As discussed in \cref{sec: group cohomology} (see \cite[Theorem~2.2.8]{Armstrong2019}), we may assume that $\omega$ is a bicharacter that vanishes on $Z_\omega$ in each coordinate, and that there is a bicharacter
$\tilde{\omega} \in Z^2(\PT/Z_\omega,\T)$ such that
\[
\tilde{\omega}(p+Z_\omega,q+Z_\omega) = \omega(p,q) \quad \text{for all } p, q \in \PT.
\]

We now construct $\sigma \in Z^2(\GT,\T)$ such that $\sigma$ is cohomologous to $\rho$, and $\sigma$ is $\omega$-constant on $\IT$. For each $x \in X$, the $2$-cocycles $\rho_x$ and $\omega$ are cohomologous, and so the map $\tilde{c}_x\colon \PT \times \PT \to \T$ defined by
\[
\tilde{c}_x(p,q) \coloneqq \overline{\omega(p,q)} \rho_x(p,q)
\]
is a $2$-coboundary on $\PT$. Since $\PT$ is a subgroup of $\Z^k$ (by \cref{prop: PT characterisation}), there is an integer $l \in \{1,\dotsc,k\}$ such that $\PT \cong \Z^l$. Fix free abelian generators $g_1, \dotsc, g_l$ for $\PT$. For $m \in \PT$, let $m_1, \dotsc, m_l$ be the unique integers such that $m = \sum_{i=1}^l m_i g_i$. For each $i \in \{1, \dotsc, l\}$, we write $\langle g_j : j \le i \rangle$ for the group generated by the set $\{ g_j : 1 \le j \le i \}$. We claim that there are maps $b_x\colon \PT \to \T$, indexed by $x \in X$, such that $x \mapsto b_x(m)$ is continuous for each $m \in \PT$, and for each $i \in \{1,\dotsc,l\}$, we have
\begin{equation} \label{eqn: b_x inductive}
b_x(m) \, \overline{b_x(m+g_i)} = \tilde{c}_x(g_i,m), \quad \text{whenever } m \in \langle g_j : j \le i \rangle.
\end{equation}
To see this, for each $x \in X$ define $b_x(0) \coloneqq 1 \in \T$. The map $x \mapsto b_x(0)$ is trivially continuous. Fix $i \in \{1,\dotsc,l\}$. Suppose inductively that the maps $b_x$ have been defined on $\langle g_j : j < i \rangle$, and that $x \mapsto b_x(m)$ is continuous for each $m \in \langle g_j : j < i \rangle$. To extend $b_x$ to $\langle g_j : j \le i \rangle$, first observe that $b_x(m)$ is already defined when $m = \sum_{j=1}^i m_j g_j$ and $m_i = 0$. Now suppose inductively that $b_x(m)$ is defined and $x \mapsto b_x(m)$ is continuous whenever $\lav m_i \rav \le a$ for some $a \in \N$, and that $b_x$ satisfies \cref{eqn: b_x inductive} whenever $\lav m_i \rav, \, \lav m_i + 1 \rav \le a$. Fix $m \in \langle g_j : j \le i \rangle$ such that $\lav m_i \rav = a + 1$. Define
\[
b_x(m) \coloneqq
\begin{cases}
b_x(m-g_i) \, \overline{\tilde{c}_x(g_i,m-g_i)} & \text{if } m_i > 0 \\
b_x(m+g_i) \, \tilde{c}_x(g_i,m) & \text{if } m_i < 0.
\end{cases}
\]
Since \cref{lemma: twistchar_gamma^sigma}\cref{item: x to eval sigma_x cts} implies that the maps $x \mapsto \tilde{c}_x(p,q)$ are continuous for all $p, q \in \PT$, the inductive hypothesis guarantees that $x \mapsto b_x(m)$ is continuous. Moreover, rearranging each of the cases in the definition of $b_x(m)$ shows that \cref{eqn: b_x inductive} is satisfied. So the claim follows by induction.

Recall the coboundary map $\delta^1$ of \cref{item: 2-coboundary}. We claim that $\delta^1 b_x = \tilde{c}_x$. To see this, first choose a normalised $1$-cochain $\tilde{b}_x\colon \PT \to \T$ such that $\delta^1 \tilde{b}_x = \tilde{c}_x$. (This is possible because $\tilde{c}_x$ is a $2$-coboundary on $\PT$.) Define $a_x\colon \PT \to \T$ by
\[
a_x(m) \coloneqq \prod_{i=1}^l \overline{\tilde{b}_x(g_i)^{m_i}}.
\]
A straightforward calculation shows that $a_x$ is a $1$-cocycle, and so $\delta^1 a_x$ is trivial. Hence $\delta^1(a_x\tilde{b}_x) = \delta^1 \tilde{b}_x = \tilde{c}_x$. Putting $m = 0$ in \cref{eqn: b_x inductive}, we see that for each $i \in \{1,\dotsc,l\}$, $b_x(g_i) = 1$. Hence
\[
(a_x \tilde{b}_x)(0) = \Bigg(\prod_{i=1}^l \overline{\tilde{b}_x(g_i)^0}\Bigg) \tilde{b}_x(0) = 1 = b_x(0),
\]
and for each $i \in \{1, \dotsc, l\}$,
\[
(a_x \tilde{b}_x)(g_i) = \Bigg(\prod_{\substack{j=1, \\ j \ne i}}^l \overline{\tilde{b}_x(g_j)^0}\Bigg)\!\left(\overline{\tilde{b}_x(g_i)^1}\right)\tilde{b}_x(g_i) = 1 = b_x(g_i).
\]
Thus, for all $i \in \{1,\dotsc,l\}$ and $m \in \langle g_j : j \le i \rangle$, we have
\begin{align*}
(a_x \tilde{b}_x)(m) \, \overline{(a_x \tilde{b}_x)(m+g_i)} \,&=\, (a_x \tilde{b}_x)(g_i) \, (a_x \tilde{b}_x)(m) \, \overline{(a_x \tilde{b}_x)(g_i+m)} \\
&=\, \delta^1(a_x \tilde{b}_x)(g_i,m) \\
&=\, \tilde{c}_x(g_i,m) \\
&=\, b_x(m) \, \overline{b_x(m+g_i)}.
\end{align*}
So $b_x$ and $a_x \tilde{b}_x$ both map $0$ and each generator $g_i$ to $1$, and they also both satisfy \cref{eqn: b_x inductive}. Hence $a_x \tilde{b}_x = b_x$, and thus $\delta^1 b_x = \delta^1(a_x \tilde{b}_x) = \tilde{c}_x$, as claimed.

Since the maps $(x,p,x) \mapsto x$ and $x \mapsto b_x(p)$ are both continuous for each fixed $p \in \PT$, the map $\tilde{b}\colon \IT \to \T$ given by $\tilde{b}(x,p,x) \coloneqq b_x(p)$ is a continuous $1$-cochain on $\IT$. We extend $\tilde{b}$ to a map $b\colon \GT \to \T$ by setting $b(\gamma) \coloneqq 1$ for all $\gamma \in \GT {\setminus} \IT$. Since $\IT$ is a clopen subset of $\GT$ (by \cref{prop: HT quotient groupoid}), this map $b$ is a continuous $1$-cochain on $\GT$. We have $b(x,0,x) = b_x(0) = 1$ for all $x \in X$, and so $b$ is normalised. Thus the map $\delta^1 b\colon \GTc \to \T$ given by $\delta^1 b(\alpha,\beta) \coloneqq b(\alpha) \, b(\beta) \, b(\alpha\beta)^{-1}$ is a continuous $2$-coboundary on $\GT$. Define $\sigma \in Z^2(\GT,\T)$ by $\sigma(\alpha,\beta) \coloneqq \rho(\alpha,\beta) \, \overline{\delta^1 b(\alpha,\beta)}$. Since $\sigma$ and $\rho$ differ by the $2$-coboundary $\delta^1 b$, they are cohomologous, and so \cite[Proposition~II.1.2]{Renault1980} implies that $C^*(\GT,\rho) \cong C^*(\GT,\sigma)$. Finally, fix $x \in X$ and $p, q \in \PT$. Since $\delta^1 b_x = \tilde{c}_x = \rho_x \, \omega$, we have
\[
\sigma\big((x,p,x),(x,q,x)\big) = \rho_x(p,q) \, \overline{\delta^1 b_x(p,q)} = \omega(p,q),
\]
and so $\sigma$ is $\omega$-constant on $\IT$.
\end{proof}

The following result is an extension of \cite[Lemma~3.6]{KPS2016JNG} to the setting of Deaconu--Renault groupoids.

\begin{prop} \label{prop: spectral action}
Let $(X,T)$ be a minimal rank-$k$ Deaconu--Renault system such that $X$ is second-countable. Suppose that $\sigma \in Z^2(\GT,\T)$, and that $\omega \in Z^2(\PT,\T)$ is a bicharacter that vanishes on $Z_\omega$ in each coordinate such that $\sigma$ is $\omega$-constant on $\IT$, as in \cref{prop: DR cohomologous vanishing bicharacter}. Let $\gamma \mapsto \twistchar_\gamma^\sigma$ be the continuous $\PTHat$-valued $1$-cocycle on $\GT$ defined in \cref{lemma: twistchar_gamma^sigma}\cref{item: twistchar_.^sigma 1-cocycle}. For all $\gamma \in \IT$ and $p \in Z_\omega$, we have $\twistchar_\gamma^\sigma(p) = 1$. Let $\HT = \GT / \IT$ be the quotient groupoid of \cref{prop: HT quotient groupoid}. There is a continuous $\widehat{Z}_\omega$-valued $1$-cocycle $[\gamma] \mapsto \tilde{\twistchar}_{[\gamma]}^\sigma$ on $\HT$ such that $\tilde{\twistchar}^\sigma_{[\gamma]}(p) = \twistchar_\gamma^\sigma(p)$ for all $\gamma \in \GT$ and $p \in Z_\omega$. There is a continuous action $\theta$ of $\HT$ on $X \times \widehat{Z}_\omega$ such that
\[
\theta_{[\gamma]}\big(s(\gamma),\, \chi\big) = \big(r(\gamma),\, \tilde{\twistchar}_{[\gamma]}^\sigma \, \chi\big) \ \text{ for all $\gamma \in \GT$ and $\chi \in \widehat{Z}_\omega$}.
\]
\end{prop}

We call the action $\theta$ of \cref{prop: spectral action} the \hl{spectral action} associated to $(T,\sigma)$. We denote the orbit of $(x,\chi) \in X \times \widehat{Z}_\omega$ under $\theta$ by $[x, \chi]_\theta$.

\begin{proof}[{Proof of \cref{prop: spectral action}}]
Fix $\gamma \in \IT$. \Cref{prop: IT characterisation} implies that there exist $y \in X$ and $m \in \PT$ such that $\gamma = (y,m,y)$. We claim that $\twistchar_\gamma^\sigma(Z_\omega) = \{1\}$. Fix $p \in Z_\omega$. Using the formula~\labelcref{eqn: twistchar_gamma^sigma(p)} from \cref{lemma: twistchar_gamma^sigma}, and that $\omega$ is a bicharacter satisfying $\sigma\bigrestr{\ITc} = 1_X \times \omega$, and that $\omega\omega^*$ is an antisymmetric bicharacter, we see that
\begin{align*}
\twistchar_\gamma^\sigma(p) &= \sigma\big(\gamma,(y,p,y)\big) \, \sigma\big(\gamma(y,p,y),\gamma^{-1}\big) \, \overline{\sigma\big(\gamma,\gamma^{-1}\big)} \\
&= \sigma\big((y,m,y),(y,p,y)\big) \, \sigma\big((y,m+p,y),(y,-m,y)\big) \, \overline{\sigma\big((y,m,y),(y,-m,y)\big)} \\
&= \omega(m,p)\, \omega(m+p,-m)\, \overline{\omega(m,-m)} \\
&= \omega(m,p)\, \overline{\omega(p,m)} \\
&= \overline{(\omega\omega^*)(p,m)},
\end{align*}
which is $1$ because $p \in Z_\omega$. Thus $\twistchar_\gamma^\sigma(Z_\omega) = \{1\}$, as claimed.

For any $\gamma \in \GT$, we have $\twistchar_\gamma^\sigma \in \PTHat$ by \cref{lemma: twistchar_gamma^sigma}\cref{item: twistchar_.^sigma 1-cocycle}, and so $\twistchar_\gamma^\sigma\restr{Z_\omega} \in \widehat{Z}_\omega$. Suppose that $\alpha,\beta \in \GT$ satisfy $[\alpha] = [\beta]$. Then $\eta = \beta^{-1}\alpha \in \IT$ satisfies $\alpha = \beta\eta$. For $p \in Z_\omega$, we have $\twistchar_\eta^\sigma(p) = 1$, and thus, since $\gamma \mapsto \twistchar_\gamma^\sigma$ is a $1$-cocycle,
\[
\twistchar_\alpha^\sigma(p) = \twistchar_{\beta\eta}^\sigma(p) = \twistchar_\beta^\sigma(p) \, \twistchar_\eta^\sigma(p) = \twistchar_\beta^\sigma(p).
\]
Therefore, there is a map $[\gamma] \mapsto \tilde{\twistchar}^\sigma_{[\gamma]}$ from $\HT$ to $\widehat{Z}_\omega$ such that $\tilde{\twistchar}_{[\gamma]}^\sigma(p) = \twistchar_\gamma^\sigma(p)$ for all $\gamma \in \GT$ and $p \in Z_\omega$. For $\alpha, \beta \in \GT$ and $p \in Z_\omega$,
\[
\tilde{\twistchar}_{[\alpha][\beta]}^\sigma(p) = \tilde{\twistchar}_{[\alpha\beta]}^\sigma(p) = \twistchar_{\alpha\beta}^\sigma(p) = \twistchar_\alpha^\sigma(p) \, \twistchar_\beta^\sigma(p) = \tilde{\twistchar}_{[\alpha]}^\sigma(p) \, \tilde{\twistchar}_{[\beta]}^\sigma(p),
\]
and so $\tilde{\twistchar}_{[\alpha][\beta]}^\sigma = \tilde{\twistchar}_{[\alpha]}^\sigma \, \tilde{\twistchar}_{[\beta]}^\sigma$. Thus $[\gamma] \mapsto \tilde{\twistchar}_{[\gamma]}^\sigma$ is a $\widehat{Z}_\omega$-valued $1$-cocycle on $\HT$.

We claim that $[\gamma] \mapsto \tilde{\twistchar}_{[\gamma]}^\sigma$ is continuous on $\HT$. Fix a finite subset $F \subseteq Z_\omega$ and an open subset $U \subseteq \T$, so that $S_{\widehat{Z}_\omega}(F,U) \coloneqq \{ \chi \in \widehat{Z}_\omega : \chi(F) \subseteq U \}$ is a typical subbasis element for the topology on $\widehat{Z}_\omega$. It suffices to show that $\big\{ [\gamma] \in \HT : \tilde{\twistchar}_{[\gamma]}^\sigma \in S_{\widehat{Z}_\omega}(F,U) \big\}$ is open in $\HT$. Since $F$ is finite,
\[
S_{\PTHat}(F,U) \coloneqq \{ \chi \in \PTHat \,:\, \chi(F) \subseteq U \}
\]
is open in $\PTHat$. By \cref{lemma: twistchar_gamma^sigma}\cref{item: twistchar_.^sigma 1-cocycle}, the map $\gamma \mapsto \twistchar_\gamma^\sigma$ is continuous on $\GT$, and hence $\{\gamma \in \GT : \twistchar_\gamma^\sigma \in S_{\PTHat}(F,U)\}$ is open in $\GT$. Let $\pi_T\colon \GT \to \HT$ denote the quotient map $\gamma \mapsto [\gamma]$. Then $\pi_T^{-1}\big(\big\{ [\gamma] : \tilde{\twistchar}_{[\gamma]}^\sigma \in S_{\widehat{Z}_\omega}(F,U) \big\}\big) = \{ \gamma \in \GT : \twistchar_\gamma^\sigma(F) \subseteq U \}$ is open. Thus, by the definition of the quotient topology, $\big\{ [\gamma] : \tilde{\twistchar}_{[\gamma]}^\sigma \in S_{\widehat{Z}_\omega}(F,U) \big\}$ is open in $\HT$.

It remains to show that $\theta$ is a continuous action of $\HT$ on $X \times \widehat{Z}_\omega$. For $\alpha, \beta \in \GT$ such that $[\alpha] = [\beta]$, we have $\alpha\beta^{-1} \in \IT$, and hence $r(\alpha) = r(\beta)$ and $s(\alpha) = s(\beta)$. Define $R\colon X \times \widehat{Z}_\omega\, \to\, \HTo$ by $R(x,\chi) \coloneqq [x]$. Then $R$ is continuous and surjective. Recall from \cref{def: groupoid action} that the fibred product $\HT \star (X \times \widehat{Z}_\omega)$ is defined by
\[
\HT \star (X \times \widehat{Z}_\omega) = \big\{ \big([\gamma],\, (x,\chi)\big) \,:\, x \in X, \, \chi \in \widehat{Z}_\omega,\, \gamma \in (\GT)_x \big\}.
\]
Since $[\gamma] \mapsto \tilde{\twistchar}_{[\gamma]}^\sigma$ is a continuous map from $\HT$ to $\widehat{Z}_\omega$ and $r\colon \HT \to X$ is continuous, the map $\big([\gamma],\, (s(\gamma),\, \chi)\big) \mapsto \theta_{[\gamma]}\big(s(\gamma),\, \chi\big) = \big(r(\gamma),\, \tilde{\twistchar}_{[\gamma]}^\sigma \, \chi\big)$ from $\HT \star (X \times \widehat{Z}_\omega)$ to $X \times \widehat{Z}_\omega$ is continuous. To see that $\theta$ is an action, we must show that conditions \cref{item: composing groupoid actions,item: action of groupoid unit} of \cref{def: groupoid action} are satisfied.

For \cref{item: composing groupoid actions}, fix $(x,\chi) \in X \times \widehat{Z}_\omega$ and $\big([\alpha],[\beta]\big) \in \HTc$ such that $\big([\beta],\, (x,\chi)\big) \in \HT \star (X \times \widehat{Z}_\omega)$. Then $s([\alpha]) = r([\beta])$, and $s([\beta]) = R(x,\chi) = [x]$. Hence $s([\alpha][\beta]) = s([\beta]) = R(x,\chi)$, and so $\big([\alpha][\beta],\, (x,\chi)\big) \in \HT \star (X \times \widehat{Z}_\omega)$. Since $s(\beta) = x$,
\[
\theta_{[\beta]}(x,\chi) = \theta_{[\beta]}\big(s(\beta),\chi\big) = \big(r(\beta),\, \tilde{\twistchar}_{[\beta]}^\sigma \, \chi\big).
\]
Thus
\[
R\big(\theta_{[\beta]}(x,\chi)\big) = R\big(r(\beta),\, \tilde{\twistchar}_{[\beta]}^\sigma \, \chi\big) = [r(\beta)] = r([\beta]) = s([\alpha]),
\]
and so $\big([\alpha], \, \theta_{[\beta]}(x,\chi)\big) \in \HT \star (X \times \widehat{Z}_\omega)$. Finally, since $[\gamma] \mapsto \tilde{\twistchar}_{[\gamma]}^\sigma$ is a $\widehat{Z}_\omega$-valued $1$-cocycle on $\HT$, we have $\tilde{\twistchar}_{[\alpha]}^\sigma \, \tilde{\twistchar}_{[\beta]}^\sigma = \tilde{\twistchar}_{[\alpha][\beta]}^\sigma = \tilde{\twistchar}_{[\alpha\beta]}^\sigma$, and hence
\begin{align*}
\theta_{[\alpha]}\big(\theta_{[\beta]}(x,\chi)\big) &= \theta_{[\alpha]}\big(\theta_{[\beta]}(s(\beta),\, \chi)\big) = \theta_{[\alpha]}\big(r(\beta), \, \tilde{\twistchar}_{[\beta]}^\sigma \, \chi\big) = \theta_{[\alpha]}\big(s(\alpha), \, \tilde{\twistchar}_{[\beta]}^\sigma \, \chi\big) \\
&= \big(r(\alpha),\, \tilde{\twistchar}_{[\alpha]}^\sigma \, (\tilde{\twistchar}_{[\beta]}^\sigma \, \chi)\big) = \big(r(\alpha\beta), \, \tilde{\twistchar}_{[\alpha\beta]}^\sigma \, \chi\big) = \theta_{[\alpha\beta]}\big(s(\alpha\beta), \, \chi\big) = \theta_{[\alpha][\beta]}(x,\chi).
\end{align*}
Thus, \cref{item: composing groupoid actions} is satisfied.

For \cref{item: action of groupoid unit}, fix $(x,\chi) \in X \times \widehat{Z}_\omega$. Then $s(R(x,\chi)) = s([x]) = [x] = R(x,\chi)$, and so $\big(R(x,\chi), (x,\chi)\big) \in \HT \star (X \times \widehat{Z}_\omega)$. Since $x \in \IT$, we have $\twistchar_x^\sigma(Z_\omega) = \{1\}$. Thus, for all $p \in Z_\omega$, we have $\tilde{\twistchar}_{[x]}^\sigma(p) = \twistchar_x^\sigma(p) = 1$, and so $\tilde{\twistchar}_{[x]}^\sigma \, \chi = \chi$. Hence
\[
\theta_{R(x,\chi)}(x,\chi) = \theta_{[x]}\big(s(x), \, \chi\big) = \big(r(x), \, \tilde{\twistchar}_{[x]}^\sigma \, \chi\big) = (x,\chi). \qedhere
\]
\end{proof}

\section{Realising \texorpdfstring{$C^*(\IT,\sigma)$}{C*(IT,sigma)} as an induced algebra}
\label{sec: induced algebra}

In this section we realise the twisted C*-algebra associated to the interior $\IT$ of the isotropy of a Deaconu--Renault groupoid $\GT$ and a continuous $2$-cocycle $\sigma \in Z^2(\GT,\T)$ as an induced algebra. We then describe the ideals of this induced algebra. We begin by introducing a spanning set $\BT$ for $C_c(\GT)$ and then giving a tensor-product decomposition of $C^*(\IT,\sigma)$.

\begin{lemma} \label{lemma: spanning set BT}
Let $(X,T)$ be a minimal rank-$k$ Deaconu--Renault system such that $X$ is second-countable. Let $c\colon \GT \to \Z^k$ be as in \cref{lemma: 1-cocycle c}. Let
\[
\BT \coloneqq \left\{ f \in C_c(\GT) \,:\, \supp(f) \text{ is a bisection contained in } c^{-1}(n), \text{ for some } n \in \Z^k \right\}.
\]
Then $C_c(\GT) = \vecspan \BT$.
\end{lemma}

\begin{proof}
Fix $f \in C_c(\GT)$. Since $\supp(f)$ is compact, there is a finite set $\FF$ of precompact open bisections that cover $\supp(f)$. Since each $U \in \FF$ is precompact, there are only finitely many $n \in \Z^k$ such that $U \cap c^{-1}(n) \ne \varnothing$. Since each $c^{-1}(n) \cap U$ is open, it is a precompact open bisection, so we can assume that $c$ is constant on each $U \in \FF$. Now, as in the proof of \cite[Lemma~9.1.3]{Sims2020}, fix a partition of unity $\{ g_U : U \in \FF \}$ on $\supp(f)$ subordinate to $\FF$. By the Tietze extension theorem, each $g_U$ extends to an element $\tilde{g}_U$ of $C_c(\GT)$. Now the pointwise products $f_U \coloneqq \tilde{g}_U \cdot f$ satisfy $\supp(f_U) \subseteq U$, and $\sum_{U \in \FF} f_U = f$.
\end{proof}

\begin{lemma} \label{lemma: h cdot 1_p span}
Let $(X,T)$ be a minimal rank-$k$ Deaconu--Renault system such that $X$ is second-countable. For each $h \in C_c(X)$ and $p \in \PT$, define $h \cdot 1_p\colon \IT \to \C$ by
\[
(h \cdot 1_p)(x,m,x) \coloneqq \delta_{p,m} \, h(x).
\]
Then $h \cdot 1_p \in C_c(\IT)$ for each $p \in \PT$, and $C_c(\IT) = \vecspan\{ h \cdot 1_p : h \in C_c(X), \, p \in \PT \}$.
\end{lemma}

\begin{proof}
For each $h \in C_c(X)$ and $p \in \PT$, we have
\[
\osupp(h \cdot 1_p) = \big(\!\osupp(h) \times \{p\} \times \osupp(h)\big) \medcap \IT,
\]
and hence $h \cdot 1_p \in C_c(\IT)$. Fix $f \in C_c(\IT)$. Since $\supp(f)$ is compact, there is a finite set $F \subseteq \PT$ such that $\supp(f) \subseteq \bigcup_{p \in F} \, c\restr{\IT}^{-1}(p)$. For $p \in F$, define $h_p\colon X \to \C$ by $h_p(x) \coloneqq f(x,p,x)$. Then $\osupp(h_p) = r\big(c\restr{\IT}^{-1}(p) \medcap \osupp(f)\big)$, and hence $h_p \in C_c(X)$. Moreover, $f = \sum_{p \in F} h_p \cdot 1_p$.
\end{proof}

\begin{prop} \label{prop: tensor prod decomp}
Let $(X,T)$ be a minimal rank-$k$ Deaconu--Renault system such that $X$ is second-countable. Suppose that $\sigma \in Z^2(\GT,\T)$, and that $\omega \in Z^2(\PT,\T)$ is a bicharacter that vanishes on $Z_\omega$ in each coordinate and satisfies $\sigma\bigrestr{\ITc} = 1_X \times \omega$, as in \cref{prop: DR cohomologous vanishing bicharacter}. Let $\{ u_p : p \in \PT \}$ be the canonical family of generating unitaries for the twisted group C*-algebra $C^*(\PT,\omega)$. There is an isomorphism $\Upsilon\colon C^*(\IT,\sigma) \to C_0(X) \otimes C^*(\PT,\omega)$ such that $\Upsilon(h \cdot 1_p) = h \otimes u_p$ for all $h \in C_c(X)$ and $p \in \PT$.
\end{prop}

\begin{proof}
The argument used to prove \cite[Lemma~4.1]{KPS2016JNG} works here---for more detail and an alternative approach to proving injectivity, see \cite[Proposition~8.1.3]{Armstrong2019}.
\end{proof}

Before stating the next theorem, we recall the following facts relating to twisted group C*-algebras. Define $B \coloneqq \PT / Z_\omega$. There is a right action of $\widehat{B}$ on $\PTHat$ such that
\[
(\phi \cdot \chi)(p) = \phi(p) \, \chi(p + Z_\omega) \quad \text{for all } \phi \in \PTHat, \, \chi \in \widehat{B}, \text{ and } p \in \PT.
\]
This action induces a continuous, free, proper, right action of $\widehat{B}$ on $X \times \PTHat$ given by $(x,\phi) \cdot \chi \coloneqq (x, \phi \cdot \chi)$. By \cite[Theorem~4.40]{Folland2016}, the map $\phi \cdot \widehat{B} \mapsto \phi\restr{Z_\omega}$ is an isomorphism $\PTHat / \widehat{B} \cong \widehat{Z}_\omega$. Thus $\phi \mapsto \phi\restr{Z_\omega}$ is a quotient map from $\PTHat$ to $\widehat{Z}_\omega$, and so \cite[Theorem~3.3.17]{Engelking1989} implies that $Q\colon (x,\phi) \mapsto (x,\phi\restr{Z_\omega})$ is a quotient map from $X \times \PTHat$ to $X \times \widehat{Z}_\omega$.

Let $\{ U_{p+Z_\omega} : p+Z_\omega \in B \}$ be the canonical family of generating unitaries for the twisted group C*-algebra $C^*(B,\tilde{\omega})$. By the universal property of $C^*(B,\tilde{\omega})$, there is a strongly continuous action $\beta^B$ of $\widehat{B}$ on $C^*(B,\tilde{\omega})$ such that
\[
\beta_\chi^B(U_{p+Z_\omega}) = \chi(p+Z_\omega) \, U_{p+Z_\omega} \quad \text{for all } \chi \in \widehat{B} \text{ and } p \in \PT.
\]
(See \cite[Theorem~4.3.1]{Armstrong2019} for proofs of the existence of these two actions of $\widehat{B}$.) Recall from \cref{def: induced algebra} the definition of the induced algebra $ \Ind_{\widehat{B}}^{X \times \PTHat}\!\big(C^*(B,\tilde{\omega}), \, \beta^B\big)$ associated to the dynamical system $\big(C^*(B,\tilde{\omega}), \widehat{B}, \beta^B\big)$.

\begin{thm} \label{thm: twisted isotropy induced algebra}
Let $(X,T)$ be a minimal rank-$k$ Deaconu--Renault system such that $X$ is second-countable. Fix $\sigma \in Z^2(\GT,\T)$, and let $\omega \in Z^2(\PT,\T)$ and $\tilde{\omega} \in Z^2(\PT/Z_\omega,\T)$ be bicharacters chosen as in \cref{prop: DR cohomologous vanishing bicharacter}. Define
\begin{align*}
\XX_T^\omega \coloneqq& \Ind_{\widehat{B}}^{X \times \PTHat}\!\big(C^*(B,\tilde{\omega}), \, \beta^B\big) \\
=&\, \Bigg\{ f \in C_0\big(X \times \PTHat, \, C^*(B,\tilde{\omega})\big) \ : \ \begin{tabular}{@{}p{\textwidth-23.5em}@{}}\setstretch{1.35}$f(x, \phi \cdot \chi) = \big(\beta_\chi^B\big)^{-1}\big(f(x,\phi)\big)$ for all $(x,\phi) \in X \times \PTHat$ and $\chi \in \widehat{B}$\end{tabular} \Bigg\}.
\end{align*}
There is an isomorphism $\psi_T\colon C^*(\IT,\sigma) \to \XX_T^\omega$ such that
\[
\psi_T(h \cdot 1_p)(x,\phi) = h(x) \, \overline{\phi(p)} \, U_{p+Z_\omega}
\]
for all $h \in C_c(X)$, $p \in \PT$, and $(x,\phi) \in X \times \PTHat$.
\end{thm}

\begin{proof}
Define $\YY_T^\omega\coloneqq \Ind_{\widehat{B}}^{\PTHat}\!\big(C^*(B,\tilde{\omega}), \, \beta^B\big)$. Recall from \cref{prop: tensor prod decomp,thm: twisted group C* induced algebra} the definitions of the isomorphisms
\[
\Upsilon\colon C^*(\IT,\sigma) \to C_0(X) \otimes C^*(\PT,\omega) \quad \text{ and } \quad \Omega\colon C^*(\PT,\omega) \to \YY_T^\omega.
\]
By \cite[Propositions~B.13~and~B.16]{RW1998}, there is an isomorphism
\[
\Gamma\colon C_0(X) \otimes C^*(\PT,\omega) \to C_0\big(X, \, \YY_T^\omega\big)
\]
such that $\Gamma(f \otimes a)(x) = f(x) \, \Omega(a)$ for all $f \in C_0(X)$, $a \in C^*(\PT,\omega)$, and $x \in X$. Hence
\begin{equation} \label{eqn: Gamma of Upsilon}
\Gamma\big(\Upsilon(h \cdot 1_p)\big)(x) = \Gamma(h \otimes u_p)(x) = h(x) \, \Omega(u_p),
\end{equation}
for all $h \in C_0(X)$, $p \in \PT$, and $x \in X$. Applications of \cite[Propositions~B.13, B.15(b), and B.16, and Corollary~B.17]{RW1998} show that there is an isomorphism
\[
\Lambda\colon C_0\big(X, \, C\big(\PTHat, \, C^*(B,\tilde{\omega})\big)\big) \to C_0\big(X \times \PTHat, \, C^*(B,\tilde{\omega})\big)
\]
given by $\Lambda(g)(x,\phi) = g(x)(\phi)$. (See the proof of \cite[Proposition~8.2.2]{Armstrong2019} for details.) We claim that for each $g \in C_0\big(X, \, C\big(\PTHat, \, C^*(B,\tilde{\omega})\big)\big)$,
\begin{equation} \label[claim]{claim: induced algebras coincide}
\Lambda(g) \in \XX_{T,\omega} \quad \text{ if and only if } \quad g(x) \in \YY_{T,\omega} \text{ for all } x \in X.
\end{equation}
To see this, fix $g \in C_0\big(X, \, C\big(\PTHat, \, C^*(B,\tilde{\omega})\big)\big)$. For all $x \in X$, $\phi \in \PTHat$, and $\chi \in \widehat{B}$, we have
\[
\Lambda(g)(x, \phi \cdot \chi) = g(x)(\phi \cdot \chi) \quad \text{ and } \quad \big(\beta_\chi^B\big)^{-1}\big(\Lambda(g)(x,\phi)\big) = \big(\beta_\chi^B\big)^{-1}\big(g(x)(\phi)\big),
\]
and hence
\[
\Lambda(g)(x, \phi \cdot \chi) = \big(\beta_\chi^B\big)^{-1}\big(\Lambda(g)(x,\phi)\big) \quad \text{ if and only if } \quad g(x)(\phi \cdot \chi) = \big(\beta_\chi^B\big)^{-1}\big(g(x)(\phi)\big).
\]
It is now clear from the definitions of $\XX_{T,\omega}$ and $\YY_{T,\omega}$ that \cref{claim: induced algebras coincide} holds. Therefore, $\Lambda$ restricts to an isomorphism $\widetilde{\Lambda}\colon C_0\big(X, \YY_{T,\omega}\big) \to \XX_{T,\omega}$, and so
\[
\psi_T \coloneqq \widetilde{\Lambda} \circ \Gamma \circ \Upsilon \colon C^*(\IT,\sigma) \to \XX_{T,\omega}
\]
is an isomorphism. Using \cref{eqn: Gamma of Upsilon} and the definitions of $\widetilde{\Lambda}$ and $\Omega$, we see that for all $h \in C_c(X)$, $p \in \PT$, and $(x,\phi) \in X \times \PTHat$,
\[
\psi_T(h \cdot 1_p)(x,\phi) = \Gamma\big(\Upsilon(h \cdot 1_p)\big)(x)(\phi) = h(x) \, \Omega(u_p)(\phi) = h(x) \, \overline{\phi(p)} \, U_{p+Z_\omega}. \qedhere
\]
\end{proof}

We now give a useful description of the ideals of the induced algebra $\XX_T^\omega$.

\begin{prop} \label{prop: ideals of twisted isotropy induced algebra}
Let $(X,T)$ be a minimal rank-$k$ Deaconu--Renault system such that $X$ is second-countable. Fix $\sigma \in Z^2(\GT,\T)$, and let $\omega \in Z^2(\PT,\T)$ and $\tilde{\omega} \in Z^2(\PT/Z_\omega,\T)$ be bicharacters chosen as in \cref{prop: DR cohomologous vanishing bicharacter}. Define $\XX_T^\omega \coloneqq \Ind_{\widehat{B}}^{X \times \PTHat}\!\big(C^*(B,\tilde{\omega}), \, \beta^B\big)$. If $I$ is an ideal of $\XX_T^\omega$, then
\[
K_I \coloneqq \{ (x,\phi) \in X \times \PTHat : f(x,\phi) = 0 \text{ for all } f \in I \}
\]
is a closed subset of $X \times \PTHat$, and we have
\[
I = \{ f \in \XX_T^\omega : f\restr{K_I} \equiv 0 \}.
\]
\end{prop}

In order to prove \cref{prop: ideals of twisted isotropy induced algebra}, we need the following special case of \cite[Proposition~32]{Green1978}, which Green in turn attributes to a preprint of Kleppner.

\begin{lemma} \label{lemma: simple twisted group C*-algebra}
Let $G$ be a countable discrete abelian group with identity $e$, and let $\varsigma \in Z^2(G,\T)$ be a bicharacter. Let $\{ u_g : g \in G \}$ be the canonical family of generating unitaries for the twisted group C*-algebra $C^*(G,\varsigma)$. Suppose that for all $g \in G$, we have $(\varsigma\varsigma^*)\big(\{g\} \times G\big) = \{1\}$ if and only if $g = e$. Then $C^*(G,\varsigma)$ is a simple C*-algebra with a unique trace $\tau_e\colon C^*(G,\varsigma) \to \C$, which satisfies $\tau_e(u_g) = \delta_{g,e}$ for all $g \in G$.
\end{lemma}

\begin{proof}
This is a special case of \cite[Proposition~32]{Green1978}, but is also proved directly in \cite[Proposition~8.2.4]{Armstrong2019}.
\end{proof}

\begin{proof}[Proof of \cref{prop: ideals of twisted isotropy induced algebra}]
We have $K_I = \medcap_{f \in I} \, f^{-1}(0)$, which is closed because each $f \in I$ is continuous. It is well known that if $C^*(B,\tilde{\omega})$ is simple, then $I = \{ f \in \XX_T^\omega : f\restr{K_I} \equiv 0 \}$ (see \cite[Proposition~4.2.1]{Armstrong2019} for a proof). We will use \cref{lemma: simple twisted group C*-algebra} to show that $C^*(B,\tilde{\omega})$ is simple. Fix $p \in \PT$. \Cref{prop: DR cohomologous vanishing bicharacter} implies that for all $q \in \PT$, we have
\begin{equation} \label{eqn: omega omega star}
(\tilde{\omega}\tilde{\omega}^*)(p + Z_\omega, q + Z_\omega) = \omega(p,q) \, \overline{\omega(q,p)} = (\omega\omega^*)(p,q).
\end{equation}
By the definition of $Z_\omega$, we have $p \in Z_\omega$ if and only if $(\omega\omega^*)(p,q) = 1$ for all $q \in \PT$. Thus, \cref{eqn: omega omega star} implies that $p + Z_\omega$ is the identity element of $B$ if and only if $(\tilde{\omega}\tilde{\omega}^*)\big(\{p + Z_\omega\} \times B\big) = \{1\}$, and so \cref{lemma: simple twisted group C*-algebra} implies that $C^*(B,\tilde{\omega})$ is simple.
\end{proof}

\section{Simplicity of twisted C*-algebras of \texorpdfstring{Deaconu--Renault}{Deaconu-Renault} groupoids}
\label{sec: simplicity}

In this section we characterise simplicity of twisted C*-algebras of Deaconu--Renault groupoids in terms of the underlying data, using the spectral action defined in \cref{prop: spectral action}.

\begin{thm} \label{thm: simplicity characterisation}
Let $(X,T)$ be a rank-$k$ Deaconu--Renault system such that $X$ is second-countable. Fix $\sigma \in Z^2(\GT,\T)$.
\begin{enumerate}[label=(\alph*)]
\item \label{item: main simple implies minimal} If $(X,T)$ is not minimal, then $C^*(\GT,\sigma)$ is not simple.
\item \label{item: main minimal implies simple} Suppose that $(X,T)$ is minimal. Let $\omega \in Z^2(\PT,\T)$ and $\tilde{\omega} \in Z^2(\PT/Z_\omega,\T)$ be bicharacters chosen as in \cref{prop: DR cohomologous vanishing bicharacter}. Let $\theta$ be the spectral action associated to $(T,\sigma)$ as in \cref{prop: spectral action}. Then $C^*(\GT,\sigma)$ is simple if and only if $\theta$ is minimal.
\end{enumerate}
\end{thm}

\begin{proof}[Proof of \cref{thm: simplicity characterisation}\cref{item: main simple implies minimal}]
This follows from \cite[Corollary~4.9]{Renault1991} applied to the groupoid dynamical system $\big(\GT, \GT \times_\sigma \T, C_0(\GTo)\big)$, but it is easy to provide a short direct proof. Since $(X,T)$ is not minimal, there exists $x \in X$ such that $\overline{\Orb{x}}$ is a proper closed invariant set. Let $\HH \coloneqq \GT\restr{\overline{\Orb{x}}} = \{ \gamma \in \GT : s(\gamma) \in \overline{\Orb{x}} \}$, and let $\tau$ be the restriction of $\sigma$ to $\HH^{(2)}$. Then the restriction map $f \mapsto f\restr{\HH}$ is a $*$-homomorphism from $C_c(\GT,\sigma)$ to $C^*(\HH,\tau)$, and so it extends to a homomorphism $R\colon C^*(\GT,\sigma) \to C^*(\HH,\tau)$. Since $\ker(R) \medcap C_0(\GTo) = C_0(X {\setminus} \overline{\Orb{x}})$ is neither $\{0\}$ nor all of $C_0(\GTo)$, we see that $\ker(R)$ is a nonzero proper ideal of $C^*(\GT,\sigma)$.
\end{proof}

In order to prove part~\cref{item: main minimal implies simple} of \cref{thm: simplicity characterisation}, we need several preliminary results. Let $\omega \in Z^2(\PT,\T)$ and $\tilde{\omega} \in Z^2(\PT/Z_\omega,\T)$ be bicharacters chosen as in \cref{prop: DR cohomologous vanishing bicharacter}. Define $B \coloneqq \PT/Z_\omega$, and recall from \cref{thm: twisted isotropy induced algebra} the definition of the isomorphism
\[
\psi_T\colon C^*(\IT,\sigma) \to \XX_T^\omega = \Ind_{\widehat{B}}^{X \times \PTHat}\!\big(C^*(B,\tilde{\omega}), \, \beta^B\big).
\]
Let $\iota\colon C^*(\IT,\sigma) \to C^*(\GT,\sigma)$ be the homomorphism of \cite[Proposition~6.1]{Armstrong2022}, so
\[
\iota(f)(\gamma) = \begin{cases}
f(\gamma) & \ \text{if } \gamma \in \IT \\
0 & \ \text{if } \gamma \notin \IT
\end{cases} \quad \text{ for all } f \in C_c(\IT,\sigma) \text{ and } \gamma \in \GT.
\]
Since $\IT$ is amenable (by \cite[Lemma~3.5]{SW2016} and \cite[Proposition~5.1.1]{ADR2000}), $\iota$ is injective by \cite[Proposition~6.1]{Armstrong2022}. Define $M_T^\sigma \coloneqq \iota\big(C^*(\IT,\sigma)\big) \subseteq C^*(\GT,\sigma)$.

We begin by showing that there is a bounded linear map on $M_T^\sigma$ given by conjugation in $C^*(\GT,\sigma)$ by a fixed element of $C_c(\GT,[0,1])$ that is supported on a bisection.

\begin{lemma} \label{lemma: conjugation map}
Let $(X,T)$ be a minimal rank-$k$ Deaconu--Renault system such that $X$ is second-countable, and fix $\sigma \in Z^2(\GT,\T)$. Let $U$ be an open bisection of $\GT$. Suppose that $g \in C_c(\GT,[0,1])$ satisfies $\supp(g) \subseteq U$. For all $f \in C_c(\IT,\sigma)$, we have $g^*\iota(f)g \in \iota\big(C_c(\IT,\sigma)\big)$. There is a linear contraction $\Xi_g \colon M_T^\sigma \to M_T^\sigma$ given by $\Xi_g(a) \coloneqq g^*ag$.
\end{lemma}

\begin{proof}
Fix $f \in C_c(\IT,\sigma)$. Since $U$ is a bisection containing $\supp(g)$, we have
\[
\supp\!\big(g^*\iota(f)g\big) \subseteq U^{-1} \, \IT \, U \subseteq \IT,
\]
and hence $g^*\iota(f)g \in \iota\big(C_c(\IT,\sigma)\big)$. Since $g$ has range in $[0,1]$ and is supported on a bisection, $\lv g \rv = \lv g \rv_\infty \le 1$, and thus
\[
\lv g^* \iota(f) g \rv \le \lv g^* \rv \, \lv \iota(f) \rv \, \lv g \rv \le \lv \iota(f) \rv.
\]
Therefore, $\iota(f) \mapsto g^* \iota(f) g$ extends to a linear contraction $\Xi_g\colon M_T^\sigma \to M_T^\sigma$.
\end{proof}

In the next lemma we introduce a bounded linear map $\Theta_{U,g}$ on the induced algebra $\XX_{T,\omega}$ that is reminiscent of the spectral action $\theta$ associated to the pair $(T,\sigma)$. This map $\Theta_{U,g}$ is defined in terms of a fixed element $g$ of $C_c(\GT,[0,1])$ that is supported on an open bisection $U$ of $\GT$, and as we show in \cref{prop: properties of Theta}\cref{item: Theta is conjugation}, it simply amounts to conjugation of elements of $M_T^\sigma \cong \XX_{T,\omega}$ by $g$.

\begin{lemma} \label{lemma: Theta is a BLM}
Let $(X,T)$ be a minimal rank-$k$ Deaconu--Renault system such that $X$ is second-countable, and fix $\sigma \in Z^2(\GT,\T)$. Let $U$ be an open bisection of $\GT$. Suppose that $g \in C_c(\GT,\sigma)$ satisfies $\supp(g) \subseteq U$ and that $g(U) \subseteq [0,1]$. For each $x \in s(U)$, let $\alpha_{U,x}$ denote the unique element of $U$ with source $x$. Let $\gamma \mapsto \twistchar_\gamma^\sigma$ be the continuous $\PTHat$-valued $1$-cocycle of \cref{lemma: twistchar_gamma^sigma}\cref{item: twistchar_.^sigma 1-cocycle}. For $f \in \XX_T^\omega$, define $\Theta_{U,g}(f)\colon X \times \PTHat \to C^*(B,\tilde{\omega})$ by
\[
\Theta_{U,g}(f)(x,\phi) \coloneqq
\begin{cases}
\lav g(\alpha_{U,x}) \rav^2 \, f\big(r(\alpha_{U,x}),\,\twistchar_{\alpha_{U,x}}^\sigma\,\phi\big) &\quad \text{if } x \in s(U) \\
0 &\quad \text{if } x \notin s(U).
\end{cases}
\]
Then $\Theta_{U,g}(f) \in \XX_T^\omega$, and $\Theta_{U,g}\colon \XX_T^\omega \to \XX_T^\omega$ is a bounded linear map.
\end{lemma}

\begin{proof}
Fix $f \in \XX_T^\omega$. Then
\begin{equation} \label{eqn: induced algebra equation}
f(x, \phi \cdot \chi) = \big(\beta_\chi^B\big)^{-1}\big(f(x,\phi)\big) \ \text{ for all } (x,\phi) \in X \times \PTHat \text{ and } \chi \in \widehat{B}.
\end{equation}
We first show that $\Theta_{U,g}(f) \in C_0\big(X \times \PTHat, \, C^*(B,\tilde{\omega})\big)$. The map $\Theta_{U,g}(f)$ is continuous because $x \mapsto \alpha_{U, x}$ is continuous. We have $\supp(\Theta_{U,g}(f)) \subseteq s(\supp(g)) \times \PTHat$, and so $\Theta_{U,g}(f)$ has compact support. Hence $\Theta_{U,g}(f) \in C_0\big(X \times \PTHat, \, C^*(B,\tilde{\omega})\big)$. We must show that $\Theta_{U,g}(f)$ satisfies \cref{eqn: induced algebra equation}. Fix $(x,\phi) \in X \times \PTHat$ and $\chi \in \widehat{B}$. If $x \notin s(U)$, then
\[
\Theta_{U,g}(f)(x, \phi \cdot \chi) = 0 = (\beta_\chi^B)^{-1}(0) = \big(\beta_\chi^B\big)^{-1}\big(\Theta_{U,g}(f)(x,\phi)\big).
\]
Suppose that $x \in s(U)$. Since $f \in \XX_T^\omega$, \cref{eqn: induced algebra equation} implies that
\begin{align*}
\Theta_{U,g}(f)(x, \phi \cdot \chi) &= \lav g(\alpha_{U,x}) \rav^2 \, f\big(r(\alpha_{U,x}),\,(\twistchar_{\alpha_{U,x}}^\sigma\,\phi) \cdot \chi\big) \\
&= \lav g(\alpha_{U,x}) \rav^2 \, \big(\beta_\chi^B\big)^{-1}\big(f\big(r(\alpha_{U,x}),\,\twistchar_{\alpha_{U,x}}^\sigma\,\phi\big)\big)
= \big(\beta_\chi^B\big)^{-1}\big(\Theta_{U,g}(f)(x,\phi)\big).
\end{align*}
Therefore, $\Theta_{U,g}(f) \in \XX_T^\omega$. Since the range of $g$ is contained in $[0,1]$, routine calculations show that $\Theta_{U,g}\colon \XX_T^\omega \to \XX_T^\omega$ is a bounded linear map.
\end{proof}

In the next lemma we show that the set of functions of the form $\iota(h \cdot 1_p)$ (as defined in \cref{lemma: h cdot 1_p span}) is invariant under conjugation in $C^*(\GT,\sigma)$ by a fixed element of $C_c(\GT,[0,1])$ that is supported on a bisection.

\begin{lemma} \label{lemma: conjugating h.1_p}
Let $(X,T)$ be a minimal rank-$k$ Deaconu--Renault system such that $X$ is second-countable, and fix $\sigma \in Z^2(\GT,\T)$. Let $U$ be an open bisection of $\GT$. Suppose that $g \in C_c(\GT,\sigma)$ satisfies $\supp(g) \subseteq U$ and that $g(U) \subseteq [0,1]$. For each $x \in s(U)$, let $\alpha_{U,x}$ denote the unique element of $U$ with source $x$. Let $\Xi_g\colon M_T^\sigma \to M_T^\sigma$ and $\gamma \mapsto \twistchar_\gamma^\sigma$ be as in \cref{lemma: conjugation map,lemma: twistchar_gamma^sigma}\cref{item: twistchar_.^sigma 1-cocycle}. For each $x \in s(U)$, let $\alpha_{U,x}$ denote the unique element of $U$ with source $x$. Fix $h \in C_c(X)$ and $p \in \PT$, and define $H_{g,p}\colon X \to \C$ by
\[
H_{g,p}(x) \coloneqq
\begin{cases}
\lav g(\alpha_{U,x}) \rav^2 \, \overline{\twistchar_{\alpha_{U,x}}^\sigma(p)} \, h(r(\alpha_{U,x})) &\quad \text{if } x \in s(U) \\
0 &\quad \text{if } x \notin s(U).
\end{cases}
\]
Then $H_{g,p} \in C_c(X)$, and we have $\,\Xi_g(\iota(h \cdot 1_p)) = \iota(H_{g,p} \cdot 1_p)$.
\end{lemma}

\begin{proof}
Since $x \mapsto \alpha_{U,x}$ is continuous on $s(U)$, and since $\gamma \mapsto \twistchar_\gamma^\sigma(p)$ is continuous by \cref{lemma: twistchar_gamma^sigma}\cref{item: twistchar_.^sigma cts}, the map $H_{g,p}$ is continuous. Since $\supp(H_{g,p}) \subseteq s(\supp(g))$, we have $H_{g,p} \in C_c(X)$.

By \cref{lemma: conjugation map}, we have $\Xi_g(\iota(h \cdot 1_p)) \in \iota\big(C_c(\IT,\sigma)\big)$. Thus, for all $\gamma \in \GT {\setminus} \IT$, we have
\[
\Xi_g(\iota(h \cdot 1_p))(\gamma) = 0 = \iota(H_{g,p} \cdot 1_p)(\gamma).
\]
Suppose that $\gamma \in \IT$. Then by \cref{prop: IT characterisation}, there exist $x \in X$ and $m \in \PT$ such that $\gamma = (x,m,x)$. We have
\[
\supp\!\big(\Xi_g(\iota(h \cdot 1_p))\big) \subseteq \supp(g^*) \, \supp(\iota(h \cdot 1_p)) \, \supp(g) \subseteq U^{-1} \, \IT \, U.
\]
Thus, if $x \notin s(U)$, then $\gamma \notin \supp\!\big(\Xi_g(\iota(h \cdot 1_p))\big)$ and $H_{g,p}(x) = 0$, and hence
\[
\Xi_g(\iota(h \cdot 1_p))(x,m,x) = 0 = \iota(H_{g,p} \cdot 1_p)(x,m,x).
\]
Suppose that $x \in s(U)$. Since $g$ is supported on the bisection $U$,
\begin{align*}
\Xi_g(\iota(h \cdot 1_p))(x,m,x) \,&=\, \twistchar_{(\alpha_{U,x})^{-1}}^\sigma(m) \ g^*(\alpha_{U,x}) \ \iota(h \cdot 1_p)\big(r(\alpha_{U,x}),m,r(\alpha_{U,x})\big) \ g(\alpha_{U,x}) \\
&=\, \lav g(\alpha_{U,x}) \rav^2 \, \overline{\twistchar_{\alpha_{U,x}}^\sigma(m)} \, \delta_{p,m} \, h(r(\alpha_{U,x})) \\
&=\, \delta_{p,m} \, H_{g,p}(x) \\
&=\, \iota(H_{g,p} \cdot 1_p)(x,m,x).
\end{align*}
Therefore, $\,\Xi_g(\iota(h \cdot 1_p)) = \iota(H_{g,p} \cdot 1_p)$.
\end{proof}

In the following proposition we describe exactly how the map $\Theta_{U,g}$ defined in \cref{lemma: Theta is a BLM}. relates to the conjugation map $\Xi_g$ defined in \cref{lemma: conjugation map}. We also show that ideals of $\XX_{T,\omega}$ induced by ideals of $C^*(\GT,\sigma)$ are invariant under $\Theta_{U,g}$, which is a key result used in the proof of \cref{thm: simplicity characterisation}\cref{item: main minimal implies simple}.

\begin{prop} \label{prop: properties of Theta}
Let $(X,T)$ be a minimal rank-$k$ Deaconu--Renault system such that $X$ is second-countable, and fix $\sigma \in Z^2(\GT,\T)$. Let $U$ be an open bisection of $\GT$. Suppose that $g \in C_c(\GT,[0,1])$ satisfies $\supp(g) \subseteq U$. Recall the definitions of the bounded linear maps $\Xi_g\colon M_T^\sigma \to M_T^\sigma$ from \cref{lemma: conjugation map} and $\Theta_{U,g}\colon \XX_T^\omega \to \XX_T^\omega$ from \cref{lemma: Theta is a BLM}.
\begin{enumerate}[label=(\alph*)]
\item \label{item: Theta is conjugation} For all $a \in C^*(\IT,\sigma)$, we have
\[
\Theta_{U,g}(\psi_T(a)) = \psi_T\big(\iota^{-1}\big(\Xi_g(\iota(a))\big)\big).
\]
\item \label{item: induced algebra ideals are invariant under Theta} Suppose that $I$ is an ideal of $C^*(\GT,\sigma)$, and that $J$ is an ideal of $C^*(\IT,\sigma)$ such that $\iota(J) = I \medcap M_T^\sigma$. Then the ideal $\psi_T(J)$ is invariant under $\Theta_{U,g}$.
\end{enumerate}
\end{prop}

\begin{proof}
For part~\cref{item: Theta is conjugation}, fix $h \in C_c(X)$ and $p \in \PT$. Since all the maps involved are bounded and linear, \cref{lemma: h cdot 1_p span} implies that it suffices to show that
\[
\Theta_{U,g}(\psi_T(h \cdot 1_p)) = \psi_T\big(\iota^{-1}\big(\Xi_g(\iota(h \cdot 1_p))\big)\big).
\]
Recall from \cref{lemma: conjugating h.1_p} that there is a function $H_{g,p} \in C_c(X)$ given by
\[
H_{g,p}(x) =
\begin{cases}
\lav g(\alpha_{U,x}) \rav^2 \, \overline{\twistchar_{\alpha_{U,x}}^\sigma(p)} \, h(r(\alpha_{U,x})) &\quad \text{if } x \in s(U) \\
0 &\quad \text{if } x \notin s(U),
\end{cases}
\]
which satisfies $\,\Xi_g(\iota(h \cdot 1_p)) = \iota(H_{g,p} \cdot 1_p)$. Thus, for all $(x,\phi) \in X \times \PT$, we have
\begin{align*}
\Theta_{U,g}(\psi_T(h \cdot 1_p))(x,\phi) \,&=\, \begin{cases}
\lav g(\alpha_{U,x}) \rav^2 \, \psi_T(h \cdot 1_p)\big(r(\alpha_{U,x}),\,\twistchar_{\alpha_{U,x}}^\sigma\,\phi\big) &\quad \text{if } x \in s(U) \\
0 &\quad \text{if } x \notin s(U)
\end{cases} \\
&= \, \begin{cases}
\lav g(\alpha_{U,x}) \rav^2 \, h(r(\alpha_{U,x})) \, \overline{\twistchar_{\alpha_{U,x}}^\sigma(p)} \, \overline{\phi(p)} \, U_{p+Z_\omega} &\quad \text{if } x \in s(U) \\
0 &\quad \text{if } x \notin s(U)
\end{cases} \\
&=\, H_{g,p}(x) \, \overline{\phi(p)} \, U_{p+Z_\omega} \\
&=\, \psi_T(H_{g,p} \cdot 1_p)(x,\phi) \\
&=\, \psi_T\big(\iota^{-1}\big(\Xi_g(\iota(h \cdot 1_p))\big)\big)(x,\phi).
\end{align*}

For part~\cref{item: induced algebra ideals are invariant under Theta}, fix $a \in J$. Then $\iota(a) \in I \medcap M_T^\sigma$. Since $I$ is an ideal of $C^*(\GT,\sigma)$ and the range of $\Xi_g$ is contained in $M_T^\sigma$, we have $\Xi_g(\iota(a)) = g^* \iota(a) g \in I \medcap M_T^\sigma = \iota(J)$, and so $\iota^{-1}\big(\Xi_g(\iota(a))\big) \in J$. Hence part~\cref{item: Theta is conjugation} implies that
\[
\Theta_{U,g}(\psi_T(a)) = \psi_T\big(\iota^{-1}\big(\Xi_g(\iota(a))\big)\big) \in \psi_T(J),
\]
and thus $\Theta_{U,g}(\psi_T(J)) \subseteq \psi_T(J)$.
\end{proof}

We now use \cref{prop: properties of Theta}\cref{item: induced algebra ideals are invariant under Theta} to show that the closed subsets of $X \times \PTHat$ characterising the ideals of the induced algebra $\XX_{T,\omega}$ are invariant under the spectral action $\theta$ associated to the pair $(T,\sigma)$.

\begin{prop} \label{prop: Q(K_J) is closed and invariant under theta}
Let $(X,T)$ be a minimal rank-$k$ Deaconu--Renault system such that $X$ is second-countable, and fix $\sigma \in Z^2(\GT,\T)$. Suppose that $I$ is an ideal of $C^*(\GT,\sigma)$, and that $J$ is an ideal of $C^*(\IT,\sigma)$ such that $\iota(J) = I \medcap M_T^\sigma$. Define
\[
K_J \coloneqq \{ (x,\phi) \in X \times \PTHat \,:\, f(x,\phi) = 0 \text{ for all } f \in \psi_T(J) \}.
\]
Let $Q\colon X \times \PTHat \to X \times \widehat{Z}_\omega$ be the quotient map $(x,\phi) \mapsto (x,\phi\restr{Z_\omega})$. Then $Q^{-1}(Q(K_J)) = K_J$, and $Q(K_J)$ is closed and invariant under the spectral action $\theta$ of \cref{prop: spectral action}.
\end{prop}

\begin{proof}
We trivially have $K_J \subseteq Q^{-1}\!\left(Q(K_J)\right)$. We must show that $Q^{-1}\!\left(Q(K_J)\right) \subseteq K_J$. Fix $(x,\phi) \in Q^{-1}\!\left(Q(K_J)\right)$. Then $(x,\phi\restr{Z_\omega}) = Q(x,\phi) \in Q(K_J)$, and so there exists $(y,\rho) \in K_J$ such that $(x,\phi\restr{Z_\omega}) = Q(y,\rho) = (y,\rho\restr{Z_\omega})$. We have $x = y$ and $\phi\restr{Z_\omega} = \rho\restr{Z_\omega}$, and hence \cite[Theorem~4.40]{Folland2016} implies that $\phi \cdot \widehat{B} = \rho \cdot \widehat{B}$. So there exists $\chi \in \widehat{B}$ such that $\phi \cdot \chi = \rho \cdot 1_{\widehat{B}} = \rho$. Since $(x, \phi \cdot \chi) = (y,\rho) \in K_J$, we have $f(x, \phi \cdot \chi) = 0$ for all $f \in \psi_T(J)$. Thus, since $\psi_T(J) \subseteq \XX_T^\omega$, we have $f(x,\phi) = \beta_\chi^B\!\left(f(x, \phi \cdot \chi)\right) = 0$ for all $f \in \psi_T(J)$. Hence $(x,\phi) \in K_J$, and so $Q^{-1}\!\left(Q(K_J)\right) = K_J$. Since $Q$ is a quotient map, \cite[Proposition~2.4.3]{Engelking1989} implies that $C \subseteq X \times \widehat{Z}_\omega$ is closed if and only if $Q^{-1}(C) \subseteq X \times \PTHat$ is closed. Since $Q^{-1}(Q(K_J)) = K_J$ is closed in $X \times \PTHat$ (by \cref{prop: ideals of twisted isotropy induced algebra}), we deduce that $Q(K_J)$ is closed.

We now show that $Q(K_J)$ is invariant under $\theta$. Fix $(x,\zeta) \in Q(K_J)$ and $\gamma \in (\GT)_x$. Then there exists $\phi \in \PTHat$ such that $(x,\phi) \in K_J$ and $\phi\restr{Z_\omega} = \zeta$. We must show that $\theta_{[\gamma]}(x,\zeta) \in Q(K_J)$. \Cref{prop: spectral action} implies that $\twistchar_\gamma^\sigma\restr{Z_\omega} = \tilde{\twistchar}_{[\gamma]}^\sigma$, and so
\[
\theta_{[\gamma]}(x,\zeta) = \big(r(\gamma),\, \tilde{\twistchar}_{[\gamma]}^\sigma \, \zeta \big) = Q\!\left(r(\gamma),\, \twistchar_\gamma^\sigma \, \phi\right)\!.
\]
Hence it suffices to show that $\left(r(\gamma),\, \twistchar_\gamma^\sigma \, \phi\right) \in K_J$. Fix $f \in \psi_T(J)$. We must show that $f\!\left(r(\gamma),\, \twistchar_\gamma^\sigma \, \phi\right) = 0$. Let $U \subseteq \GT$ be an open bisection containing $\gamma$. By Urysohn's lemma there exists $g \in C_c(\GT,[0,1])$ such that $\supp(g) \subseteq U$ and $g(\gamma) = 1$. Let $\Theta_{U,g}\colon \XX_T^\omega \to \XX_T^\omega$ be as in \cref{lemma: Theta is a BLM}. Since $s\restr{U}^{-1}(x) = \gamma$ and $g(\gamma) = 1$,
\begin{equation} \label{eqn: theta is Theta}
\Theta_{U,g}(f)(x,\phi) = f\!\left(r(\gamma), \, \twistchar_\gamma^\sigma \, \phi\right)\!.
\end{equation}
Since $f \in \psi_T(J)$, \cref{prop: properties of Theta}\cref{item: induced algebra ideals are invariant under Theta} implies that $\Theta_{U,g}(f) \in \psi_T(J)$. Since $\psi_T(J)$ is an ideal of $\XX_T^\omega$, \cref{prop: ideals of twisted isotropy induced algebra} implies that
\[
\psi_T(J) = \{ f \in \XX_T^\omega : f\restr{K_J} \equiv 0 \}.
\]
Thus, since $(x,\phi) \in K_J$ and $\Theta_{U,g}(f) \in \psi_T(J)$, we have
\begin{equation} \label{eqn: Theta x phi zero}
\Theta_{U,g}(f)(x,\phi) = 0.
\end{equation}
Together, \cref{eqn: theta is Theta,eqn: Theta x phi zero} imply that $f\!\left(r(\gamma), \, \twistchar_\gamma^\sigma \, \phi\right) = 0$, as required.
\end{proof}

We now prove several technical results that we use in the proof of \cref{thm: simplicity characterisation}\cref{item: main minimal implies simple} to show that when the spectral action $\theta$ is not minimal, the twisted groupoid C*-algebra $C^*(\GT,\sigma)$ is not simple. We first show that, given an element $(x,\phi) \in X \times \PTHat$ with non-dense orbit under $\theta$, there is a nonzero element of $\XX_{T,\omega} \cong M_T^\sigma$ that is supported off the orbit of $(x,\phi)$.

\begin{lemma} \label{lemma: supp outside orbit}
Let $(X,T)$ be a minimal rank-$k$ Deaconu--Renault system such that $X$ is second-countable, and fix $\sigma \in Z^2(\GT,\T)$. Let $Q\colon X \times \PTHat \to X \times \widehat{Z}_\omega$ be the quotient map $(x,\phi) \mapsto (x,\phi\restr{Z_\omega})$. Suppose that $(x,\phi) \in X \times \PTHat$ satisfies $\overline{[x, \phi\restr{Z_\omega}]_\theta} \ne X \times \widehat{Z}_\omega$. Then $Q^{-1}\big(\overline{[x, \phi\restr{Z_\omega}]_\theta}\big)$ is a proper closed subset of $X \times \PTHat$, and there exists $f \in M_T^\sigma {\setminus} \{0\}$ such that
\[
\supp\!\big((\psi_T \circ \iota^{-1})(f)\big) \subseteq (X \times \PTHat) \setminus Q^{-1}\big(\overline{[x, \phi\restr{Z_\omega}]_\theta}\big).
\]
\end{lemma}

\begin{proof}
Let $C_{(x,\phi)} \coloneqq Q^{-1}\big(\overline{[x, \phi\restr{Z_\omega}]_\theta}\big)$. Since $\overline{[x, \phi\restr{Z_\omega}]_\theta} \ne X \times \widehat{Z}_\omega$ and $Q$ is surjective, $C_{(x,\phi)} \ne X \times \PTHat$. Since $Q$ is continuous, $C_{(x,\phi)}$ is closed. By Urysohn's lemma there exists $h \in C_c\big(X \times \PTHat, [0,1]\big) \setminus \{0\}$ such that $\supp(h) \subseteq (X \times \PTHat) \setminus C_{(x,\phi)}$. Define $g\colon X \times \PTHat \to C^*(B,\tilde{\omega})$ by
\[
g(y,\rho) \,\coloneqq\, \int_{\widehat{B}} \, h(y,\rho\cdot\chi) \, \beta^B_\chi(U_{0+Z_\omega}) \, \d\chi \,=\, \int_{\widehat{B}} \, h(y,\rho\cdot\chi) \, U_{0+Z_\omega} \, \d\chi.
\]
By \cite[Lemma~6.17]{RW1998}, we have $g \in \XX_T^\omega$. Since $h \ne 0$ and $h(y,\rho) \ge 0$ for all $(y,\rho) \in \supp(h)$, we have $g \ne 0$. We claim that $\supp(g) \subseteq (X \times \PTHat) \setminus C_{(x,\phi)}$. Fix $(y,\rho) \in C_{(x,\phi)}$. Then $Q(y,\rho) \in \overline{[x, \phi\restr{Z_\omega}]_\theta}$. It suffices to show that $g(y,\rho) = 0$. Fix $\chi \in \widehat{B}$. For all $m \in Z_\omega$, we have $\chi(m + Z_\omega) = 1$, and hence $(\rho \cdot \chi)(m) = \rho(m) \, \chi(m+Z_\omega) = \rho(m)$. Thus
\[
Q(y, \rho \cdot \chi) = \big(y, \, (\rho \cdot \chi)\restr{Z_\omega}\big) = \big(y, \rho\restr{Z_\omega}\big) = Q(y,\rho) \in \overline{[x, \phi\restr{Z_\omega}]_\theta},
\]
and hence $(y, \rho \cdot \chi) \in C_{(x,\phi)}$. Since $\supp(h) \subseteq (X \times \PTHat) \setminus C_{(x,\phi)}$, we have $h(y, \rho \cdot \chi) = 0$ for all $\chi \in \widehat{B}$, and therefore,
\[
g(y,\rho) = \int_{\widehat{B}} h(y,\rho\cdot\chi) \, U_{0+Z_\omega} \, d\chi = 0.
\]
Define $f \coloneqq (\iota \circ \psi_T^{-1})(g) \in M_T^\sigma$. Since $g \ne 0$ and $\iota \circ \psi_T^{-1}$ is injective, we have $f \ne 0$. Since $(\psi_T \circ \iota^{-1})(f) = g$, we have
\[
\supp\!\big((\psi_T \circ \iota^{-1})(f)\big) = \supp(g) \subseteq (X \times \PTHat) \setminus C_{(x,\phi)}. \qedhere
\]
\end{proof}

Recall from \cite[Lemma~6.2(b)]{Armstrong2022} that since $\IT$ is closed in $\GT$ (by \cref{prop: HT quotient groupoid}) and amenable, there is a conditional expectation $\Phi\colon C^*(\GT,\sigma) \to M_T^\sigma$ satisfying $\Phi \circ \iota = \iota$ and $\Phi(f) = \iota(f\restr{\IT})$ for all $f \in C_c(\GT,\sigma)$.

\begin{lemma} \label{lemma: conditional expectation and conjugation}
Let $(X,T)$ be a minimal rank-$k$ Deaconu--Renault system such that $X$ is second-countable, and fix $\sigma \in Z^2(\GT,\T)$. Recall from \cref{lemma: spanning set BT} the definition of the spanning set $\BT$ for $C_c(\GT,\sigma)$. Given $a, b \in \BT$ and $f \in M_T^\sigma$, there exist $p, q, g \in \BT$ such that $gq^*, pg^* \in \iota\big(C_c(\IT,\sigma)\big)$, the range of $g$ is contained in $[0,1]$, and the map $\Xi_g$ of \cref{lemma: conjugation map} satisfies $\Phi(b^*fa) = \Xi_g(gq^*fpg^*)$.
\end{lemma}

\begin{proof}
Define $U \coloneqq \osupp(a)$ and $V \coloneqq \osupp(b)$. Since $a, b \in \BT$, both $\overline{U}$ and $\overline{V}$ are compact bisections, and there exist $m, n \in \Z^k$ such that $\overline{U} \subseteq c^{-1}(m)$ and $\overline{V} \subseteq c^{-1}(n)$. Define $X
\coloneqq \IT V \medcap U$ and $Y \coloneqq \IT U \medcap V$. Define $p,q\colon \GT \to \C$ by
\[
p(\gamma) \coloneqq \begin{cases}
a(\gamma) & \text{if } \gamma \in \overline{X} \\
0 & \text{if } \gamma \notin \overline{X},
\end{cases}
\qquad \text{and} \qquad
q(\gamma) \coloneqq \begin{cases}
b(\gamma) & \text{if } \gamma \in \overline{Y} \\
0 & \text{if } \gamma \notin \overline{Y}.
\end{cases}
\]
Since $\supp(p) \subseteq \overline{X} \subseteq \overline{U}$ and $\supp(q) \subseteq \overline{Y} \subseteq \overline{V}$, we have $p, q \in \BT$. Let $W$ be an open bisection of $\GT$ such that $\overline{V} \subseteq W \subseteq c^{-1}(n)$. By Urysohn's lemma there exists $g \in C_c(\GT,[0,1])$ such that $\supp(g) \subseteq W$ and $g\restr{\overline{V}} \equiv 1$. Then $g \in \BT$, and $gq^*, pg^* \in C_c(\GT,\sigma)$. We claim that $gq^*, pg^* \in \iota\big(C_c(\IT,\sigma)\big)$. To see this, it suffices to show that $\osupp(gq^*) \medcup \osupp(pg^*) \subseteq \IT$. Since $q$ and $g$ are
supported on bisections, we have
\[
\osupp(gq^*) = \osupp(g) (\osupp(q))^{-1} \subseteq W V^{-1} \subseteq W W^{-1} = r(W) \subseteq \IT.
\]
By \cref{lemma: IT U/V bar}, $\overline{X} = \overline{\IT V \medcap U} \subseteq \IT \overline{V} \subseteq \IT W$, and since $p$ and $g$ are supported on bisections, we deduce that
\[
\osupp(pg^*) = \osupp(p) (\osupp(g))^{-1} \subseteq \overline{X} W^{-1} \subseteq \IT W W^{-1} = \IT \, r(W) \subseteq \IT.
\]
Therefore, $gq^*, pg^* \in \iota\big(C_c(\IT,\sigma)\big)$, and \cref{lemma: conjugation map} implies that
\begin{equation} \label{eqn: osupp Xi_g IT}
\osupp\!\big(\Xi_g(g q^* f p g^*)\big) \subseteq \IT.
\end{equation}
We conclude by showing that $\Phi(b^*fa) = \Xi_g(gq^*fpg^*)$. Since $\iota$, $\Phi$, and $\Xi_g$ are bounded linear maps, \cref{lemma: h cdot 1_p span} implies that it suffices to consider $f = \iota(h \cdot 1_p)$ for some $h \in C_c(X)$ and $p \in \PT$. Define $D \coloneqq \osupp(f) \subseteq c\restr{\IT}^{-1}(p)$. Then $\osupp(g^*g) \subseteq s(W)$, and so
\begin{equation} \label{eqn: osupp(Xi_g(gq^*fpg^*))}
\osupp\!\big(\Xi_g(g q^* f p g^*)\big) = \osupp(g^* g q^* f p g^* g) \subseteq s(W) (V^{-1} D U) s(W) \subseteq V^{-1} D U.
\end{equation}
Together, \cref{eqn: osupp Xi_g IT,eqn: osupp(Xi_g(gq^*fpg^*))} imply that
\[
\osupp\!\big(\Xi_g(g q^* f p g^*)\big) \subseteq (V^{-1} D U) \medcap \IT = \osupp(b^*fa) \medcap \IT = \osupp\!\big(\Phi(b^*fa)\big).
\]
Thus, if $\Phi(b^*fa)(\gamma) = 0$ for some $\gamma \in \GT$, then $\Xi_g(gq^*fpg^*)(\gamma) = 0$. Suppose that $\gamma \in \GT$ satisfies $\Phi(b^*fa)(\gamma) \ne 0$. Then $\gamma \in \IT$, and \cref{eqn: osupp(Xi_g(gq^*fpg^*))} implies that there exist $\alpha \in U$, $\beta \in V$, and $\xi \in D \subseteq \IT$ such that $\gamma = \beta^{-1} \xi \alpha \in \IT$. A routine calculation gives
\begin{equation} \label{eqn: Phi(b^*fa) expression}
\Phi(b^*fa)(\gamma) = (b^*fa)(\beta^{-1}\xi\alpha) = \sigma(\beta^{-1}\xi,\alpha) \, \sigma(\beta^{-1},\xi) \, \overline{\sigma(\beta^{-1},\beta)} \, \overline{b(\beta)} \, f(\xi) \, a(\alpha).
\end{equation}
Define $y \coloneqq s(\gamma)$. Since $\gamma \in \IT$, we have $s(\beta) = r(\gamma) = y$. Since $\beta \in V$ and $g\restr{\overline{V}} \equiv 1$,
\begin{equation} \label{eqn: g^*g is 1}
(g^*g)(y) = (g^*g)(s(\beta)) = \lav g(\beta) \rav^2 = 1.
\end{equation}
A routine calculation using \cref{eqn: g^*g is 1} and that $\sigma$ is normalised gives
\begin{align*}
\Xi_g(gq^*fpg^*)(\gamma) &= (g^* g q^* f p g^* g)(y \gamma y) \\
&= \sigma(y, \gamma y) \, \sigma(\gamma, y) \, (g^*g)(y) \, (q^* f p)(\gamma) \, (g^*g)(y) \\
&= (q^*fp)(\beta^{-1} \xi \alpha) \\
&= \sigma(\beta^{-1}\xi,\alpha) \, \sigma(\beta^{-1},\xi) \, \overline{\sigma(\beta^{-1},\beta)} \, \overline{q(\beta)} \, f(\xi) \, p(\alpha). \numberthis \label{eqn: Xi_g(g q^* f p g^*) expression}
\end{align*}
We claim that $p(\alpha) = a(\alpha)$ and $q(\beta) = b(\beta)$. Since $\gamma \in (V^{-1} \IT U) \medcap \IT$, \cref{lemma: s(gamma) in IT U cap V} implies that
\[
y = s(\gamma) \in s(\IT U \medcap V) = s(\IT V \medcap U).
\]
So there exist $\eta \in \IT V \medcap U = X \subseteq U$ and $\zeta \in \IT U \medcap V = Y \subseteq V$ such that $s(\eta) = y = s(\zeta)$. Since $s\restr{U}$ and $s\restr{V}$ are homeomorphisms onto their ranges and
\[
s(\eta) = s(\alpha) = y = s(\beta) = s(\zeta),
\]
we deduce that $\alpha = \eta \in X$ and $\beta = \zeta \in Y$. Hence $p(\alpha) = a(\alpha)$ and $q(\beta) = b(\beta)$. Together, \cref{eqn: Phi(b^*fa) expression,eqn: Xi_g(g q^* f p g^*) expression} now give
\[
\Phi(b^*fa)(\gamma) = \sigma(\beta^{-1}\xi,\alpha) \, \sigma(\beta^{-1},\xi) \, \overline{\sigma(\beta^{-1},\beta)} \, \overline{b(\beta)} \, f(\xi) \, a(\alpha) = \Xi_g(gq^*fpg^*)(\gamma). \qedhere
\]
\end{proof}

\begin{prop} \label{prop: conjugation supp outside orbit}
Let $(X,T)$ be a minimal rank-$k$ Deaconu--Renault system such that $X$ is second-countable, and fix $\sigma \in Z^2(\GT,\T)$. Let $Q\colon X \times \PTHat \to X \times \widehat{Z}_\omega$ be the quotient map $(x,\phi) \mapsto (x,\phi\restr{Z_\omega})$, and let $\Phi\colon C^*(\GT,\sigma) \to M_T^\sigma$ be the conditional expectation of \cite[Lemma~6.2(b)]{Armstrong2022} that extends restriction of functions to $\IT$. Fix $(x,\phi) \in X \times \PTHat$. Suppose that $f \in M_T^\sigma$ satisfies
\[
\supp\!\big((\psi_T \circ \iota^{-1})(f)\big) \subseteq (X \times \PTHat) \setminus Q^{-1}\big(\overline{[x, \phi\restr{Z_\omega}]_\theta}\big).
\]
Then for all $a, b \in C^*(\GT,\sigma)$, we have
\[
\big(\psi_T \circ \iota^{-1} \circ \Phi \big)(b^*fa)(x,\phi) = 0.
\]
\end{prop}

\begin{proof}
Let $\ev_{(x,\phi)}\colon \XX_T^\omega \to C^*(B,\tilde{\omega})$ denote the evaluation map $f \mapsto f(x,\phi)$. Recall from \cref{lemma: spanning set BT} the definition of the spanning set $\BT$ for $C_c(\GT,\sigma)$. Let $C_{(x,\phi)} \coloneqq Q^{-1}\big(\overline{[x, \phi\restr{Z_\omega}]_\theta}\big)$. Fix $a, b \in \BT$, and suppose that $f \in M_T^\sigma$ satisfies
\[
\supp\!\big((\psi_T \circ \iota^{-1})(f)\big) \subseteq (X \times \PTHat) \setminus C_{(x,\phi)}.
\]
Since $\ev_{(x,\phi)}$, $\psi_T$, $\iota^{-1}$, and $\Phi$ are all bounded linear maps, it suffices to show that
\[
\big(\ev_{(x,\phi)} \circ \, \psi_T \circ \iota^{-1} \circ \Phi \big)(b^*fa) = 0.
\]
Let $\Xi_g$ be the bounded linear map defined in \cref{lemma: conjugation map}. By \cref{lemma: conditional expectation and conjugation} there exist $p, q, g \in \BT$ such that $gq^*, pg^* \in \iota\big(C_c(\IT,\sigma)\big)$, the range of $g$ is contained in $[0,1]$, and
\begin{equation} \label{eqn: Phi(b^*fa) = Xi_g(g q^* f p g^*)}
\Phi(b^*fa) = \Xi_g(gq^*fpg^*).
\end{equation}
Let $U$ be an open bisection of $\GT$ containing $\supp(g)$. For $y \in s(U)$, let $\alpha_{U,y}$ denote the unique element of $U$ with source $y$. Define $h_q \coloneqq \psi_T\big(\iota^{-1}(gq^*)\big)$ and $h_p \coloneqq \psi_T\big(\iota^{-1}(pg^*)\big)$. Then
\begin{equation} \label{eqn: psi_T iota inverse homomorphism}
\psi_T\big(\iota^{-1}(g q^* f p g^*)\big) = h_q \, \psi_T\big(\iota^{-1}(f)\big) \, h_p.
\end{equation}
By \cref{prop: properties of Theta}\cref{item: Theta is conjugation},
\begin{equation} \label{eqn: theta is conjugation of g q^* f p g^*}
\psi_T\big(\iota^{-1}\big(\Xi_g(g q^* f p g^*)\big)\big) = \Theta_{U,g}\big(\psi_T\big(\iota^{-1}(g q^* f p g^*)\big)\big).
\end{equation}
Together, \cref{eqn: Phi(b^*fa) = Xi_g(g q^* f p g^*),eqn: theta is conjugation of g q^* f p g^*} imply that
\begin{align*}
\big(\!\ev_{(x,\phi)} \circ \, \psi_T \,\circ\, & \iota^{-1} \circ \Phi\big)(b^*fa) \\
&= \psi_T\big(\iota^{-1}\big(\Xi_g(g q^* f p g^*)\big)\big)(x,\phi) \\
&= \Theta_{U,g}\big(\psi_T\big(\iota^{-1}(g q^* f p g^*)\big)\big)(x,\phi) \\
&= \begin{cases}
\lav g(\alpha_{U,x}) \rav^2 \ \psi_T\big(\iota^{-1}(g q^* f p g^*)\big)\big(r(\alpha_{U,x}), \, \twistchar_{\alpha_{U,x}}^\sigma \, \phi\big) & \quad \text{if } x \in s(U) \\
0 & \quad \text{if } x \notin s(U). \\
\end{cases}
\end{align*}
Thus, to see that $\big(\!\ev_{(x,\phi)} \circ\, \psi_T \circ \iota^{-1} \circ \Phi\big)(b^*fa) = 0$, it suffices to show that if $x \in s(U)$, then
\[
\psi_T\big(\iota^{-1}(g q^* f p g^*)\big) \big(r(\alpha_{U,x}), \, \twistchar_{\alpha_{U,x}}^\sigma \, \phi\big) = 0.
\]
If $x \in s(U)$, then
\[
Q\big(r(\alpha_{U,x}),\, \twistchar_{\alpha_{U,x}}^\sigma \, \phi\big) = \big(r(\alpha_{U,x}), \, \tilde{\twistchar}_{[\alpha_{U,x}]}^\sigma \, \phi\restr{Z_\omega}\big) = \theta_{[\alpha_{U,x}]}(x,\phi\restr{Z_\omega}),
\]
and hence
\[
\big(r(\alpha_{U,x}), \, \twistchar_{\alpha_{U,x}}^\sigma \, \phi\big) \in Q^{-1}\big({\theta_{[\alpha_{U,x}]}(x,\phi\restr{Z_\omega})}\big) \subseteq C_{(x,\phi)}.
\]
Since $\supp\!\big((\psi_T \circ \iota^{-1})(f)\big) \subseteq (X \times \PTHat) \setminus C_{(x,\phi)}$, we obtain $\psi_T\big(\iota^{-1}(f)\big) \big(r(\alpha_{U,x}), \, \twistchar_{\alpha_{U,x}}^\sigma \, \phi\big) = 0$. Combining this with \cref{eqn: psi_T iota inverse homomorphism} gives
\[
\psi_T\big(\iota^{-1}(g q^* f p g^*)\big) \big(r(\alpha_{U,x}), \, \twistchar_{\alpha_{U,x}}^\sigma \, \phi\big) = \big(h_q \, \psi_T\big(\iota^{-1}(f)\big) \, h_p\big) \big(r(\alpha_{U,x}), \, \twistchar_{\alpha_{U,x}}^\sigma \, \phi\big) = 0. \qedhere
\]
\end{proof}

We now construct a state $\kappa_{(x,\phi)}$ of $C^*(\GT,\sigma)$ defined in terms of a fixed element $(x,\phi) \in X \times \PTHat$. In the proof of \cref{thm: simplicity characterisation}\cref{item: main minimal implies simple}, we show that if some point $(x,\phi)$ has non-dense orbit under $\theta$ (so that $\theta$ is not minimal), then the GNS representation associated to $\kappa_{(x,\phi)}$ is nonzero and has nontrivial kernel, and thus $C^*(\GT,\sigma)$
is not simple.

\begin{lemma} \label{lemma: state kappa}
Let $(X,T)$ be a minimal rank-$k$ Deaconu--Renault system such that $X$ is second-countable, and fix $\sigma \in Z^2(\GT,\T)$. Fix $(x,\phi) \in X \times \PTHat$. Let $\Phi\colon C^*(\GT,\sigma) \to M_T^\sigma$ be the conditional expectation of \cite[Lemma~6.2(b)]{Armstrong2022} that extends restriction of functions to $\IT$, and let $\ev_{(x,\phi)}\colon \XX_T^\omega \to C^*(B,\tilde{\omega})$ be the evaluation map $f \mapsto f(x,\phi)$. Let $\trace$ denote the canonical trace on $C^*(B,\tilde{\omega})$ (as defined in \cref{lemma: simple twisted group C*-algebra}). Let
\[
\kappa_{(x,\phi)} \coloneqq \trace \circ \ev_{(x,\phi)} \circ\, \psi_T \circ \iota^{-1} \circ \Phi\colon C^*(\GT,\sigma) \to \C.
\]
For all $h \in C_c(X)$ such that $h(x) = 1$, we have $\kappa_{(x,\phi)}\big(\iota(h \cdot 1_0)\big) = 1$. Moreover, $\kappa_{(x, \phi)}$ is a state of $C^*(\GT,\sigma)$.
\end{lemma}

\begin{proof}
Suppose that $h \in C_c(X)$ satisfies $h(x) = 1$. Since $\Phi \circ \iota = \iota$, we have
\[
\big(\!\ev_{(x,\phi)} \circ\, \psi_T \circ \iota^{-1} \circ \Phi\big) \big(\iota(h \cdot 1_0)\big) = \psi_T(h \cdot 1_0)(x,\phi) = h(x) \, \overline{\phi(0)} \, U_{0+Z_\omega} = U_{0+Z_\omega},
\]
and hence
\begin{equation} \label{eqn: kappa h is 1}
\kappa_{(x,\phi)}\big(\iota(h \cdot 1_0)\big) = \trace(U_{0+Z_\omega}) = 1.
\end{equation}
Since $\Phi$, $\iota^{-1}$, $\psi_T$, $\ev_{(x,\phi)}$, and $\trace$ are all positive norm-decreasing linear maps, $\kappa_{(x,\phi)}$ is a positive linear functional, and $\lv \kappa_{(x,\phi)} \rv \le 1$. By Urysohn's lemma there exists $h \in C_c(X)$ such that $h(x) = 1$. Then \cref{eqn: kappa h is 1} implies that $\lv \kappa_{(x,\phi)} \rv \ge 1$, and thus $\kappa_{(x,\phi)}$ is a state of $C^*(\GT,\sigma)$.
\end{proof}

We conclude this section by proving \cref{thm: simplicity characterisation}\cref{item: main minimal implies simple}, which says that if $(X,T)$ is minimal, then $C^*(\GT,\sigma)$ is simple if and only if the spectral action $\theta$ is minimal.

\begin{proof}[Proof of \cref{thm: simplicity characterisation}\cref{item: main minimal implies simple}]
Suppose that $\theta$ is minimal. Let $I$ be a nontrivial ideal of $C^*(\GT,\sigma)$. Then there exists a non-injective homomorphism $\Psi$ of $C^*(\GT,\sigma)$ such that $I = \ker(\Psi)$. By \cite[Theorem~6.3]{Armstrong2022}, $J \coloneqq \ker(\Psi \circ \iota)$ is a nontrivial ideal of $C^*(\IT,\sigma)$. We have
\[
\iota(J) \,=\, \{ \iota(a) : a \in C^*(\IT,\sigma),\, \Psi(\iota(a)) = 0 \} \,=\, \{ b \in M_T^\sigma : \Psi(b) = 0 \} \,=\, I \medcap M_T^\sigma \,\subseteq\, I.
\]
Thus, to see that $C^*(\GT,\sigma)$ is simple, it suffices to show that $J = C^*(\IT,\sigma)$, because then $\iota(C_0(X)) \subseteq \iota(J) \subseteq I$, and (as argued in \cite[Theorem~5.3.13]{Armstrong2019}) \cite[Proposition~3.18]{Exel2008} implies that $I =
C^*(\GT,\sigma)$. Define
\[
K_J \coloneqq \{ (x,\phi) \in X \times \PTHat \,:\, f(x,\phi) = 0 \text{ for all } f \in \psi_T(J) \}.
\]
Since $\psi_T(J)$ is an ideal of $\XX_T^\omega$, \cref{prop: ideals of twisted isotropy induced algebra} implies that $K_J$ is a closed subset of $X \times \PTHat$, and
\[
\psi_T(J) = \{f \in \XX_T^\omega : f\restr{K_J} \equiv 0 \}.
\]
Let $Q_T\colon X \times \PTHat \to X \times \widehat{Z}_\omega$ be the quotient map $(x,\phi) \mapsto (x,\phi\restr{Z_\omega})$. Suppose that $Q(K_J)$ is nonempty, and fix $(x,\zeta) \in Q(K_J)$. By \cref{prop: Q(K_J) is closed and invariant under theta}, $Q(K_J)$ is closed and invariant under $\theta$, and hence
\[
\overline{[x,\zeta]_\theta} = \overline{\{ \theta_{[\gamma]}(x,\zeta) : \gamma \in (\GT)_x \}} \subseteq Q(K_J).
\]
Since $\theta$ is minimal by assumption, $Q(K_J) = X \times \widehat{Z}_\omega$. Thus, \cref{prop: Q(K_J) is closed and invariant under theta} implies that
\[
K_J = Q^{-1}\!\left(Q(K_J)\right) = Q^{-1}\big(X \times \widehat{Z}_\omega\big) = X \times \PTHat.
\]
Hence
\[
\psi_T(J) = \{f \in \XX_T^\omega : f\restr{K_J} \equiv 0 \} = \{0\},
\]
which contradicts that $J$ is nontrivial, because $\psi_T$ is injective. Therefore, $Q(K_J) = \varnothing$, forcing $K_J = \varnothing$, and hence $\psi_T(J) = \XX_T^\omega$. Since $\psi_T$ is an isomorphism, $J = C^*(\IT,\sigma)$, and hence $C^*(\GT,\sigma)$ is simple.

For the converse, we prove the contrapositive. Suppose that $\theta$ is not minimal. Then there exists $(x,\phi) \in X \times \PTHat$ such that
\[
\overline{[x, \phi\restr{Z_\omega}]_\theta} = \overline{\{ \theta_{[\gamma]}(x,\phi\restr{Z_\omega}) : \gamma \in (\GT)_x \}} \ne X \times \widehat{Z}_\omega.
\]
Let
\[
\kappa_{(x,\phi)} \coloneqq \trace \circ \ev_{(x,\phi)} \circ\, \psi_T \circ \iota^{-1} \circ \Phi\colon C^*(\GT,\sigma) \to \C
\]
be the state of $C^*(\GT,\sigma)$ defined in \cref{lemma: state kappa}. Let $\kappa \coloneqq \kappa_{(x, \phi)}$, let
\[
N_\kappa \coloneqq \{ f \in C^*(\GT,\sigma) : \kappa(f^*f) = 0 \}
\]
be the null space for $\kappa$, and let $\pi_\kappa\colon C^*(\GT,\sigma) \to B(\HH_\kappa)$ be the GNS representation associated to $\kappa$. To see that $C^*(\GT,\sigma)$ is not simple, it suffices to prove that
\[
\{0\} \ne \ker(\pi_\kappa) \ne C^*(\GT,\sigma).
\]
Since $\kappa \ne 0$, we have $\HH_\kappa \ne \{0\}$. So since $\pi_\kappa$ is nondegenerate, $\ker(\pi_\kappa) \ne C^*(\GT, \sigma)$. We now show that $\ker(\pi_\kappa) \ne \{0\}$. Define $C_{(x,\phi)} \coloneqq Q^{-1}\big(\overline{[x, \phi\restr{Z_\omega}]_\theta}\big)$. Since $\overline{[x, \phi\restr{Z_\omega}]_\theta} \ne X \times \widehat{Z}_\omega$, \cref{lemma: supp outside orbit} shows that $C_{(x,\phi)}$ is a proper closed subset of $X \times \PTHat$, and there exists $f \in M_T^\sigma {\setminus} \{0\}$ such that
\[
\supp\!\big((\psi_T \circ \iota^{-1})(f)\big) \subseteq (X \times \PTHat) \setminus C_{(x,\phi)}.
\]
Fix $a, b \in C^*(\GT,\sigma)$. To see that $\pi_\kappa(f) = 0$, it suffices to show that
\[
\big(\pi_\kappa(f)(a + N_k) \mid b + N_k\big) = 0.
\]
Since $\pi_\kappa$ is the GNS representation associated to $\kappa$, we have
\[
\big(\pi_\kappa(f)(a + N_k) \mid b + N_k\big) = \big(fa + N_k \mid b + N_k\big) = \kappa(b^*fa).
\]
By \cref{prop: conjugation supp outside orbit}, we have $(\psi_T \circ \iota^{-1} \circ \Phi)(b^*fa)(x,\phi) = 0$, and hence
\[
\kappa(b^*fa) = \big(\!\trace \circ \ev_{(x,\phi)} \circ\, \psi_T \circ \iota^{-1} \circ \Phi\big)(b^*fa) = \trace(0) = 0.
\]
Hence $\big(\pi_\kappa(f)(a + N_k) \mid b + N_k\big) = 0$, giving $\ker(\pi_\kappa) \ne \{0\}$.
\end{proof}

\begin{remark}
If $X$ is the infinite-path space of a cofinal, row-finite $k$-graph with no sources, and each $T^n$ is the degree-$n$ shift map, then \cref{thm: simplicity characterisation} coincides with the simplicity characterisation given in \cite[Corollary~4.8]{KPS2016JNG}.
\end{remark}

\begin{remark}
Theorem~5.1 of \cite{BCFS2014} shows that $C^*(\GT)$ is simple if and only if $\GT$ is minimal and effective. We claim that \cite[Theorem~5.1]{BCFS2014}, in the special case of Deaconu--Renault groupoids, is equivalent to \cref{thm: simplicity characterisation} when $\sigma$ is trivial. In this case, $\omega$ and each $\tilde{\twistchar}_\gamma^\sigma$ are also trivial, and $Z_\omega = \PT$. So \cref{thm: simplicity characterisation} says that $C^*(\GT,\sigma) = C^*(\GT)$ is simple if and only if the set
\[
[x,\phi]_\theta = \left\{ \big(r(\gamma),\, \tilde{\twistchar}_{[\gamma]}^\sigma \, \phi\big) : \gamma \in (\GT)_x \right\} = r\big((\GT)_x\big) \times \{\phi\}
\]
is dense in $X \times \PTHat$ for all $(x,\phi) \in X \times \PTHat$. Since $r\big((\GT)_x\big)$ is dense in $X$ and $X \times \{\phi\}$ is closed, we deduce that $\theta$ is minimal if and only if $\PT = \{0\}$. By \cref{cor: GT effective iff PT trivial}, this occurs precisely when $\GT$ is effective.
\end{remark}

\section{An application to some crossed products by \texorpdfstring{$\Z$}{the integers}}
\label{sec: application}

In this section we apply our theorem to characterise simplicity of crossed products of C*-algebras of rank-$1$ Deaconu--Renault groupoids arising from continuous $\T$-valued functions on the underlying spaces. We then specialise this to the analogue of quasi-free actions on topological-graph C*-algebras.

\subsection{Crossed products as twisted groupoid C*-algebras}

To apply our main theorem to understand crossed products of C*-algebras of Deaconu--Renault groupoids, we need to realise the latter as twisted groupoid C*-algebras. This follows from a more general result about crossed products of \'etale groupoids by actions of $\Z$ induced by $\T$-valued $1$-cocycles that may be of independent interest; so we record the general result first. We thank the referee for suggesting the more general formulation.

We will first need the following folklore result about multipliers of the C*-algebras of Hausdorff \'etale groupoids. We write $C_b(Y)$ for the set of continuous, bounded, complex-valued functions on a locally compact Hausdorff space $Y$.

\begin{lemma} \label{lemma: bisection multipliers}
Let $\GG$ be a Hausdorff \'etale groupoid, and fix $\sigma \in Z^2(\GG,\T)$. Suppose that $B \subseteq \GG$ is a clopen bisection of $\GG$ such that $s(B)$ and $r(B)$ are closed, and fix $f \in C_b(B)$. For $g \in C_c(\GG,\sigma)$, the convolution product $f * g$ given by
\[
(f * g)(\gamma) = \sum_{\alpha \in \GG^{r(\gamma)}} \sigma(\alpha,\alpha^{-1}\gamma) \, f(\alpha) \, g(\alpha^{-1}\gamma)
\]
belongs to $C_c(\GG,\sigma)$. There is a multiplier $M_f$ of $C^*(\GG,\sigma)$ such that for $g \in C_c(\GG,\sigma)$ we have $M_f(g) = f * g$. If $s(B) = r(B) = \GGo$ and $f(B) \subseteq \T$, then $M_f$ is a unitary multiplier of $C^*(\GG,\sigma)$. If $B_1$ and $B_2$ are two clopen bisections such that $r(B_i)$ and $s(B_i)$ are closed for each $i$, and $f_i \in C_b(B_i)$ for each $i$, then the convolution product $f_1 * f_2$ belongs to $C_b(B_1 B_2)$, and we have $M_{f_1} \circ M_{f_2} = M_{f_1 * f_2}$;
likewise, $f_1^* \in C_b(B_1^{-1})$ and $M^*_{f_1} = M_{f_1^*}$.
\end{lemma}

\begin{proof}
Fix $g \in C_c(\GG,\sigma)$. Since $\supp(g)$ is compact, its image $r(\supp(g))$ under the continuous range map is also compact. Use Urysohn's lemma to fix a compactly supported function $h \in C_c\big(\GGo,[0,1]\big)$ such that $h\restr{r(\supp(g))} \equiv 1$. Then $f * g = f * (h * g) = (f * h) * g$. The function $f * h$ is given by $(f * h)(\gamma) = f(\gamma) h(s(\gamma))$, and since $\supp(f * h) \subseteq s\restr{B}^{-1}\big(\!\supp(h) \cap s(B)\big) \subseteq B$, it follows that $f * h \in C_c(\GG,\sigma)$. So $f * g = (f * h) * g \in C_c(\GG,\sigma)$.

Using the same function $h$ as above, we see that
\[
\lv f * g \rv^2 = \lv ((f * h) * g)^* * ((f * h) * g) \rv = \lv g^* * (h * f^* * f * h) * g \rv.
\]
In $C^*(\GG,\sigma)$, we have
\[
g^* * (h * f^* * f * h) * g \le \lv h * f^* * f * h\rv_\infty \, g^*g \le \lv f \rv_\infty^2 \, g^*g,
\]
and so we deduce that $\lv f * g \rv \le \lv f \rv_\infty \lv g \rv$. So the map $g \mapsto f * g$ on $C_c(\GG,\sigma)$ extends to a bounded linear map $M_f$ of norm at most $\lv f \rv_\infty$ on $C^*(\GG,\sigma)$. Defining $f^*\colon B^{-1} \to \C$ by $f^*(\gamma) = \overline{\sigma(\gamma^{-1}, \gamma)} \overline{f(\gamma^{-1})}$ as usual, associativity of multiplication shows that for $g,h \in C_c(\GG,\sigma)$, we have $M_f(g)^* * h = g^* * f^* * h = g^* * M_{f^*}(h)$. Thus $M_f$ is adjointable with respect to the standard inner product on $C^*(\GG,\sigma)$, with adjoint $M_{f^*}$. So $M_f$ is a multiplier, as claimed. If $s(B) = r(B) = \GGo$ and $f(B) \subseteq \T$, then for $g \in C_c(\GG,\sigma)$ we have $M^*_f(M_f(g)) = f^* * f * g = 1_{\GGo} * g = g$, and similarly, $M_f(M^*_f(g)) = g$. So continuity gives $M^*_f M_f = M_f M^*_f = 1_{\MM(C^*(\GG,\sigma))}$, and thus $M_f$ is a unitary.

For the final statement, we already saw that $M_{f^*} = M^*_f$ for all $f$, so we just have to establish the multiplicativity. If $B_1$ and $B_2$ are clopen bisections and $f_i \in C_b(B_i)$, then $B_1B_2$ is an open bisection because multiplication is open. To see that $B_1B_2$ is also closed, suppose that $\gamma_i \to \gamma$ and each $\gamma_i \in B_1B_2$. Then each $\gamma_i$ can be written as $\alpha_i\beta_i$ with each $\alpha_i$ in $B_1$ and each $\beta_i$ in $B_2$. Since $\gamma_i \to \gamma$, we have $r(\alpha_i) = r(\gamma_i) \to r(\gamma)$, and then since $r\restr{B_1}$ is a homeomorphism, $\alpha_i \to \alpha$ for some $\alpha \in B_1$. Similarly (using $s$ in place of $r$), we have $\beta_i \to \beta$ for some $\beta \in B_2$. Since each $s(\alpha_i) = r(\beta_i)$, continuity gives $s(\alpha) = r(\beta)$, and since $\alpha_i\beta_i = \gamma_i \to \gamma$, continuity also gives $\alpha\beta = \gamma$. So $\gamma \in B_1B_2$. The convolution formula shows that $\supp(f_1 * f_2) \subseteq B_1 B_2$ and that $\lv f_1 * f_2 \rv_\infty \le \lv f_1 \rv_\infty \lv f_2 \rv_\infty$. For $g \in C_c(\GG,\sigma)$ we have $M_{f_1}(M_{f_2}(g)) = f_1 * (f_2 * g) = (f_1 * f_2) * g = M_{f_1 * f_2}(g)$, and then continuity gives $M_{f_1}M_{f_2} = M_{f_1*f_2}$.
\end{proof}

We can now discuss how to realise certain crossed products of \'etale-groupoid C*-algebras as twisted groupoid C*-algebras.

Let $\GG$ be a locally compact Hausdorff groupoid with a Haar system and let $c\colon \GG \to \T$ be a continuous $1$-cocycle. By \cite[Proposition~II.5.1]{Renault1980} there is an action $\alpha = \alpha^c$ of $\Z$ on $C^*(\GG)$ such that
\begin{equation} \label{eqn: induced action}
\alpha_n(f)(\gamma) = c(\gamma)^n f(\gamma) \quad \text{ for all } n \in \Z, \, f \in C_c(\GG), \text{ and } \gamma \in \GG.
\end{equation}

\begin{prop} \label{prop: CP gives twist}
Let $\GG$ be a second-countable Hausdorff \'etale groupoid. Suppose that $c\colon \GG \to \T$ is a continuous $1$-cocycle, and let $\alpha = \alpha^c$ be the corresponding action of $\Z$ on $C^*(\GG)$ as in \cref{eqn: induced action}. There is a continuous $\T$-valued $2$-cocycle $\omega = \omega_c$ on $\GG \times \Z$ given by $\omega((\beta,m),(\gamma,n)) \coloneqq c(\beta)^n$, and there is an isomorphism $\phi\colon C^*(\GG) \rtimes_\alpha \Z \to C^*(\GG \times \Z, \omega)$ such that
\[
\phi\big(i_{C^*(\GG)}(f)\,i_\Z(n)\big)(\gamma,p) = \delta_{-n,p} \, f(\gamma),
\]
for $f \in C_c(\GG)$, $n \in \Z$, and $(\gamma,p) \in \GG \times \Z$.
\end{prop}

\begin{proof}
First note that $\omega$ is normalised because $c(r(\gamma))^n = c(\gamma)^0 = 1$ for all $\gamma \in \GG$ and $n \in \Z$. To see that $\omega$ satisfies the $2$-cocycle identity, fix a composable triple $\big((\beta,m), (\gamma,n), (\lambda,p)\big)$ in $\GG \times \Z$. Then
\begin{align*}
\omega\big((\beta,m), (\gamma,n)\big) \, \omega\big((\beta,m)(\gamma,n), (\lambda, p)\big) &= c(\beta)^n \, c(\beta\gamma)^p = c(\beta)^{n+p} \, c(\gamma)^p \\
&= \omega\big((\beta,m), (\gamma,n)(\lambda,p)\big) \, \omega\big((\gamma,n),(\lambda,p)\big).
\end{align*}

For the final statement, first note that for $n \in \Z$, the set $\GGo \times \{n\}$ is a clopen bisection of $\GG \times \Z$ with range and source equal to $(\GG \times \Z)^{(0)}$. Hence \cref{lemma: bisection multipliers} shows that there is a unitary multiplier $U_n$ of $C^*(\GG \times \Z, \omega)$ that acts on $C_c(\GG \times \Z, \omega)$ by convolution with the characteristic function $1_{\GGo \times \{-n\}}$. Since $c$ vanishes on $\GGo$, the final statement of \cref{lemma: bisection multipliers} shows that $n \mapsto U_n$ is a unitary representation of $\Z$.

Since $\GG \times \{0\}$ is isomorphic to $\GG$ and $\omega$ is trivial on $\GG \times \{0\}$, the universal property of $C^*(\GG)$ yields a homomorphism $\pi\colon C^*(\GG) \to C^*(\GG \times \Z,\omega)$ such that $\pi(f)(\gamma,m) = \delta_{m,0} \, f(\gamma)$ for all $f \in C_c(\GG)$ and $(\gamma,m) \in \GG \times \Z$. We claim that $\pi$ is nondegenerate. To see this, fix $g \in C_c(\GG \times \Z, \omega)$, and use Urysohn's lemma to choose $f \in C_c(\GG)$ such that $\supp(f) \subseteq \GGo$ and $f\restr{\pi_1(r(\supp(g)))} \equiv 1$, where $\pi_1$ is the projection of $\GG \times \Z$ onto the first coordinate. A routine calculation shows that $\pi(f)g = g$, and hence $\pi$ is nondegenerate.

We claim that $(\pi,U)$ is a covariant representation of $(C^*(\GG), \Z, \alpha)$. To see this, fix $f \in C_c(\GG)$ and $n \in \Z$. For all $(\gamma, m) \in \GG \times \Z$, we have
\begin{align*}
(U_n \, \pi(f) \, &U^*_n)(\gamma,m) \\
&= (U_n \, \pi(f) \, U_{-n})(\gamma,m) \\
&= \sum_{(\eta,p)(\beta,q)(\lambda,l) = (\gamma,m)} c(\eta)^{q+l} \, c(\beta)^l \, 1_{\GGo \times \{-n\}}(\eta,p) \, \pi(f)(\beta,q) \, 1_{\GGo \times \{n\}}(\lambda,l).
\end{align*}
If $(\eta,p)(\beta,q)(\lambda,l) = (\gamma,m)$ contributes a nonzero term, then $(\eta, p) \in \GGo \times \{-n\}$ and $(\lambda, l) \in \GGo \times \{n\}$; thus $(\eta, p) = (r(\gamma), -n)$ and $(\lambda, l) = (s(\gamma), n)$, and hence $(\beta, q) = (\gamma, m)$. So we obtain
\[
(U_n \, \pi(f) \, U^*_n)(\gamma, m) = c(\gamma)^n \, \pi(f)(\gamma,m) = \delta_{m,0} \, c(\gamma)^n f(\gamma) = \pi(\alpha_n(f))(\gamma, m).
\]
Therefore, $(\pi,U)$ is a nondegenerate covariant representation of $(C^*(\GG), \Z, \alpha)$, and so the universal property of the crossed product gives a homomorphism $\phi\colon C^*(\GG) \rtimes_{\alpha} \Z \to C^*(\GG \times \Z, \omega)$ such that $\phi\big(i_{C^*(\GG)}(f) \, i_\Z(n)\big)(\gamma,p) = (\pi(f) \, U_n)(\gamma, p) = \delta_{-n,p} \, f(\gamma)$.

To see that $\phi$ is injective, it suffices by \cite[Proposition~4.5.1]{BO2008} to show that $\pi$ is injective and that there is a strongly continuous action $\beta$ of $\T$ on $C^*(\GG \times \Z, \omega)$ such that for each $z \in \T$, we have $\beta_z(\pi(f)) = \pi(f)$ for all $f \in C_c(\GG)$, and the extension $\overline{\beta}_z$ of $\beta_z$ to the multiplier algebra $\MM(C^*(\GG \times \Z, \omega))$ satisfies $\overline{\beta}_z(U_n) = z^n U_n$ for all $n \in \Z$.

We first show that $\pi$ is injective. Let $Y$ denote the right-$C^*(\GG)$-module direct sum $\bigoplus_{n \in \Z} C^*(\GG)$. For $f \in C_c(\GG \times \Z, \omega)$ and $n \in \Z$, we write $f_n \in C_c(\GG)$ for the function such that $f(\gamma,n) = f_n(\gamma)$ for all $\gamma \in \GG$. For $f \in C_c(\GG \times \Z)$ and $\xi \in C_c(\Z, C^*(\GG)) \subseteq Y$, define $f \cdot \xi \in C_c(\Z, C^*(\GG))$ by $(f \cdot \xi)(n) \coloneqq \sum_{p+q = n} \alpha_p(f_q) \xi(p)$, where the product $\alpha_p(f_q) \xi(p)$ is computed in $C^*(\GG)$.

We claim that $f \mapsto \big(\xi \mapsto f\cdot \xi\big)$ is a $*$-homomorphism from $C_c(\GG \times \Z, \omega)$ to $\LL(Y)$. To see that it is multiplicative, fix $f \in C_c(\GG \times \{a\}) \subseteq C_c(\GG \times \Z, \omega)$ and $g \in C_c(\GG \times \{b\}) \subseteq C_c(\GG \times \Z, \omega)$. Then $fg \in C_c(\GG \times \{a+b\})$, and for $\gamma \in \GG$, we have
\begin{align*}
(fg)_{a+b}(\gamma) &= \sum_{(\beta,p)(\lambda,q) = (\gamma, a+b)} \omega((\beta,p),(\lambda,q)) \, f(\beta,p) \, g(\lambda,q) \\
&= \sum_{\beta\lambda = \gamma} \omega((\beta,a),(\lambda,b)) \, f_a(\beta) \, g_b(\lambda) = \sum_{\beta\lambda = \gamma} c(\beta)^b \, f_a(\beta) \, g_b(\lambda) = (\alpha_b(f_a)g_b)(\gamma).
\end{align*}
Thus, for $\xi \in C_c(\Z, C^*(\GG))$ and $n \in \Z$, we have
\begin{align*}
(f \cdot (g \cdot \xi))(n) &= \alpha_{n-a}(f_a) \, (g \cdot \xi)(n-a) = \alpha_{n-a}(f_a) \, \alpha_{n-(a+b)}(g_b) \, \xi(n-(a+b)) \\
&= \alpha_{n-(a+b)}(\alpha_b(f_a)g_b) \, \xi(n-(a+b)) \\
&= \alpha_{n-(a+b)}((fg)_{a+b}) \, \xi(n-(a+b)) = \sum_{p+q = n} \alpha_p((fg)_q) \, \xi(p) = ((fg) \cdot \xi)(n).
\end{align*}
Hence $f \mapsto \big(\xi \mapsto f \cdot \xi\big)$ is multiplicative.

To see that it preserves adjoints, fix $f \in C_c(\GG \times \{m\}) \subseteq C_c(\GG \times \Z, \omega)$, and $\xi \in C_c(\{p\}, C^*(\GG)) \subseteq Y$ and $\eta \in C_c(\{q\}, C^*(\GG)) \subseteq Y$. Then
\[
\langle f \cdot \xi, \eta\rangle_{C^*(\GG)} = (f \cdot \xi)(q)^* \, \eta(q) = \delta_{p+m,q} \, \xi(p)^* \, \alpha_p(f_m)^* \, \eta(q).
\]
Since $\omega((\gamma,-m), (\gamma,-m)^{-1}) = c(\gamma)^m$ for $\gamma \in \GG$, a computation shows that $\alpha_m((f^*)_{-m}) = (f_m)^*$. Hence $f^* \in C_c(\GG \times \{-m\})$, and
\[
\langle \xi, f^* \cdot \eta \rangle_{C^*(\GG)} = \xi(p)^* \, (f^* \cdot \eta)(p) = \delta_{q-m,p} \, \xi(p)^* \, \alpha_{p+m}((f^*)_{-m}) \, \eta(q) = \langle f \cdot \xi, \eta \rangle_{C^*(\GG)}.
\]
So $f \mapsto \big(\xi \mapsto f \cdot \xi\big)$ preserves adjoints, and hence is a $*$-homomorphism.

The universal property of $C^*(\GG \times \Z, \omega)$ therefore implies that there is a homomorphism $\psi\colon C^*(\GG \times \Z, \omega) \to \LL(Y)$ such that $\psi(f)\xi = f \cdot \xi$ for $f \in C_c(\GG \times \Z, \omega)$ and $\xi \in Y$.

Let $i_0\colon C^*(\GG)_{C^*(\GG)} \to Y$ be the inclusion as the $0$-submodule. A quick calculation shows that for $a,b \in C_c(\GG)$, we have $\psi(\pi(a))i_0(b) = i_0(ab)$. Since $i_0$ is isometric and since the left action of $C^*(\GG)$ on itself by multiplication is isometric, we deduce that $\psi \circ \pi$ is injective, and hence $\pi$ is injective.

So we just need to construct the action $\beta$. For $z \in \T$, the map $\beta_z\colon C_c(\GG \times \Z, \omega) \to C^*(\GG \times \Z, \omega)$ given by $\beta_z(f)(\gamma,n) \coloneqq z^n f(\gamma,n)$ is a $*$-homomorphism, and hence the universal property of $C^*(\GG \times \Z, \omega)$ implies that it extends to an endomorphism $\beta_z$ of $C^*(\GG \times \Z, \omega)$. Since $\beta_{\overline{z}} \circ \beta_z$ is the identity map on $C_c(\GG \times \Z, \omega)$, each $\beta_z$ is an automorphism, and since $\beta_{z} \circ \beta_w$ agrees with $\beta_{zw}$ on $C_c(\GG \times \Z, \omega)$, we see that $z \mapsto \beta_z$ is a homomorphism. For $f \in C_c(\GG \times \{n\})$, the map $z \mapsto \beta_z(f)$ is clearly continuous, and an $\frac{\varepsilon}{3}$-argument then shows that $\beta$ is a strongly continuous action of $\T$.

We claim that the extension of each $\beta_z$ to the multiplier algebra $\MM(C^*(\GG \times \Z, \omega))$ satisfies $\overline{\beta}_z(U_n) = z^n U_n$ for each $n \in \Z$. To see this, fix $n \in \Z$ and an increasing sequence $K_i \subseteq \GGo$ of compact sets with $\bigcup_{i \in \N} K_i = \GGo$, and for each $i \in \N$, fix $h_i \in C_c\big(\GGo \times \{n\}, [0,1]\big)$ such that $h_i\restr{K_i \times \{n\}} \equiv 1$. For $f \in C_c(\GG \times \Z)$ there exists $N \in \N$ large enough so that $r(\supp(f)) \subseteq K_N \times \{0\}$, and then $h_i * f = U_n f$ for all $i \ge N$. So the sequence $(h_i)_{i \in \N}$ converges strictly to $U_n$, and since $\beta_z(h_i) = z^n h_i$ for all $z \in \T$, this establishes the claim. Thus $\phi$ is injective.

It remains only to prove that $\phi$ is surjective. For this, fix an open bisection $B$ of $\GG \times \Z$ and distinct points $\beta,\gamma \in B$. Then $B = B' \times \{n\}$ for some open bisection $B'$ of $\GG$ and some $n \in \Z$, and so $\beta = (\beta', n)$ and $\gamma = (\gamma',n)$ for distinct $\beta', \gamma' \in \GG$. Fix $f \in C_c(\GG)$ such that $\supp(f) \subseteq B'$, $f(\beta') = 1$, and $f(\gamma') = 0$. Then the support of $\phi\big(i_{C^*(\GG)}(f) \, i_\Z(-n)\big) = \pi(f) \, U_{-n}$ is contained in $B$, and we have $\phi\big(i_{C^*(\GG)}(f) \, i_\Z(-n)\big)(\beta',n) = f(\beta') = 1$ and $\phi\big(i_{C^*(\GG)}(f) \, i_\Z(-n)\big)(\gamma',n) = f(\gamma') = 0$. So \cite[Corollary~9.3.5]{Sims2020} shows that $\phi$ is surjective.
\end{proof}

\subsection{An application \texorpdfstring{of \cref{thm: simplicity characterisation}}{} to crossed products}

We now make use of \cref{thm: simplicity characterisation} and \cref{prop: CP gives twist} to study simplicity of certain crossed products of C*-algebras of Deaconu--Renault groupoids.

The following is an immediate corollary of results of Renault \cite[Section~4.1]{Renault2003} (see also \cite{FHKP2022}) together with \cref{prop: CP gives twist}; we have written it out primarily to establish our set-up for the rest of the section.

\begin{cor} \label{cor: the CP}
Let $(X,T)$ be a rank-$1$ Deaconu--Renault system such that $X$ is second-countable, and let $h\colon X \to \T$ be a continuous function.
\begin{enumerate}[label=(\alph*)]
\item \label{item: eventually constant} For each $(x, p, y) \in \GT \subseteq X \times \Z \times X$, the sequence
\[
\Big(\prod^N_{i=0} h(T^i(x)) \prod^{N-p}_{i=0} \overline{h(T^i(y))}\Big)^\infty_{N = \lav p \rav}
\]
is eventually constant.
\item \label{item: continuous 1-cocycle} There is a continuous $1$-cocycle $\widetilde{h}\colon \GT \to \T$ such that
\[
\widetilde{h}(x, p, y) = \prod^N_{i=0} h(T^i(x)) \prod^{N-p}_{i=0} \overline{h(T^i(y))}
\]
for large $N \in \N$.
\item \label{item: action exists} There is an action $\alpha^h\colon \Z \to \Aut(C^*(\GT))$ such that $\alpha^h_n(f)(\gamma) = \widetilde{h}(\gamma)^n f(\gamma)$ for all $f \in C_c(\GT)$ and $\gamma \in \GT$.
\item \label{item: 2-cocycle} There is a continuous $2$-cocycle $c_h\colon (\GT \times \Z)^{(2)} \to \T$ given by
\[
c_h((\alpha, m),(\beta, n)) \coloneqq \widetilde{h}(\alpha)^n,
\]
and there is an isomorphism $\phi\colon C^*(\GT) \rtimes_{\alpha^h} \Z \to C^*(\GT \times \Z, c_h)$ such that
\[
\phi\big(i_{C^*(\GT)}(f) \, i_\Z(n)\big)(\gamma,p) = \delta_{-n,p} \, f(\gamma),
\]
for $f \in C_c(\GT)$, $n \in \Z$, and $(\gamma,p) \in \GT \times \Z$.
\end{enumerate}
\end{cor}

\begin{proof}
Statements \cref{item: eventually constant}~and~\cref{item: continuous 1-cocycle} follow from the arguments of \cite[Section~4.1]{Renault2003} or \cite[Proposition~3.10]{FHKP2022}. The action \cref{item: action exists} is the one described in \cref{eqn: induced action}. The final statement is a special case of \cref{prop: CP gives twist}.
\end{proof}

\begin{thm} \label{thm: CP simple}
Let $(X,T)$ be a rank-$1$ Deaconu--Renault system such that $X$ is second-countable, and let $h\colon X \to \T$ be a continuous function. Let $\widetilde{h}\colon \GT \to \T$ be the $1$-cocycle of \cref{cor: the CP}\cref{item: continuous 1-cocycle}, and let $\alpha^h\colon \Z \to \Aut(C^*(\GT))$ be the action of \cref{cor: the CP}\cref{item: action exists}. Write $\rho$ for the action of $\GT$ on $X \times \T$ given by $\rho_\gamma(s(\gamma), z) \coloneqq (r(\gamma), \widetilde{h}(\gamma) z)$. Suppose that $X$ is an uncountable space. Then the crossed product $C^*(\GT) \rtimes_{\alpha^h} \Z$ is simple if and only if $\rho$ is minimal.
\end{thm}

In order to prove \cref{thm: CP simple}, we need the following lemma.

\begin{lemma} \label{lemma: X uncountable GT top principal}
Let $(X,T)$ be a minimal rank-$1$ Deaconu--Renault system such that $X$ is second-countable. If $X$ is uncountable, then $\GT$ is topologically principal.
\end{lemma}

\begin{proof}
Since $\GT$ is second-countable, it suffices by \cite[Lemma~3.1]{BCFS2014} to show that $\IT = \GTo$. To see this, we suppose that $\IT \ne \GTo$ and derive a contradiction. Recall from \cref{prop: IT characterisation} that $\IT = \{ (x, p, x) : p \in \PT \}$. Since $\IT$ is nontrivial, there exists $p \in \Z {\setminus} \{0\}$ such that $(x,p,x) \in \GT$ for all $x$. By definition of the topology on $\GT$, it follows that for each $x \in X$ there is an open neighbourhood $U$ of $x$ and a pair $m > n \in \N$ such that $T^m(x) = T^n(x)$ for all $x \in U$. Since the pairs $m > n \in \N$ are countable and $X$ is not countable, that $\GT$ is second-countable implies that there exist $x, U, m, n$ as above so that $U$ is not countable. Since $X$ is second-countable and $T^n$ is a local homeomorphism, $(T^n)^{-1}(x)$ is countable for every $x \in X$, and so $V = T^n(U)$ is an uncountable open set and $p = m - n > 0$ satisfies $T^p(x) = x$ for all $x \in V$. Fix $x \in V$. Since $V$ is uncountable, there exists $y \in V$ such that
\begin{equation} \label[condition]{cond: for contradiction}
T^q(x) \ne y \ \text{ for all } q \in \N.
\end{equation}
Since $\GT$ is minimal, there is a sequence $(z_i, m_i, x)^\infty_{i=1}$ in $\GT$ such that $z_i \to y$. Write each $m_i = a_i - b_i$ with $a_i, b_i \ge 0$ so that $T^{a_i}(z_i) = T^{b_i}(x)$. For each $i \in \N$, there exists $k > 0$ such that $kp > b_i$; and then $T^{a_i + (kp-b_i)}(z_i) = T^{kp}(x) = x$. So we can assume that each $m_i > 0$ and that $T^{m_i}(z_i) = x$ for all $i \in \N$. By passing to a subsequence, we may assume that each $m_i - m_1$ is divisible by $p$. Fix $l > 0$ such that $lp > m_1$, let $d \coloneqq lp - m_1$, let $z \coloneqq T^d(x)$, and let $n_i \coloneqq m_i + d$ for all $i \in \N$. Then $T^{n_i}(z_i) = z$ for all $i \in \N$, and each $n_i$ is divisible by $p$. Since $z_i \to y$, we eventually have $z_i \in V$, and so we eventually have $z = T^{n_i}(z_i) = z_i$. But this forces $y = z = T^d(x)$, which contradicts \cref{cond: for contradiction}. Thus $\IT$ is trivial, as claimed.
\end{proof}

\begin{proof}[Proof of \cref{thm: CP simple}]
By \cref{cor: the CP}\cref{item: 2-cocycle}, the crossed product $C^*(\GT) \rtimes_{\alpha^h} \Z$ is isomorphic to the twisted groupoid C*-algebra $C^*(\GT \times \Z, c_h)$, to which we aim to apply \cref{thm: simplicity characterisation}. For this, observe first that if $\overline{T}$ is the action of $\N^2$ given by $\overline{T}^{(m,n)} = T^m$, then $\GT \times \Z \cong \GG_{\overline{T}}$.

First suppose that $(X,T)$ is not minimal. Then $\GG_{\overline{T}}$ is also not minimal, and the action $\rho$ is not minimal. So $(X,\overline{T})$ is not minimal, and hence \cref{thm: simplicity characterisation}\cref{item: main simple implies minimal} implies that $C^*(\GT) \rtimes_{\alpha^h} \Z \cong C^*(\GG_{\overline{T}},c_h)$ is not simple. So it suffices to prove the result when $(X,T)$ is minimal.

Now suppose that $(X,T)$ is minimal. Since $X$ is uncountable, \cref{lemma: X uncountable GT top principal} implies that $\IT = \GTo$. The isomorphism $\GT \times \Z \to \GG_{\overline{T}}$ is given by $((x, m, y), n) \mapsto (x, (m,n), y)$. So the interior $\II_{\overline{T}}$ of the isotropy of $\GG_{\overline{T}}$ is precisely $\{ (x, (m,n), x) : (x,m,x) \in \IT, \, n \in \Z \}$. So the preceding paragraph implies that $\II_{\overline{T}} = \{ (x, (0,n), x) : x \in X, \, n \in \Z \}$. The isomorphism $\GT \times \Z \to \GG_{\overline{T}}$ intertwines $c_h$ with the $2$-cocycle $\sigma \in Z^2(\GG_{\overline{T}},\T)$ given by $\sigma((x, (m,n), y), (y, (p, q), z)) \coloneqq \widetilde{h}(x,m,y)^q$. The restriction of this $\sigma$ to $\II_{\overline{T}}^{(2)}$ satisfies
\[
\sigma\big((x, (0,m), x), (x, (0,n), x)\big) = \widetilde{h}(x,0,x)^n = 1.
\]
Hence $\sigma$ is $\omega$-constant on $\II_{\overline{T}}$ with $\omega = 1$. It follows that $Z_\omega = P_{\overline{T}} = \{0\} \times \Z \cong \Z$.

The $\widehat{\Z}$-valued $1$-cocycle $\twistchar^\sigma$ obtained from \cref{lemma: twistchar_gamma^sigma}\cref{item: twistchar_.^sigma 1-cocycle} satisfies
\begin{align*}
\twistchar_{(x, (m,n), y)}^\sigma(p) &= \sigma\big((x, (m,n), y), (y, (0,p), y)\big) \\
&\qquad \cdot \sigma\big((x, (m,n+p), y), (y, (-m,-n), x)\big) \\
&\qquad \cdot \overline{\sigma\big((x, (m,n), y), (y, (-m,-n), x)\big)} \\
&= \widetilde{h}(x,m,y)^p \, \widetilde{h}(x, m, y)^{-n} \, \overline{\widetilde{h}(x, m, y)^{-n}} \\
&= \widetilde{h}(x, m, y)^p.
\end{align*}
So the isomorphism $\chi \mapsto \chi(1)$ from $\widehat{\Z}$ to $\T$ carries $\twistchar_{(x,(m,n),y)}^\sigma$ to $\widetilde{h}(x,m,y) \in \T$.

We have
\[
\HH_{\overline{T}} \coloneqq \GG_{\overline{T}}/\II_{\overline{T}} \cong (\GT \times \Z)/(\GTo \times \Z) \cong \GT
\]
and the isomorphism is the map $[(x, (m,n), y)] \mapsto (x,m,y)$. So the spectral action $\theta$ of $\HH_{\overline{T}}$ on $X \times \widehat{Z}_\omega$ of \cref{prop: spectral action} is identified with the action of $\GT$ on $X \times \T$ given by $\theta_{(x,m,y)}(y, z) \coloneqq (x, \widetilde{h}(x,m,y) z)$, which is precisely the action $\rho$. So \cref{thm: simplicity characterisation} shows that $C^*(\GG_{\overline{T}}, c_h)$ is simple if and only if $\rho$ is minimal.
\end{proof}

In the following result, we write $t$ and $o$ for the terminus (range) and origin (source) map in a topological graph, so as to avoid confusion with the range and source maps $r$ and $s$ in the associated groupoid. We write $X(E)$ for the graph correspondence associated to a topological graph $E = (E^0, E^1, t, o)$, and we write $\OO_{X(E)}$ for the associated Cuntz--Pimsner algebra. We write $(j_{C_0(E^0)}, \, j_{X(E)})$ for the universal Cuntz--Pimsner-covariant representation of $X(E)$ that generates $\OO_{X(E)}$. See \cite{Katsura2004TAMS, LPS2014, Yeend2006CM} for background on topological graphs and their C*-algebras.

\begin{cor} \label{cor: topgraph exact}
Let $E = (E^0, E^1, t, o)$ be a second-countable topological graph such that the terminus map $t\colon E^1 \to E^0$ is proper and surjective, and the infinite-path space $E^\infty$ is uncountable. Suppose that $\ell\colon E^1 \to \T$ is a continuous function. There is an action $\beta^\ell\colon \Z \curvearrowright C^*(E)$ such that $\beta^\ell_n(j_{X(E)}(\xi)) = j_{X(E)}(\ell^n \cdot \xi)$ for all $\xi \in C_c(E^1)$. Extend $\ell$ to a continuous function $\ell\colon E^* \to \T$ by defining $\ell(e_1 \dotsb e_n) \coloneqq \prod^n_{i=1} \ell(e_i)$ and $\ell\restr{E^0} \equiv 1$, and let $T\colon E^\infty \to E^\infty$ be the shift map $T(x_1 x_2 x_3 \dotsb) = x_2 x_3 \dotsb$. Then $C^*(E) \rtimes_{\beta^\ell} \Z$ is simple if and only if for every infinite path $x \in E^\infty$, the set
\begin{equation} \label{eqn: topgraph orbit}
\big\{ \big(\lambda T^n(x), \, \ell(\lambda) \overline{\ell(x(0,n))}\big) : n \in \N, \, \lambda \in E^* t(T^n(x)) \big\}
\end{equation}
is dense in $E^\infty \times \T$.
\end{cor}

\begin{proof}
The map $\xi \mapsto \ell \cdot \xi$ is a unitary operator $U_\ell$ on the graph correspondence $X(E)$. If $\xi \in C_c(E^1)$ is a positive-valued function such that $o$ is injective on $\supp(\xi)$, then a quick calculation shows that conjugation by $U_\ell$ fixes the rank-$1$ operator $\Theta_{\xi, \xi}$. Using this, it is routine to see that if $(\psi,\pi)$ is a covariant Toeplitz representation of $E$ as in \cite[Definitions 2.2~and~2.10]{LPS2014}, then so is $(\psi \circ U_\ell, \pi)$. So the universal property of $C^*(E) \cong \OO_{X(E)}$ described by \cite[Theorem~2.13]{LPS2014} yields a unique automorphism $\beta^\ell$ that fixes $j_{C_0(E^0)}(C_0(E^0))$ and satisfies $\beta^\ell(j_{X(E)}(\xi)) = j_{X(E)}(\ell \cdot \xi)$ for $\xi \in C_c(E^1)$. The formula $\beta^\ell_n \coloneqq (\beta^\ell)^n$ then gives the desired action.\footnote{We could also appeal to the fourth paragraph of \cite[Page~462]{LN2004}.}

Since $t\colon E^1 \to E^0$ is proper, \cite[Propositions~3.11~and~3.16]{AB2018} show that $E^\infty$ is a locally compact Hausdorff space and $T$ is a local homeomorphism. By \cite[Theorem~5.2]{Yeend2006CM}, there is an isomorphism $\phi\colon C^*(E) \to C^*(\GT)$ such that
\[
\phi(j_{C_0(E^0)}(f))(x,m,y) = \delta_{x,y} \, \delta_{m,0} \, f(t(x)) \quad \text{ for } f \in C_0(E^0)
\]
and
\[
\phi(j_{X(E)}(\xi))(x,m,y) = \delta_{T(x), y} \, \delta_{m,1} \, \xi(x_1)\quad \text{ for } \xi \in C_c(E^1).
\]

Define $h\colon E^\infty \to \T$ by $h(x) \coloneqq \ell(x_1)$. Then $h$ is continuous. Let $\alpha^h \in \Aut(C^*(\GT))$ be the automorphism $\alpha^h_1$ of \cref{cor: the CP}\cref{item: action exists}. A routine calculation shows that $\alpha^h \circ \phi$ agrees with $\phi \circ \beta^\ell$ on $j_{C_0(E^0)}(C_0(E^0)) \medcup j_{X(E)}(C_c(E^1))$, and hence the uniqueness of the automorphism $\beta^\ell$ discussed in the first paragraph shows that $\alpha^h \circ \phi = \phi \circ \beta^\ell$. It therefore suffices to show that $C^*(\GT) \rtimes_{\alpha^h} \Z$ is simple if and only if the set described in \cref{eqn: topgraph orbit} is dense for each $x \in E^\infty$.

Fix $x \in E^\infty$. Let $\widetilde{h}\colon \GT \to \T$ be the $1$-cocycle of \cref{cor: the CP}\cref{item: continuous 1-cocycle}. We claim that the set described in \cref{eqn: topgraph orbit} is precisely the orbit of $(x,1)$ under the action $\rho$ of \cref{thm: CP simple}. We have
\[
(\GT)_x = \left\{ (\lambda T^n(x), \lav \lambda \rav - n, x) \,:\, n \in \N, \, \lambda \in E^* t(T^n(x)) \right\},
\]
and so for each $\gamma \in (\GT)_x$, we have
\begin{equation} \label{eqn: rho_gamma orbit}
\rho_\gamma(x,1) = \big(\lambda T^n(x), \, \widetilde{h} (\lambda T^n(x), \lav \lambda \rav - n, x)\big),
\end{equation}
for some $n \in \N$ and $\lambda \in E^* t(T^n(x))$. Direct calculation shows that
\begin{equation} \label{eqn: tildeh vs ell}
\widetilde{h}(\mu x, \lav \mu \rav - \lav \nu \rav, \nu x) = \ell(\mu)\overline{\ell(\nu)},
\end{equation}
for all $x \in E^\infty$ and $\mu,\nu \in E^* t(x)$. Together, \cref{eqn: rho_gamma orbit,eqn: tildeh vs ell} imply that the set described in \cref{eqn: topgraph orbit} is the orbit of $(x,1)$ under $\rho$. Since $\rho$ commutes with the action of $\T$ on $E^\infty \times \T$ by translation in the second coordinate, the orbit of $(x, 1)$ is dense if and only if the orbit of $(x, z)$ is dense for every $z \in \T$. That is, the set described in \cref{eqn: topgraph orbit} is dense for each $x \in E^\infty$ if and only if every $\rho$-orbit is dense. So the result follows from \cref{thm: CP simple}.
\end{proof}

To conclude, for the class of topological graphs appearing in \cref{cor: topgraph exact}, we give a sufficient condition phrased purely in terms of the graph without reference to the shift map on its infinite-path space, for simplicity of the crossed product described there. (The hypothesis that $E^\infty$ is uncountable is quite weak, and follows from a number of elementary conditions on the graph: for example, that $E^0$ is uncountable, or that $E$ has at least one vertex that supports at least two distinct cycles.)

\begin{cor} \label{cor: topgraph checkable}
Let $E = (E^0, E^1, t, o)$ be a second-countable topological graph such that the terminus map $t\colon E^1 \to E^0$ is proper and surjective, and the infinite-path space $E^\infty$ is uncountable. Let $\ell\colon E^1 \to \T$ be a continuous function. Extend $\ell$ to $E^*$ by defining $\ell(e_1 \dotsb e_n) \coloneqq \prod^n_{i=1} \ell(e_i)$ and $\ell\restr{E^0} \equiv 1$. For each $v \in E^0$, define
\[
\ForwardOrbit{v} \coloneqq \bigcup_{\mu \in E^*v} (t(\mu), \ell(\mu)) \subseteq E^0 \times \T.
\]
If $\ForwardOrbit{v}$ is dense in $E^0 \times \T$ for each $v \in E^0$, then the crossed product $C^*(E) \rtimes_{\beta^\ell} \Z$ of \cref{cor: topgraph exact} is simple.
\end{cor}

\begin{proof}
Suppose that $\ForwardOrbit{v}$ is dense in $E^0 \times \T$ for each $v \in E^0$. We aim to invoke \cref{thm: CP simple}. Fix $(x, w), (y, z) \in E^\infty \times \T$. Recall from \cite[Proposition~3.11~and~Lemma~3.13]{AB2018} that for $n \in \N$ and an open neighbourhood $U \subseteq E^n$ of $y(0,n)$ such that $o\restr{U}$ is injective, the set $Z(U) = \{ y' \in E^\infty : y'(0,n) \in U \}$ is a basic open neighbourhood of $y$. Let $d$ be the metric on $\T$ induced by the usual metric on $\R$ via the local homeomorphism $t \mapsto e^{it}$ from $\R$ to $\T$. Let $\rho$ be the action of $\GT$ on $E^\infty \times \T$ from \cref{thm: CP simple}. It suffices to fix a neighbourhood $U$ as above and an $\varepsilon > 0$ and show that there exists $\gamma \in (\GT)_x$ such that $\rho_\gamma(x,w) \in Z(U) \times B_d(z; \varepsilon)$. Let $\mu_y \coloneqq y(0,n) \in U$. Since $\ell$ is continuous, by shrinking $U$ if necessary, we may assume that
\begin{equation} \label[condition]{cond: U small}
d\big(\ell(\mu), \ell(\mu_y)\big) < \frac{\varepsilon}{2} \quad \text{ for all } \mu \in U.
\end{equation}
Since $o\colon E^1 \to E^0$ is a local homeomorphism, it is an open map, and so $o(U)$ is open. Since $\ForwardOrbit{t(x)}$ is dense in $E^0 \times \T$, we can find $\lambda \in E^* t(x)$ such that
\[
(t(\lambda), \ell(\lambda)) \in o(U) \times B_d\big(z\overline{\ell(\mu_y)}\overline{w}; \, \tfrac{\varepsilon}{2}\big).
\]
Let $\mu_{t(\lambda)}$ be the unique element of $U$ such that $o(\mu_{t(\lambda)}) = t(\lambda)$. Since $d$ is rotation-invariant, \cref{cond: U small} implies that
\begin{equation} \label{eqn: distance between mu terms}
d\big(\ell(\mu_{t(\lambda)} \lambda) w, \ell(\mu_y) \ell(\lambda) w\big) = d\big(\ell(\mu_{t(\lambda)}) \ell(\lambda) w, \ell(\mu_y) \ell(\lambda) w\big) = d\big(\ell(\mu_{t(\lambda)}), \ell(\mu_y)\big) < \frac{\varepsilon}{2}.
\end{equation}
Moreover, since $d$ is rotation-invariant and $\ell(\lambda) \in B_d\big(z\overline{\ell(\mu_y)}\overline{w}; \, \frac{\varepsilon}{2}\big)$, we have
\begin{equation} \label{eqn: distance to z}
d\big(\ell(\mu_y) \ell(\lambda) w, z\big) = d\big(\ell(\lambda), z \overline{\ell(\mu_y)} \overline{w}\big) < \frac{\varepsilon}{2}.
\end{equation}
Together, \cref{eqn: distance between mu terms,eqn: distance to z} imply that
\begin{equation} \label{eqn: distance to z less than epsilon}
d\big(\ell(\mu_{t(\lambda)} \lambda) w, z\big) \le d\big(\ell(\mu_{t(\lambda)} \lambda) w, \ell(\mu_y) \ell(\lambda) w\big) + d\big(\ell(\mu_y) \ell(\lambda) w, z\big) < \frac{\varepsilon}{2} + \frac{\varepsilon}{2} = \varepsilon.
\end{equation}
Now \cref{eqn: tildeh vs ell,eqn: distance to z less than epsilon} imply that
\[
\rho_{(\mu_{t(\lambda)} \lambda x, n + \lav \lambda \rav, x)}(x,w) = \big(\mu_{t(\lambda)} \lambda x, \, \ell(\mu_{t(\lambda)} \lambda) w\big) \in Z(U) \times B_d(z; \varepsilon),
\]
as required.
\end{proof}

\appendix

\section{Realising twisted group \texorpdfstring{C*-algebras}{C*-algebras} as induced algebras}
\label{sec_appendix: twisted group C*s are induced algebras}

In this appendix we describe how to realise twisted group C*-algebras as induced algebras, which is a key step in the proof of \cref{thm: twisted isotropy induced algebra}. These results are fairly well known and a detailed treatment is given in \cite[Theorem~4.3.1]{Armstrong2019}, so we give relatively little detail here. We assume knowledge of $C(X)$-algebras (or, more generally, $C_0(X)$-algebras). See \cite[Section~C.1]{Williams2007} for the definition and relevant results.

We first recall the definition of the induced algebra of a dynamical system. (See \cite[Section~6.3]{RW1998} for more details.)

\begin{definition} \label{def: induced algebra}
Let $G$ be a compact Hausdorff group acting continuously on the right of a locally compact Hausdorff space $X$, and let $\alpha$ be a strongly continuous action of $G$ on a C*-algebra $D$. The \hl{induced algebra} of the dynamical system $(D,G,\alpha)$ is defined by
\[
\Ind_G^X(D,\alpha) \coloneqq \{ f \in C_0(X,D) : f(x \cdot g) = \alpha_g^{-1}(f(x)) \text{ for all } x \in X \text{ and } g \in G \}.
\]
\end{definition}

\begin{thm} \label{thm: twisted group C* induced algebra}
Let $A$ be a countable discrete abelian group. Suppose that $\omega \in Z^2(A,\T)$ is a bicharacter that vanishes on $Z_\omega$, in the sense that $\omega(Z_\omega,A) \,\cup\, \omega(A,Z_\omega) = \{1\}$. Define $B \coloneqq A/Z_\omega$, and let $\tilde{\omega} \in Z^2(B,\T)$ be the bicharacter satisfying $\tilde{\omega}(p+Z_\omega,q+Z_\omega) = \omega(p,q)$ for all $p, q \in A$. Let $\{ u_p : p \in A \}$ be the canonical family of generating unitaries for the twisted group C*-algebra $C^*(A,\omega)$, and let $\{ U_{p+Z_\omega} :\, p+Z_\omega \in B \}$ be the canonical family of generating unitaries for the twisted group C*-algebra $C^*(B,\tilde{\omega})$.
\begin{enumerate}[label=(\alph*)]
\item \label{item: action of B hat on A hat for induced algebra} There is a continuous, free, proper right action of $\widehat{B}$ on $\widehat{A}$ given by
\[
(\phi \cdot \chi)(p) \coloneqq \phi(p) \, \chi(p+Z_\omega) \ \text{ for all } \phi \in \widehat{A}, \, \chi \in \widehat{B}, \, \text{and } p \in A.
\]
The orbit space $\widehat{A}/\widehat{B}$ is compact.
\item \label{item: action of B hat on C* for induced algebra} There is a strongly continuous action $\beta^B$ of $\widehat{B}$ on $C^*(B,\tilde{\omega})$ such that
\[
\beta_\chi^B(U_{p+Z_\omega}) = \chi(p+Z_\omega) \, U_{p+Z_\omega} \ \text{ for all } \chi \in \widehat{B} \text{ and } p \in A.
\]
\item \label{item: induced algebra isomorphism} There is an isomorphism $\Omega\colon C^*(A,\omega) \to \Ind_{\widehat{B}}^{\widehat{A}}\!\big(C^*(B,\tilde{\omega}),\,\beta^B\big)$ such that
\[
\Omega(u_p)(\phi) = \overline{\phi(p)} \, U_{p+Z_\omega} \ \text{ for all } p \in A \text{ and } \phi \in \widehat{A}.
\]
In particular, $C^*(A,\omega)$ is a $C(\widehat{A}/\widehat{B})$-algebra.
\end{enumerate}
\end{thm}

A detailed proof of \cref{thm: twisted group C* induced algebra} can be found in \cite[Theorem~4.3.1]{Armstrong2019}. Parts \cref{item: action of B hat on A hat for induced algebra,item: action of B hat on C* for induced algebra} are routine, but we reproduce some of the details of part~\cref{item: induced algebra isomorphism} below. For this, we need the following preliminary result.

\begin{lemma} \label{lemma: action details}
Let $A$ be a countable discrete abelian group, and let $\omega \in Z^2(A,\T)$ be a bicharacter.
\begin{enumerate}[label=(\alph*)]
\item \label{item: strongly cts action} There is a strongly continuous action $\beta^A$ of $\widehat{A}$ on $C^*(A,\omega)$ such that $\beta_\phi^A(u_p) = \phi(p) u_p$ for all $\phi \in \widehat{A}$ and $p \in A$.
\item \label{item: cond exp action} There is a faithful conditional expectation $\Phi^A\colon C^*(A,\omega) \to \C 1_{C^*(A,\omega)}$ such that
\[
\Phi^A(x) = \int_{\widehat{A}} \, \beta_\phi^A(x) \, \d\phi \quad \text{ for all } x \in C^*(A,\omega).
\]
\item \label{item: injective homo of twisted group C*} Suppose that $Y$ is a nonzero unital C*-algebra and $\Psi\colon C^*(A,\omega) \to Y$ is a unital homomorphism. If $x \in C^*(A,\omega)$ satisfies $\Psi\big(\beta_\phi^A(x)\big) = 0$ for all $\phi \in \widehat{A}$, then $x = 0$.
\end{enumerate}
\end{lemma}

\begin{proof}
Parts~\cref{item: strongly cts action,item: cond exp action} follow from standard arguments (see \cite[Lemmas 4.3.2 and 4.3.4]{Armstrong2019}).

For part~\cref{item: injective homo of twisted group C*}, fix $x \in C^*(A,\omega)$ such that $\Psi\big(\beta_\phi^A(x)\big) = 0$ for all $\phi \in \widehat{A}$. Then
\[
\Psi\big(\beta_\phi^A(x^*x)\big) = \Psi\big(\beta_\phi^A(x)\big)^* \, \Psi\big(\beta_\phi^A(x)\big) = 0,
\]
and \cite[Lemma~C.3]{RW1998} implies that
\[
\Psi\big(\Phi^A(x^*x)\big) \,=\, \Psi\Big( \int_{\widehat{A}} \beta_\phi^A(x^*x) \, \d\phi \Big) \,=\, \int_{\widehat{A}} \Psi\big(\beta_\phi^A(x^*x)\big) \, \d\phi \,=\, 0.
\]
Since $\Psi$ is unital, it is injective on $\C 1_{C^*(A, \omega)}$, and so we deduce that $\Phi^A(x^*x) = 0$. Hence $x = 0$, because $\Phi^A$ is faithful.
\end{proof}

\begin{proof}[Proof of \cref{thm: twisted group C* induced algebra}\cref{item: induced algebra isomorphism}]
Let $\YY_{A,\omega} \coloneqq \Ind_{\widehat{B}}^{\widehat{A}}\!\big(C^*(B,\tilde{\omega}),\,\beta^B\big)$. We aim to use the universal property of $C^*(A,\omega)$ to find a homomorphism $\Omega\colon C^*(A,\omega) \to \YY_{A,\omega}$ such that $\Omega(u_p)(\phi) = \overline{\phi(p)} \, U_{p+Z_\omega}$ for all $p \in A$ and $\phi \in \widehat{A}$. For each $p \in A$, define $v_p\colon \widehat{A} \to C^*(B,\tilde{\omega})$ by $v_p(\phi) \coloneqq \overline{\phi(p)} \, U_{p+Z_\omega}$. A routine argument shows that each $v_p$ is continuous and that $v_p(\phi \cdot \chi) = \big(\beta_\chi^B\big)^{-1}\big(v_p(\phi)\big)$ for all $\phi \in \widehat{A}$ and $\chi \in \widehat{B}$, and hence $v_p \in \YY_{A,\omega}$. It is clear that each $v_p$ is a unitary. For all $p, q \in A$, we have
\[
U_{p+Z_\omega} \, U_{q+Z_\omega} \,=\, \tilde{\omega}(p+Z_\omega,q+Z_\omega) \, U_{p+q+Z_\omega} \,=\, \omega(p,q) \, U_{p+q+Z_\omega},
\]
and hence for all $\phi \in \widehat{A}$, we have
\[
(v_p v_q)(\phi) \,=\, \overline{\phi(p)} \, U_{p+Z_\omega} \, \overline{\phi(q)} \, U_{q+Z_\omega} \,=\, \omega(p,q) \, \overline{\phi(p+q)} \, U_{p+q+Z_\omega} \,=\, \omega(p,q) \, v_{p+q}(\phi).
\]
Therefore, $v_p v_q = \omega(p,q) \, v_{p+q}$, and so the universal property of $C^*(A,\omega)$ implies that there is a homomorphism $\Omega\colon C^*(A,\omega) \to \YY_{A,\omega}$ such that $\Omega(u_p) = v_p$ for each $p \in A$.

We first show that $\Omega$ is surjective. Let $Z\big(\YY_{A,\omega}\big)$ denote the centre of $\YY_{A,\omega}$. By \cite[Proposition~3.49]{Williams2007}, the unital C*-algebra $\YY_{A,\omega}$ is a $C(\widehat{A}/\widehat{B})$-algebra with respect to the nondegenerate homomorphism $\Phi_{\YY_{A,\omega}}\colon C(\widehat{A}/\widehat{B}) \to Z\big(\YY_{A,\omega}\big)$ given by
\[
\Phi_{\YY_{A, \omega}}(f)(\phi) = f(\phi \cdot \widehat{B}) \, 1_{C^*(B, \tilde{\omega})}.
\]
For each $\phi \in \widehat{A}$, the set
\[
I_{\phi \cdot \widehat{B}} \,\coloneqq\, \clspan\big\{ \Phi_{\YY_{A,\omega}}(f) \, g \,:\, f \in C(\widehat{A}/\widehat{B}),\, g \in \YY_{A,\omega},\, f(\phi \cdot \widehat{B}) = 0 \big\}
\]
is an ideal of $\YY_{A,\omega}$. Define
\[
\AA \,\coloneqq\, \bigsqcup_{\phi \cdot \widehat{B} \in \widehat{A}/\widehat{B}} \YY_{A,\omega} / I_{\phi \cdot \widehat{B}}\,,
\]
and let $\rho\colon \AA \to \widehat{A}/\widehat{B}$ be the surjective map given by $\rho(g + I_{\phi \cdot \widehat{B}}) \coloneqq \phi \cdot \widehat{B}$. An application of \cite[Proposition~C.10(a)~and~Theorem~C.25]{Williams2007} shows that there is a unique topology on $\AA$ such that $\big(\AA, \rho, \widehat{A}/\widehat{B}\big)$ is an upper semicontinuous C*-bundle, and that for each $g \in \YY_{A,\omega}$, the section $\phi \cdot \widehat{B} \mapsto g + I_{\phi \cdot \widehat{B}}$ is continuous. Define
\[
\Gamma(\AA) \coloneqq \big\{ h\colon \widehat{A}/\widehat{B} \to \AA \,:\, h \text{ is continuous, and } \rho\big(h(\phi \cdot \widehat{B})\big) = \phi \cdot \widehat{B} \text{ for all } \phi \in \widehat{A} \big\},
\]
and let $F\colon \YY_{A,\omega} \to \Gamma(\AA)$ be the map given by $F(g)(\phi \cdot \widehat{B}) \coloneqq g + I_{\phi \cdot \widehat{B}}$, for all $g \in \YY_{A,\omega}$ and $\phi \in \widehat{A}$. By \cite[Theorem~C.26]{Williams2007}, $F$ is a $C(\widehat{A}/\widehat{B})$-linear isomorphism of $\YY_{A,\omega}$ onto the $C(\widehat{A}/\widehat{B})$-algebra $\Gamma(\AA)$. An application of \cite[Proposition~C.24]{Williams2007} shows that $F\!\left(\Omega\big(C^*(A,\omega)\big)\right)$ is a dense subspace of $\Gamma(\AA)$. Since $F$ is an isomorphism, it follows that $\Omega\big(C^*(A,\omega)\big)$ is dense in $\YY_{A,\omega}$, and hence $\Omega$ is surjective.

To see that $\Omega$ is injective, let $\beta^A$ be the strongly continuous action of \cref{lemma: action details}\cref{item: strongly cts action}. For each $\phi \in \widehat{A}$, the map $v_p\mapsto \phi(p)v_p$ extends to an automorphism $\alpha_\phi^A$ of $\YY_{A,\omega}$ satisfying $\alpha_\phi^A \circ \Omega = \Omega \circ \beta_\phi^A$. If $x \in C^*(A,\omega)$ satisfies $\Omega(x) = 0$, then for all $\phi \in \widehat{A}$, we have $\Omega\big(\beta_\phi^A(x)\big) = \alpha_\phi^A\big(\Omega(x)\big) = 0$, and then \cref{lemma: action details}\cref{item: injective homo of twisted group C*} gives $x = 0$. Hence $\Omega$ is injective. Thus $\Omega$ is an isomorphism, and $C^*(A,\omega)$ is a $C(\widehat{A}/\widehat{B})$-algebra with respect to the homomorphism $\Omega^{-1} \circ \Phi_{\YY_{A,\omega}}$.
\end{proof}

\newpage
\bibliographystyle{amsplain}
\makeatletter\renewcommand\@biblabel[1]{[#1]}\makeatother
\bibliography{ABS1_references}
\vspace{4ex}

\end{document}

%% file: ABS1_preamble.tex
\pdfoutput=1

\usepackage[margin=2.5cm, headsep=0.75cm, footskip=0.75cm]{geometry}
\usepackage{amsmath, amsthm, amssymb, mathtools}
\usepackage{setspace}
\usepackage[english]{isodate}
\usepackage{enumitem}
\usepackage{xcolor}
\usepackage{tikz, tikzsymbols, tikz-cd}
\usepackage[colorlinks=true, linkcolor=blue, linkbordercolor=blue, citecolor=red, citebordercolor=red, urlcolor=indigo, linktocpage=true]{hyperref}
\usepackage[capitalise, nameinlink, noabbrev, nosort]{cleveref}
\usepackage{doi}

\definecolor{indigo}{HTML}{492DA5}

\setenumerate{listparindent=\parindent}

\providecommand{\noopsort}[1]{}

\makeatletter\g@addto@macro\bfseries{\boldmath}\makeatother

\makeatletter
\let\origsection\section
\renewcommand\section{\@ifstar{\starsection}{\nostarsection}}
\newcommand\sectionspace{\vspace{0.5ex}}
\newcommand\nostarsection[1]{\sectionspace\origsection{#1}\sectionspace}
\newcommand\starsection[1]{\sectionspace\origsection*{#1}\sectionspace}
\makeatother

\setlist[enumerate]{font=\normalfont}
\crefname{enumi}{}{}

\newcommand\numberthis{\addtocounter{equation}{1}\tag{\theequation}}

\numberwithin{equation}{section}
\crefname{equation}{Equation}{Equations}

\crefname{condition}{Condition}{Conditions}
\creflabelformat{condition}{#2(#1)#3}

\crefname{claim}{Claim}{Claims}
\creflabelformat{claim}{#2(#1)#3}

\newtheorem{theorem}{Theorem}[section]

\newtheorem{thm}[theorem]{Theorem}
\crefname{thm}{Theorem}{Theorems}

\newtheorem{lemma}[theorem]{Lemma}
\crefname{lemma}{Lemma}{Lemmas}

\newtheorem{prop}[theorem]{Proposition}
\crefname{prop}{Proposition}{Propositions}

\newtheorem{cor}[theorem]{Corollary}
\crefname{cor}{Corollary}{Corollaries}

\theoremstyle{definition}

\newtheorem{definition}[theorem]{Definition}
\crefname{definition}{Definition}{Definitions}

\theoremstyle{remark}

\newtheorem{remark}[theorem]{Remark}
\crefname{remark}{Remark}{Remarks}

\newcommand{\hl}[1]{\textcolor{magenta}{\emph{#1}}}
\newcommand{\C}{\mathbb{C}}
\newcommand{\N}{\mathbb{N}}
\newcommand{\R}{\mathbb{R}}
\newcommand{\T}{\mathbb{T}}
\newcommand{\Z}{\mathbb{Z}}
\renewcommand{\AA}{\mathcal{A}}
\newcommand{\FF}{\mathcal{F}}
\newcommand{\GG}{\mathcal{G}}
\newcommand{\GGo}{\GG^{(0)}}
\newcommand{\GGc}{\GG^{(2)}}
\newcommand{\HH}{\mathcal{H}}
\newcommand{\II}{\mathcal{I}}
\newcommand{\LL}{\mathcal{L}}
\newcommand{\MM}{\mathcal{M}}
\newcommand{\OO}{\mathcal{O}}
\newcommand{\XX}{\mathcal{X}}
\newcommand{\YY}{\mathcal{Y}}
\renewcommand{\d}{\mathrm{d}}
\newcommand{\BT}{\mathcal{B}_T}
\newcommand{\GT}{\GG_T}
\newcommand{\GTo}{\GG_T^{(0)}}
\newcommand{\GTc}{\GG_T^{(2)}}
\newcommand{\HT}{\HH_T}
\newcommand{\HTo}{\HH_T^{(0)}}
\newcommand{\HTc}{\HH_T^{(2)}}
\newcommand{\IT}{\II_T}

\newcommand{\ITc}{\II_T^{(2)}}
\newcommand{\PT}{P_T}
\newcommand{\PTHat}{\widehat{P}_T}
\newcommand{\vecspan}{\operatorname{span}}
\newcommand{\clspan}{\overline{\vecspan}}
\newcommand{\supp}{\operatorname{supp}}
\newcommand{\osupp}{\operatorname{osupp}}
\newcommand{\ev}{\operatorname{ev}}

\newcommand{\Aut}{\operatorname{Aut}}
\newcommand{\Ext}{\operatorname{Ext}}
\newcommand{\ForwardOrbit}[1]{\operatorname{Orb}^+(#1)}
\newcommand{\Ind}{\operatorname{Ind}}
\newcommand{\Iso}{\operatorname{Iso}}
\newcommand{\Orb}[2][{}]{[#2]_{#1}}

\newcommand{\Stabess}[1]{\operatorname{Stab}^{\operatorname{ess}}(#1)}
\newcommand{\lav}{\ensuremath{\lvert}}
\newcommand{\rav}{\ensuremath{\rvert}}
\newcommand{\lv}{\ensuremath{\lVert}}
\newcommand{\rv}{\ensuremath{\rVert}}
\newcommand{\medcap}{\mathbin{\scalebox{1.2}{\ensuremath{\cap}}}}
\newcommand{\medcup}{\mathbin{\scalebox{1.2}{\ensuremath{\cup}}}}

\newcommand{\restr}[1]{\ensuremath{\vert_{#1}}}
\newcommand{\bigrestr}[1]{\ensuremath{\big\vert_{#1}}}
\newcommand{\trace}{\ensuremath{\operatorname{Tr}_T^\omega}}
\newcommand{\twistchar}{\ensuremath{\tau}}